\documentclass[reqno, hidelinks, 11pt]{amsart}
\usepackage{various, geometry, xcolor}
\numberwithin{equation}{section}
\newtheorem{mainTheorem}{Theorem}

\usepackage{tikz}

\begin{document}
\title{\vspace{-1cm}The divisor function along sums of two biquadrates}
\author{Wing Hong Leung}
\address{Department of Mathematics, Rutgers University, Piscataway, NJ 08854, U.S.A.}
\email{joseph.leung@rutgers.edu}
\author{Mayank Pandey\vspace{-1cm}}
\address{Department of Mathematics, Princeton University, Princeton, NJ 08540, U.S.A.}
\email{mayankpandey9973@gmail.com}
\maketitle
\begin{abstract}
  We establish power saving asymptotics for the sum of the divisor function along
  a binary quartic form, improving on work of Daniel.
  The proof involves an application of a recent two dimensional delta method due to
  Li, Rydin-Myerson, and Vishe and an exploitation of $\GL_2 $ automorphic forms arising
  from the factorization of varying cubic Dedekind zeta functions.
\end{abstract}
\renewcommand{\baselinestretch}{0.68}\normalsize
\tableofcontents
\renewcommand{\baselinestretch}{1.0}\normalsize

\newpage
\section{Introduction and statement of results}\label{sec:intro}

Let
\begin{equation*}
  d_k(n) = \sum_{d_1\dots d_k = n } 1
\end{equation*}
be the $k$-fold divisor function, and write $d(n)= d_2(n)$ (which we'll refer
to as the divisor function).
It is often of interest to estimate sums of the form
\begin{equation}\label{eq:crb36qtzpt}
  \sum_{ n < X} a(n)d_k(n)
\end{equation}
for sequences $a : \mathbb{N} \to \mathbb{C}$ of arithmetic interest.
Beyond their own intrinsic interest, questions of estimating \eqref{eq:crb36qtzpt} are relevant
to the study of sums over primes
\begin{equation}
  \sum_{p < X } a(p),
\end{equation}
which is most clearly seen through an identity of Linnik
(see \cite[Proposition 13.2]{IK}) that 
\begin{equation}
  \sum_{p^\nu<X}a(p^\nu)=-\sum_{k}\frac{(-1)^k}{k}\sum_{n< X}a(n)d_k'(n),
\end{equation}
where $d_k'(n)$ is the number of ways to write $n$ as the product of $k$ integers greater than
$1$. Note that $d_k'$ is closely related to the usual $d_k$ for we have
\begin{equation}
  d_k'(n) = \sum_{0\le \ell\le k} (-1)^{k - \ell} \binom{ k}{\ell}d_\ell(n),
\end{equation}
with $d_0(n)=\charf{n=1}$.

In this paper, we obtain power saving asymptotics for sums of the divisor function
along values of a binary quartic form, showing the following result.

\begin{mainTheorem}\label{theorem:cq5oju0wf8}
  There exists $\delta > 0$ such that for $N\ge 1$, we have that
  \begin{equation}\label{eq:cq5o3i91id}
    \sum_{m,n \in \mathbb{Z} \\ 0 < m^4 + n^4\le N } d(m^4 + n^4)
    = \kappa N^{\frac{1}{2}}(c_{-1}\log N + 2(c_0 - c_{-1})) + O(N^{\frac{1}{2} - \delta}),
  \end{equation}
  where
  \begin{equation}
    \kappa = \int_{\mathbb{R}^2} \mathbbm{1}_{ x_1^4 + x_2^4\le  1 } \,d x_1 \,d x_2,
  \end{equation}
  and $c_{-1}, c_0$ are such that
  \begin{equation}
    \sum_{q\ge 1} \frac{\rho(q)}{q^{s + 1}}
    = \frac{c_{-1}}{s - 1} + c_0 + O(|s - 1|)
  \end{equation}
  for $\Re (s) > 1$, where $\rho(q) = \#\set{x_1, x_2\in \mathbb{Z}/q \mathbb{Z} :q | x_1^4 + x_2^4}$.
\end{mainTheorem}

This improves upon work of Daniel \cite{MR1670278}, who established asymptotics for \eqref{eq:cq5o3i91id} with a 
saving in the remainder of $(\log N)^{-1 + o(1)}$. Consequently, the secondary main term was
obscured in the prior work, which could not handle the case of two divisors of equal
size whose contribution to  \eqref{eq:cq5o3i91id} is $\asymp N$.
See \S\ref{sec:cuhe3z5rxo} for further discussion of Daniel's argument, a sketch of our argument, and
a discussion of the relevant bottlenecks.

We do not make $\delta$ explicit (though it is, in principle, computable), since the resulting exponent is almost certainly very poor and does not reflect any substantive barrier.
In Theorem \ref{theorem:cq5oju0wf8}, the precise value would mainly reflect the sharp cutoff and the loss from sieving out nonsplit primes in our passage to algebraic integers.
An explicit exponent is given in Theorem \ref{theorem:cq5o3jikeq}, which may be more useful in applications.

We also expect that by combining our results (see Theorem \ref{proposition:crb1e3s2w0}) with techniques of Hooley, one can show that for any irreducible binary quadratic
form $Q$ which does not represent $x^4 + y^4$ and for any $\mathcal{R} \subset \mathbb{R}^2\setminus \set{0}$ with $\partial \mathcal{R} $ of finite
length, we have
\begin{equation}\label{eq:cuhhvzyczs}
  \#\set{(x, y)\in X \mathcal{R} : \gcd(x,y ) = 1, \, \exists u,v\text{ s.t. } Q(u, v) = x^4 + y^4}
  \asymp_{\mathcal{R}} \frac{X^2}{\sqrt{\log X}}.
\end{equation}
This answers (with a slightly different quadratic form) a question posed by Loughran in
\cite{SkorobogatovBrightRamdoraiSofosVishe2024NewDirections}. The upper bound follows from an upper bound sieve in an argument due to Serre,
but the question of a lower bound, as is common with sieving, appears harder.
The analogous question (in somewhat more generality) with $x^4 + y^4$ replaced with a binary
cubic form is answered by Sofos in \cite{MR3534973}.

In future work, we intend to generalize Theorem \ref{theorem:cq5oju0wf8} to general binary quartic forms and to replace the divisor function with the Fourier coefficients of general $\GL_2 $ automorphic
forms, thereby obtaining bounds such as \eqref{eq:cuhhvzyczs} in greater generality than the
example above.

We now state the main theorem, which takes place over $K = \mathbb{Q}(\zeta)$ with
$\zeta$ a primitive $8$th root of unity, and refer the reader to \S\ref{sec:notation} for notation not
yet defined.

\begin{mainTheorem}\label{theorem:cq5o3jikeq}
  Let $M, \Omega\ge 1$ and suppose that $\Phi^\infty \in C_c^\infty((K_\infty\setminus K_\infty^0)^2)$ and for $\beta_1', \beta_2'\in \mathcal{O}_K/M \mathcal{O}_K$,
  write
  \begin{equation}
    \Phi^f(\beta_1, \beta_2) = \mathbbm{1}_{ \substack{\beta_1 \equiv \beta_1' \, (M)\\ \beta_2 \equiv \beta_2'\, (M)} }.
  \end{equation}
  Suppose furthermore that  $\Phi^\infty$ satisfies the following properties:
  \begin{enumerate}
  \item For all places $v | \infty$ of $K$ and $(x_1^\infty, x_2^\infty)\in\mathrm{supp}(\Phi^\infty)$, we have that 
    \begin{equation}\label{eq:crb7dyjvca}
      \frac{1}{\Omega}\ll |x_1^\infty|_v, |x_2^\infty|_v\ll \Omega.
    \end{equation}
  \item For all differential operators $P\in \mathcal{D}(K_\infty^2)$ of order $k$, we have that
    \begin{equation}
      \|  P\Phi^\infty \|_{\infty} \ll_P \Omega^k.
    \end{equation}
  \end{enumerate}
  Let $\ensuremath{\boldsymbol\ell} : \mathcal{O}_K\to \mathbb{Z}^2$ denote the linear map
  \begin{equation}
    \ensuremath{\boldsymbol\ell}(n_0 + n_1\zeta + n_2\zeta^2 + n_3\zeta^3) = (n_3, n_2).
  \end{equation}
  Then, for any $X_1, X_2\ge 1$, we have that 
  \begin{multline}\label{eq:crb7sil3ax}
    \sum_{\alpha_1, \alpha_2\in \mathcal{O}_K \\ \ensuremath{\boldsymbol\ell}(\alpha_1\alpha_2) = 0 } \Phi^\infty \biggl( \frac{\alpha_1}{X_1}, \frac{\alpha_2}{X_2} \biggr)\Phi^f(\alpha_1, \alpha_2)
    = X_1^2X_2^2\sigma_\infty \prod_p\sigma_p \\ 
    + O\biggl((\Omega M)^{O(1)}X_1^2X_2^2
    \min\biggl( (X_1X_2)^{ - \eta + o(1)}, \frac{X_1}{X_2} + \frac{1}{X_1},
    \frac{X_2}{X_1} + \frac{1}{X_2} \biggr) \biggr) ,
  \end{multline}
  where
  \begin{align}
    \sigma_\infty = \sigma_\infty(\Phi^\infty, \ensuremath{\boldsymbol\ell}) &:= \int_{K_\infty^2} \Phi^\infty(x_1^\infty, x_2^\infty)
                                                \delta(\ensuremath{\boldsymbol\ell}(x_1^\infty x_2^\infty)) \,d x_1^\infty \,d x_2^\infty, \\
    \sigma_p = \sigma_p(\Phi^f, \ensuremath{\boldsymbol\ell}) &:= \int_{\mathcal{O}_{K, p}^2} \Phi^f(x_1^p, x_2^p)
                                                \delta(\ensuremath{\boldsymbol\ell}(x_1^px_2^p)) \,d x_1^p \,d x_2^p
  \end{align}
  and
  \begin{equation}
    \eta = \frac{1}{236}.
  \end{equation}
\end{mainTheorem}

\subsection{Sketch of proof and discussion of argument}\label{sec:cuhe3z5rxo}

We end this section with a discussion of Daniel's methods in \cite{MR1670278}, the bottlenecks
and limitations in his method, along with a sketch of our proofs of Theorem \ref{theorem:cq5oju0wf8} and
Theorem \ref{theorem:cq5o3jikeq}. As is typical with such sketches, we will avoid discussing the
technicalities of analytic
number theory such as smoothing and non-coprimality; we will only discuss
the top dyadic ranges which are typically the main difficulty.

Taking the sum in Theorem \ref{theorem:cq5oju0wf8}, opening up the divisor function, and splitting the
divisors into dyadic intervals, the problem of estimating the left-hand side of
\eqref{eq:cq5o3i91id} reduces to estimating 
\begin{equation}\label{eq:cuhe337fzw}
  \sum_{d\sim D} \sum_{m,n \sim N^{1/4} \\ d | m^4 + n^4 } 1
\end{equation}
for $D\ll N^{1/2}$.

Daniel~\cite{MR1670278}, using lattice point methods of Greaves \cite{MR271026, MR1150469},
was able to handle divisors of size $D \ll N^{1/2}\log^{-A}N$. 
The idea is that the inner sum in \eqref{eq:cuhe337fzw} ought to approximately be
\begin{equation}
  \frac{\rho(d)}{d^2}N^{1/2}\approx \frac{N^{1/2}}{D},
\end{equation}
which Daniel shows on average for $d\sim D$, bounding by $O(N^{1/4}\sqrt{D} (\log N)^{O(1)})$
the size of 
\begin{equation}\label{eq:cuhe35mno2}
  \sum_{d\sim D} \biggl| \sum_{n_1, n_2\sim N^{1/4} \\ d | n_1^4 + n_2^4 } 1 -
  \frac{\rho(d)}{d^2}N^{1/2}
  \biggr|.
\end{equation}

Daniel obtains such a bound by splitting into arithmetic progressions modulo $d$,
interpreting the count over $n_1, n_2$ as an average of lattice point counts, and
nontrivially estimating the count on average by using the homogeneity of $x_1^4 + x_2^4$
to show that the lattices on average do not have a shortest vector of size smaller than
one would expect given the covolume.

We remark that Friedlander--Iwaniec in \cite[Theorem 22.20]{FI}
obtain the same level of distribution of $1/2$ as Daniel with Poisson summation, in
which case the homogeneity plays a related role, this time providing more cancellation
in the resulting exponential sums than one would expect for an inhomogeneous polynomial.

The large moduli are treated similarly, but at this point, there is a strict limit
to the saving one can obtain. Focusing on the case of $D\asymp N^{1/2}$, the inner sum in
\eqref{eq:cuhe337fzw} has $1$ term on average and so one cannot hope for a bound on \eqref{eq:cuhe35mno2}
better than $N^{1/2}$, and indeed, this is the bound Daniel obtains. Summing over the
$\log\log N$ scales $N^{1/2}\log^{-A}N\ll D\ll N^{1/2}$, this gives Daniel an asymptotic for \eqref{eq:cq5o3i91id}
with remainder term $O(N^{1/2}\log\log N)$.

To obtain anything more than a saving of $(\log N)^{1 - o(1)}$ over the main term, it is
necessary then to obtain estimates for the sum \eqref{eq:cuhe35mno2}
without the absolute values. This is the barrier we refer to as the limit of the
hyperbola method: the point at which the sequence is too sparse to allow for
further gains without exploiting cancellation in more than one divisor.

One direct way to seek to breach this barrier, following either Daniel's or
Friedlander--Iwaniec's treatment, is to seek equidistribution of roots of the
congruence $x_1^4 + x_2^4 \equiv 0(q)$ as $q$ varies (beyond the equidistribution already implied by the
homogeneity). Such equidistribution of roots of a polynomial congruence to a varying
modulus is a notoriously difficult problem. This is only known for irreducible single
variable quadratic polynomials.

Incidentally, the case of single variable quadratic polynomials is
the only case of the hyperbola method barrier having been breached, which was first
done by Hooley \cite{MR153648} who obtained power saving asymptotics for sums of $d(n^2 + a)$.
Every aspect of the single variable case turns out to be intimately tied to $\GL_2$
automorphic forms (see \cite{MR750670, MR3096570}, for example), but no such interpretation appears
to exist for any other natural case of either polynomial root equidistribution or
for the divisor sum problem (though see \cite{MR4467125}).

Our work represents the first case of this hyperbola method barrier being breached for
polynomial sequences since Hooley \cite{MR153648}. We shall now sketch our argument.
As is typical with such sketches, the symbol ``='' should not be
interpreted too literally for it shall hide many nongeneric ranges and phenomena we deal
with carefully in the full proof.

We proceed with the estimation of the sum
\begin{equation}
  \sum_{d\sim N^{1/2}} \sum_{m, n\sim N^{1/4} \\ d | m^4 + n^4 } 1
  = \sum_{a_1a_2 = m^4 + n^4 \\ a_1, a_2\sim X^{2} \\ m, n\sim X} 1
\end{equation}
``as a whole'', treating $a_1, a_2$ equally (we've relabelled $X\asymp N^{1/4}$ for future compatibility
with Theorem \ref{theorem:cq5o3jikeq}). Passing to the number field $K=\mathbb{Q}(\zeta)$, we observe
that $N_{K/\mathbb{Q}}(m + n\zeta) = m^4 + n^4$, and therefore, by unique factorization\footnote{Finiteness of the class
  group would suffice in practice}, there is (modulo minor technicalities on split primes\footnote{In the case
  at hand, this holds as stated so long as $(m, n) = 1$}) a bijection between factorizations into rational integers $m^4 + n^4=a_1a_2$ and
factorizations $(m + n\zeta)=\mathfrak{a}_1 \mathfrak{a}_2$ into integral ideals $\mathfrak{a}_1, \mathfrak{a}_2\subset \mathcal{O}_K$, with
$\mathfrak{a}_i\mapsto N\mathfrak{a}_i = a_i$ giving one direction of the correspondence.

Writing $\mathfrak{a}_i = (\alpha_i)$ for some $\alpha_i\in \mathcal{O}_K$ by picking some fundamental domain for $\mathcal{O}_K/\mathcal{O}_K^\times$
and summing over the $\alpha_i$, detecting if $\alpha_1\alpha_2$ is of the form $m + n\zeta$, we obtain that
\begin{equation}
  \sum_{ a_1a_2 = m^4 + n^4 \\ a_1, a_2\sim X^2 \\ m, n\sim X } 1
  = \sum_{ \alpha_1, \alpha_2\sim X^{1/2}} \mathbbm{1}_{ \ensuremath{\boldsymbol\ell}(\alpha_1\alpha_2) = 0 }.
\end{equation}
The estimation of the right-hand side of the above is the content of Theorem \ref{theorem:cq5o3jikeq}, whose
proof we now sketch.

We detect the condition $\ensuremath{\boldsymbol\ell}(\alpha_1\alpha_2) = 0$ with a perturbation of a two dimensional
delta method due to \cite{2024arXiv2411.11355}, the statement of which amounts to
\begin{equation}
  \mathbbm{1}_{ \ensuremath{\boldsymbol\ell}(\alpha_1\alpha_2) = 0 }
  \approx -\frac{L^2}{X^{2/3}} \sum_{q\sim X^{2/3}L } \mathbbm{1}_{ q | \ensuremath{\boldsymbol\ell}(\alpha_1\alpha_2) }
  + \frac{L^2}{X^{2/3}}\sum_{d\ll X^{1/3}/L \\ \mathbf{c} \in \mathbb{Z}^2 \\ |\mathbf{c}|d \sim X^{1/3}/L \\ (c_1, c_2) = 1 }
  \frac{d L^{1/2}}{X^{2/3}}\sum_{ q\sim X^{2/3}/(d L^{1/2}) }
  \mathbbm{1}_{ \substack{dq | \det(\mathbf{c}, \ensuremath{\boldsymbol\ell}(\alpha_1\alpha_2)) \\ d | \ensuremath{\boldsymbol\ell}(\alpha_1\alpha_2) }},
\end{equation}
where $L$ is a small power of $X$ to be chosen (note here that $|\ensuremath{\boldsymbol\ell}(\alpha_1\alpha_2)| \ll X$).
Note that the conductor of the conditions on the first sum is $\ll X^{2/3}L$ and
on the second $\ll X^{2/3}/L^{1/2}$; the purpose of the introduction of the parameter $L$ is to
shift difficulty from the second sum onto the first sum for reasons that will soon be
clear.

After applying Poisson summation in $\alpha_1, \alpha_2$, we are reduced to bounding, in the generic
case (so apart from main terms coming from zero frequencies)
\begin{equation}
  \Sigma_1 = \frac{1}{L^{3}} \sum_{q\sim X^{2/3}L }\sum_{\alpha_1, \alpha_2\sim X^{1/6}L} S_1(\alpha_1,\alpha_2; q),
\end{equation}
where
\begin{equation*}
  S_1(\alpha_1,\alpha_2; q)
  = \frac{1}{q^3}\sum_{\beta_1, \beta_2\in \mathcal{O}_K/(q) \\ \ensuremath{\boldsymbol\ell}(\beta_1\beta_2) \equiv 0 \, (q) }
  e_q \biggl( \Tr \frac{\alpha_1\beta_1 + \alpha_2\beta_2}{\delta_K} \biggr)
\end{equation*}
turns out to be multiplicative in $q$ and satisfy
\begin{equation}\label{eq:4}
  S_1(\alpha_1,\alpha_2;p) = -1 + \#\set{x\, (p) : n_0 + \dots + n_3x^3 \equiv 0\,(p)} + O \biggl(\frac{1}{p}\biggr)
\end{equation}
with the notation $\alpha_1\alpha_2 = n_0 + \dots + n_3\zeta^3$ from now on, and 
\begin{multline}
  \Sigma_2 
  = X^{2/3}L^{5/2} \frac{L^2}{X^{2/3}}\sum_{|\mathbf{c}|d \sim X^{1/3}/L \\ (c_1, c_2) = 1 } \frac{d L^{1/2}}{X^{2/3}}
  \sum_{q\sim X^{2/3}/(d L^{1/2}) } \\ 
  \sum_{\alpha_1, \alpha_2\sim X^{1/6}/L^{1/2} } S_1(\alpha_1, \alpha_2; d)\mathbbm{1}_{ q | n_0c_1^3 - n_1 c_1^2c_2 + n_2c_1c_2^2 - n_3c_2^3}.
\end{multline}
At this point, we wish to bound $\max(|\Sigma_1|,|\Sigma_2|)$, and the trivial bounds (bounding each term
by $X^{o(1)}$ on average) on $\Sigma_1, \Sigma_2$ are $X^2L^6, X^2L^{-3/2}$, respectively, so when $L = 1$,
we are right at the boundary.

We are able to go beyond the trivial bound for $\Sigma_1$ by noting that because of \eqref{eq:4}, we
have that
\begin{equation}\label{eq:cuhgf3r9nr}
  \sum_{q } \frac{S_1(\alpha_1, \alpha_2; q)}{q^s}\approx \frac{L_{k_\alpha}(s)}{\zeta(s)} = L(s, \pi_\alpha),
\end{equation}
where $k_\alpha$ is the splitting field over $\mathbb{Q}$ of $f_\alpha(x) = n_0 + \dots + n_3x^3$, and $\pi_\alpha$ is a 
$\GL_2 $ cusp form of level $n_3^2\mathrm{Disc}(n_0 +\dots + n_3\zeta^3)\ll X^{2}L^4$. Letting $(\lambda_\alpha(n))_{n\ge 1} $ be the Fourier coefficients, we are at this point reduced to bounding
\begin{equation}
  \frac{1}{L^3}\sum_{\alpha\sim X^{1/3}L^2 } \biggl| \sum_{q\sim X^{2/3}L } \lambda_\alpha(q)\biggr|.
\end{equation}
At this point, we will use very little about $\lambda_\alpha $ to proceed. 
Using H\"older's inequality and the multiplicativity of $\lambda_\alpha $, at the cost of introducing
coefficients, we can lengthen the sum over $q $ and reduce ourselves to obtaining
a power saving over the trivial bound for 
\begin{equation}
  \sum_{\alpha\sim X^{1/3}L^2 } \biggl( \sum_{q\sim X^{2/3}L }\lambda_\alpha(q) \biggr)^8 \approx \sum_{\alpha\sim X^{1/6}L } \sum_{q\sim X^{16/3}L^8 }
  d_8(q)\lambda_\alpha(q).
\end{equation}
Now, applying Cauchy to smooth out the $d_8(q) $, we are reduced to getting a power saving
over the trivial bound for
\begin{equation}
  \sum_{\alpha, \alpha'\sim X^{1/3}L^2 }\sum_{q\sim X^{16/3}L^7 }\lambda_\alpha(q)\lambda_{\alpha'}(q).
\end{equation}
Note that at this point, the inner sum is essentially over 
the coefficients of the Rankin-Selberg convolution, whose L-function $L(s, \pi_\alpha\times\pi_{\alpha'}) $ has
conductor $\ll (X^2L^4)^4 = X^8L^{16} $. It is known that $\zeta_{k_\alpha, p} \neq \zeta_{k_{\alpha'}, p} $ for infinitely many primes
$p $ whenever $k_\alpha \neq k_{\alpha'} $, and so it follows that $L(s, \pi_\alpha\times\pi_{\alpha'}) $ has no pole at $s = 1 $. A convexity bound (equivalently, one application of Voronoi summation followed by a
trivial bound) yields that we may
save $ X^{-4/3}L^{O(1)} $ over the trivial bound on the inner sum whenever $k_\alpha \neq k_{\alpha'} $.
A crude count (ranging over the constant term) yields that $k_\alpha = k_{\alpha'} $ a $O(X^{-1/3}L^{-2}) $-fraction
of the time, and therefore, $\Sigma_1\ll X^{2 - \delta}L^{O(1)}$ for some $\delta > 0 $. 

Putting this together with the bound $|\Sigma_2|\ll X^{2}L^{-3/2}$ and taking $L$ a sufficiently small
power of $X$ depending on $\delta$, we obtain that $|\Sigma_1| + |\Sigma_2| \ll X^{2 - \eta}$ for some $\eta > 0$, as
desired.
\\

We end this section with some remarks on what more our methods should yield, in principle.
With some work handling the class group, it should be possible to prove Theorem \ref{theorem:cq5oju0wf8}
with $m^4 + n^4$ replaced by any irreducible binary quartic form.

Furthermore, note that at no point in our proof did we truly use the convolution
structure of the divisor function: all we do is Poisson in both variables simultaneously.
Therefore, we expect that Theorem \ref{theorem:cq5oju0wf8} should also hold with
$d(n)$ replaced by the Fourier coefficients of any $\GL_2$ automorphic form (perhaps assuming
Ramanujan), with the substitute for unique factorization functoriality of the base
change lift to the splitting field of the binary quartic. The presence of a main term
should then depend on the cuspidality of the automorphic form after lifting.

\section{Notation}\label{sec:notation}

Fix $\delta_K = 4\zeta^3$ as a generator of the different ideal, and for $\alpha, \beta\in K$ set
\begin{equation}
  \langle \alpha, \beta \rangle := \Tr \biggl( \frac{\alpha\beta}{\delta_K} \biggr),
\end{equation}
the trace pairing. Note that $\alpha = x_0 + x_1\zeta + x_2\zeta^2 + x_3\zeta^3\in K$ satisfies
$\langle \alpha, 1 \rangle = x_3$. We write $\psi(\alpha) = e(\langle \alpha, 1 \rangle)$ and let $\psi_\gamma(\alpha) = \psi(\alpha/\gamma)$.

Let $K_\infty = K \otimes_\mathbb{Q} \mathbb{R}$, $K_p = K\otimes_\mathbb{Q} \mathbb{Q}_p$, and $\mathcal{O}_{K, p} = \mathcal{O}_K\otimes_\mathbb{Z} \mathbb{Z}_p$.
The measure we take on $K_\infty$ and $K_p$ will be the Haar measure, normalized so that
$\vol(\mathcal{O}_{K,p}) = 1$. The norm can be extended as a function from $K_\infty\to \mathbb{R}_{\ge0}$ and
$\langle \cdot, \cdot \rangle$ can be extended to $K_\infty, K_p$ as well.

We let $|-|_\infty$ and  $|-|_p$ be defined on $K_\infty$ and $K_p$ as
\begin{align}
  |x|_\infty := \biggl(\prod_{v | \infty}|x_v|_v \biggr)^{1/[K : \mathbb{Q}]} \quad \text{ and } \quad |x|_p := \biggl( \prod_{v | p} |x_v|_v \biggr)^{1/[K: \mathbb{Q}]}.
\end{align}
On $K$, we have that $|x|_\infty = |N(x)|^{1/[K : \mathbb{Q}]}$ and similarly at $p$. 
With this, we write $K_\infty^0 = \set{x\in K_\infty : |x|_\infty = 0}$.

We also write $|x|_{\sup} = \sup_{v | \infty} |\sigma_v(x)|$. Clearly, we have $|x|_\infty \le |x|_{\sup}$.

For a vector $\mathbf{x} = (x_1, x_2)$, we write $|\mathbf{x}| = \sqrt{x_1^2 + x_2^2}$ and call a vector $\mathbf{x} = (x_1, x_2)\in \mathbb{Z}^2$
\emph{primitive} if $(x_1, x_2) = 1$. Define $\mathbf{x}^\perp := (-x_2, x_1)$.

For the rest of the paper, Greek letters such as $\alpha, \beta, \gamma$ will generally denote elements of
$\mathcal{O}_K$ unless otherwise specified.
$\delta $ in particular will generally denote a sufficiently small fixed positive quantity, and
may vary between different occurrences.

For a positive integer $q$, we employ the standard notation $\displaystyle\sum_{a\, (q)}$ and
$\displaystyle\sumCp_{a\, (q)}$ to denote sums over residue classes modulo $q$ and primitive residue classes modulo $q$
respectively. On the other hand, $\displaystyle\sum_{\alpha\, (q)}$ with greek letters refers to a sum $\displaystyle\sum_{\alpha\in \mathcal{O}_K/q\mathcal{O}_K}$.
Bold letters such as $\mathbf{c}$ denote vectors in $\mathbb{Z}^2$, with components denoted with subscripts
of the corresponding non-bold letter (e.g. $\mathbf{c} = (c_1, c_2)$).

Throughout, implied constants may depend on $\Omega^{O(1)}$; we do not mention this dependence further.
For instance, we may say that for $(x_1^\infty, x_2^\infty)$ in the
support of $\Phi^\infty$, by the condition \eqref{eq:crb7dyjvca}, we have
\begin{equation}
  |x_j^\infty|_\infty\asymp |x_j^\infty|_{\sup}.
\end{equation}
On the other hand, the statement $\exp(\Omega)\ll 1$ does not comply with our convention here.

\section{Preliminary lemmas}\label{sec:prelim_lemmas}

\subsection{Polynomial root counts in residue classes}

We shall need the following lemma on counting roots of polynomials
modulo prime powers. 

\begin{lemma}\label{lem.RoosofPoly}
  Suppose that $f(X) \in \mathbb{Z}[X] $ has degree $3 $ and $p $ is prime such that $p\nmid f(X) $.
  Then, for all $k\ge 0$ we have
  \begin{equation}
    \#\set*{x (p^k) : f(x) \equiv 0(p^k) }\ll (k + 1)^{k - 1}(p^k, \mathrm{Disc}(f))^{2/3}.
  \end{equation}
\end{lemma}

\begin{proof}
  This follows when $p^k | \mathrm{Disc}(f) $ from a result of Kamke \cite[(35)]{Kam1924} which provides
  a bound of $\le (k + 1)^{k - 1}p^{2k/3} $ for the count in question, and otherwise, follows from
  work of S\'andor \cite{MR47679}, who showed a bound of $\le n p^{v_p(\mathrm{Disc}(f))/2} $ when $v_p(\mathrm{Disc}(f)) < k $.
\end{proof}

\subsection{Factorization in a rank $2$ submodule of $\mathcal{O}_K$}\label{sec:number-fields}

In this subsection, we shall prove some basic facts about the factorization of elements
in the orthogonal complement of $\ensuremath{\boldsymbol\ell}(\mathcal{O}_K)$ in $\mathcal{O}_K$, given by $\mathbb{Z} \oplus \mathbb{Z} \zeta$.

\begin{lemma}\label{lem.Only1Prime}
  Take $p$ a rational prime, $x, y\in \mathbb{Z}$, and $\mathfrak{p}_1, \mathfrak{p}_2 | p$ prime ideals in $\mathcal{O}_K$ such that
  $\mathfrak{p}_1, \mathfrak{p}_2 | x + y\zeta $. Then, at least one of the following holds:
  \begin{enumerate}
  \item $p | x, y$.
  \item $\mathfrak{p}_1 = \mathfrak{p}_2$, $N\mathfrak{p}_1 = p $. 
  \end{enumerate}
  In particular, if $p\nmid (x, y)$, we have
  \begin{equation}
    p^m | N(x + y\zeta) \iff  \text{ there exists } \, \mathfrak{p} | p  \, \text{ such that } \, N\mathfrak{p} = p \text{ and } \mathfrak{p}^m | x + y\zeta.
  \end{equation}
\end{lemma}
\begin{proof}
  
  Suppose now that $p\nmid (x, y)$. Without loss of generality, suppose that
  $p\nmid y$, which we can harmlessly scale to be $1$.
  
  The condition that $\mathfrak{p}_1, \mathfrak{p}_2 | x + \zeta$ implies that $I = (p, x + \zeta)\subset \mathfrak{p}_1, \mathfrak{p}_2$.
  Then, we have that $\mathcal{O}_K/I$ is isomorphic to
  \begin{equation}
    (\mathcal{O}_K/p \mathcal{O}_K)/(x + \zeta)
    \cong \mathbb{Z}[\zeta]/(p, x + \zeta)
    \cong \mathbb{Z}[T]/(p, T^4 + 1, T + x)
    \cong (\mathbb{Z}/p \mathbb{Z})/(x^4 + 1).
  \end{equation}
  It follows that $N(I) = p$ and that $\mathfrak{p}_1 = \mathfrak{p}_2 = I$, from which the desired result follows.
\end{proof}

\begin{lemma}\label{lem.sumNorm}
  For $p$ a rational prime and $k\ge 1$, we have that
  \begin{equation}
    \sum_{x(p^k)} N((x + \zeta, p^k)) \ll p^k.
  \end{equation}
\end{lemma}
\begin{proof}
  This follows by noting that for $k\ge 1$, Lemma \ref{lem.Only1Prime} implies that $N((x + \zeta, p^k)) = (x^4 + 1, p^k)$, at which point
  the desired result follows from Hensel's lemma.
\end{proof}

\subsection{Poisson summation over $K$}\label{subsec:poisson}

The statement of Poisson summation we shall use is the following.

\begin{proposition}\label{proposition:cq6r4erq4l}
  Suppose that $w\in \mathcal{S}(K_\infty)$. Then, for any $\gamma\in \mathcal{O}_K\setminus\set{0}$, $g_\gamma : \mathcal{O}_K/(\gamma)\to \mathbb{C}$ and $X > 0$,
  we have that 
  \begin{equation}
    \sum_{\alpha\in \mathcal{O}_K} g_\gamma(\alpha) w \biggl(\frac{\alpha}{X}\biggr)
    = \frac{X^4}{\sqrt{N(\gamma)}}\sum_{\alpha\in \mathcal{O}_K} \hat g_\gamma(\alpha) 
    \hat w \biggl(\frac{\alpha X}{\gamma}\biggr)
  \end{equation}
  where the transforms $\hat w, \hat g_\gamma$ are defined as
  \begin{equation}
    \hat w(y) = \int_{K_\infty} w(x^\infty) e(-\langle x^\infty, y \rangle) \,d x^\infty,
  \end{equation}
  \begin{equation}
    \hat g_\gamma(\alpha) = \frac{1}{\sqrt{N(\gamma)}}\sum_{\beta\in \mathcal{O}_K/(\gamma)} g_\gamma(\beta)\psi_\gamma(\alpha\beta).
  \end{equation}
\end{proposition}

Note that this simply amounts to Poisson summation for sums over a sublattice of $\mathbb{Z}^4$
with coordinates $n_i\in \mathbb{Z} $ such that $\alpha = n_0 + \dots + n_3\zeta^3$.
In the case we are concerned with, we will only ever be considering $\gamma\in \mathbb{Z}$, so the
correspondence is even more plain.

We elect to work over $K$ rather than return to sums over $\mathbb{Z}^8$ because doing so motivates
maneuvers in our exponential sum evaluation which might have otherwise been opaque.

\subsection{Factorization and coefficients of cubic Dedekind zeta functions}\label{sec:cuhk0c94z7}

We record in this subsection some facts about cubic Dedekind zeta functions that will
be used later. The setup we will suppose throughout is that for some $n_0,\dots, n_3\in \mathbb{Z}$,
$f(x) = n_3x^3 + \dots + n_0$, and let $k = \mathbb{Q}[x]/(f(x))$ be a cubic extension in which $f$ has a
root. With $k^{\mathrm{cl}}/\mathbb{Q}$ the Galois closure of $k/\mathbb{Q}$, we'll suppose that $\Gal(k^{\mathrm{cl}}/\mathbb{Q})\cong S_3$.

Then, we have the following lemma:
\begin{lemma}\label{proposition:cuhklr5wqk}
  We have that the degree $3$ Dedekind zeta function $\zeta_k(s) $ factors as $\zeta(s)L(s, \pi)$,
  where $\pi$ is a dihedral cuspidal automorphic representations over $\mathbb{Q}$
  of level $N | n_3^2 \mathrm{Disc}(f)$ and central character $\omega_{\pi} = \chi_L$.

  In particular, 
  \begin{equation}
    L(s, \pi) = \sum_{n } \frac{\lambda_\pi(n)}{n^s}
  \end{equation}
  with $|\lambda_\pi(n)|\leq d(n)$ and satisfies
  \begin{equation}
    \lambda_\pi(p) = -1 + \#\set{x\in \mathbb{Z}/p \mathbb{Z} : f(x) \equiv 0\,(p)}
  \end{equation}
  for $p\nmid N$.
\end{lemma}

\begin{proof}
  Let $G = \Gal(k^{\mathrm{cl}}/\mathbb{Q}) \cong S_3$, $H = \Gal(k^{\mathrm{cl}}/k)$. Also, let $L$ be the unique quadratic subfield
  of $k^{\mathrm{cl}}$ (the fixed field of $A_3$). Furthermore, let $\Delta_k$ be the discriminant of $k/\mathbb{Q}$.

  Then, we have that $\zeta_k(s) = \zeta(s)L_L(s, \psi)$  for $\psi$ an order $3$ Dirichlet character of conductor
  $\mathfrak{f}$ such that $N\mathfrak{f} | \Delta_k$. The existence of $\pi$ of level dividing $\Delta_k $ with $L(s, \pi) = L_L(s, \psi) $ is
  standard.

  It remains to show that $\Delta_k | n_3^2 \mathrm{Disc}(f)$.

  To see this, note that $k = \mathbb{Q}[x]/(f_0(x))$ for $f_0(x) = n_3^2 f(x/n_3)$ monic.
  Then, we have that because $f_0$ is monic, $\mathrm{Disc}_k | \mathrm{Disc}(f_0)$, so the desired
  result follows upon noting that $ \mathrm{Disc}(f_0(x)) = n_3^2\mathrm{Disc}(f)$, from which
  the desired result follows.
\end{proof}

We record the following consequence of Rankin-Selberg too.
\begin{proposition}\label{proposition:cuhkndxg2y}
  Suppose that $f_1, f_2\in \mathbb{Z}[x]$ are of degree $3 $ with $\Gal(f_i)\cong S_3 $, and suppose furthermore that
  $k_1\not\simeq k_2$, with $k_i=\mathbb{Q}[x]/(f_i(x))$. Let $\pi_i$ be the corresponding automorphic representations
  from Lemma \ref{proposition:cuhklr5wqk}.
  Then, we have that for any $\phi\in C_c^\infty(\mathbb{R}_{> 0})$ with $|\phi|\le \mathbbm{1}_{ [1, 2] }$ and $B\ge 1$, we have that 
  \begin{equation}
    \sumSf_{(q, B) = 1} \lambda_{\pi_1}(q)\lambda_{\pi_2}(q) \phi \biggl(\frac{q}{Q}\biggr) \ll Q^{\frac{1}{2} + o(1)}(N_1 N_2)^{\frac{1}{2} + o(1)} B^{o(1)} \| \phi''' \|_\infty .
  \end{equation}
\end{proposition}

\begin{proof}
  We begin by noting that for $\Re s>1$,
  \begin{equation}
    L(s) := \sumSf_{(q, B) = 1 } \frac{\lambda_{\pi_1}(q)\lambda_{\pi_2}(q)}{q^s} = \frac{L(s, \pi_1\times\pi_2)}{ L(2s, \omega_{\pi_1}\omega_{\pi_2})}
    \prod_{p} \biggl( 1 + \frac{\mathbbm{1}_{ p | B }}{p^s} + \frac{O(1)}{p^{2s}} \biggr).
  \end{equation}
  As long as $\pi_{1, p} $ is not equal to $  \pi_{2, p} $ at all but finitely many primes, $L(s, \pi_1\times\pi_2)$
  is holomorphic. Since $k_1\not\simeq k_2$ and $\zeta_{k_i}(s) = \zeta(s)L(s, \pi_i) $, it follows
  from a result of Perlis \cite{MR447188} and Perlis and Stuart \cite{MR1348765} that for infinitely many
  primes, $\pi_{1, p} \not\simeq \pi_{2, p} $. In particular, in our situation, we may conclude that $L(s)$ has an
  analytic continuation to $\Re s > 1/2 $. 

  Therefore, by Mellin inversion, moving the contour to the $1/2 + 1/\log X$-line, we have that
  \begin{multline}
    \sumSf_{(q, B) = 1 } \lambda_{\pi_1}(q) \lambda_{\pi_2}(q) \phi \biggl(\frac{q}{Q}\biggr) 
    = \int_{(2)} L(s)Q^s \tilde\phi(s) \,d s \\ 
    \ll (B N_1N_2)^{o(1)} Q^{\frac{1}{2} + o(1)} \| \phi''' \|_\infty
    \int_{\mathbb{R}} \biggl| L \biggl( \frac{1}{2} + \frac{1}{\log X} + it , \pi_1\times\pi_2 \biggr) \biggr| \frac{\,d t }{1+|t|^3}. 
  \end{multline}
  Here we have integrated by parts three times on $\tilde{\phi}(s)=\int_0^\infty \phi(x)x^{s-1}dx$.
  The desired result follows immediately from the convexity bound
  $L(s, \pi_1\times\pi_2) \ll ((1 + |s|)^4(N_1N_2)^2)^{1/4 + o(1)} $.
\end{proof}

\section{The two dimensional $\delta$-method}\label{sec:delta-method}

In this section, we show the expansion of the two dimensional $\delta$-symbol we require. We
follow along the lines of \cite{2024arXiv2411.11355}, though we shall not require expanding into
additive characters, the lack thereof simplifies the proof.

Furthermore, a perturbation of the parameters in \cite{2024arXiv2411.11355} is required, so we
need to show the following version of their result ourselves.

\begin{theorem}\label{theorem:cq5o3o8yaf}
  Suppose that $\mathbf{n}\in \mathbb{Z}^2$ with $|n_1|, |n_2| < X$. Take $1 \ll D \ll X^{1/2}$.
  Let $\omega_1\in C_c^\infty(\mathbb{R}_{> 0})$. Let $\omega_2\in C_c^\infty(\mathbb{R})$ be even with $\omega_2(0)=0$. Further suppose that
  \begin{equation}\label{eq:cq5pgpuzx6}
    \int_{\mathbb{R}^2} \omega_1(|\mathbf{x}|) \, d \mathbf{x} = \int_{\mathbb{R}_{> 0}} \omega_2 (x) \,d x   = 1.
  \end{equation}
  Then, for any $A > 0$, we have that 
  \begin{multline}\label{multline:cq6r39jnv5}
    \charf{\mathbf{n} = 0} = \frac{1}{D^2}\sum_{d\ge 1}
    \sum_{\mathbf{c}\in \mathbb{Z}^2 \\ (c_1, c_2) = 1} \omega_1 \biggl(\frac{|\mathbf{c}|d}{D}\biggr)
    \frac{d}{\sqrt{DX}}\sum_{q\ge 1} \biggl( \omega_2 \biggl(\frac{d q}{\sqrt{DX}}\biggr)
    - \omega_2 \biggl(\frac{\det(\mathbf{c}, \mathbf{n})}{q\sqrt{DX}}\biggr) \biggr)
    \charf{\substack{\det(\mathbf{c}, \mathbf{n})\equiv 0 \, (d q) \\ \mathbf{n}\equiv 0 \, (d)}}\\
    - \frac{2}{D^2}\sum_{q\ge 1} \omega_1 \biggl(\frac{|\mathbf{n}|}{q D}\biggr)
    \charf{\mathbf{n}\equiv 0 \, (q)} + O_{A, \omega_1, \omega_2}(D^{-A}),
  \end{multline}
\end{theorem}
As an ingredient to the proof of Theorem \ref{theorem:cq5o3o8yaf}, we shall require
the following simple version of the usual $\delta$-method of Duke--Friedlander--Iwaniec~\cite{DFI93}.

\begin{proposition}\label{theorem:cq5o3smo7r}
  Suppose that $\omega\in C_c^\infty(\mathbb{R})$ is even with $\omega(0)=0$ and $\hat\omega(0) = 2$. Then, for any $A>0$, $Q\geq 1$, and $n\in \mathbb{Z}$, we have that
  \begin{equation}
    \charf{n = 0} = \frac{1}{Q}\sum_{q\ge 1} \biggl( \omega \biggl(\frac{q}{Q}\biggr) -
    \omega \biggl(\frac{n}{qQ}\biggr) \biggr)\charf{n\equiv 0 \, (q)} + O_{A, \omega}(Q^{-A}).
  \end{equation}
\end{proposition}

\begin{proof}[Proof of Theorem~\ref{theorem:cq5o3o8yaf}]
  We begin by noting for all $\mathbf{n} \neq 0$, there is an involution of
  \begin{equation}
    \set{(\mathbf{c}, d, q) : n = \mathbf{c} dq, \mathbf{c} \text{ primitive}, d, q\ge 1}
  \end{equation}
  given by $(\mathbf{c}, d, q)\mapsto (\mathbf{c}, q, d)$. Therefore, for $\mathbf{n} \neq 0$, we have 
  \begin{align}\label{multline:cq5o3ss851}
    0 = & \ \sum_{\mathbf{c}\text{ primitive} \\ \langle \mathbf{c}, \mathbf{n} \rangle > 0}
    \charf{\det(\mathbf{c}, \mathbf{n}) = 0}
    \sum_{d q\mathbf{c} = \mathbf{n}} \biggl(\omega_1 \biggl(\frac{|\mathbf{c}|d}{D}\biggr)
    - \omega_1 \biggl(\frac{|\mathbf{n}|}{q D}\biggr)\biggr)\nonumber\\
    = & \ \sum_{\mathbf{c}\text{ primitive} \\ \langle \mathbf{c}, \mathbf{n} \rangle > 0}
    \sum_{d | \mathbf{n}}\charf{\det(\mathbf{c}, \mathbf{n}/d)=0} \, 
    \omega_1 \biggl(\frac{|\mathbf{c}|d}{D}\biggr) -\sum_{q | \mathbf{n}} \omega_1 \biggl(\frac{|\mathbf{n}|}{qD}\biggr)  \\
    =& \ \frac{1}{2}\sum_{\mathbf{c} \text{ primitive} } \sum_{d | \mathbf{n} }
       \mathbbm{1}_{ \det(\mathbf{c}, \mathbf{n} / d) = 0 } \,
       \omega_1 \biggl(\frac{|\mathbf{c}|d}{D}\biggr)-\sum_{q | \mathbf{n} } \omega_1 \biggl(\frac{|\mathbf{n}|}{q D}\biggr).
  \end{align}
  At this point, we apply Theorem~\ref{theorem:cq5o3smo7r} with $Q = \sqrt{DX}/d, \omega = \omega_2$ to
  expand out $\charf{\det(\mathbf{c}, \mathbf{n}/d) = 0}$ and obtain that \eqref{multline:cq5o3ss851} equals
  \begin{multline}\label{eqwithinDeltaProp}
    0 = \frac{1}{2}\sum_{d|\mathbf{n}} \sum_{\mathbf{c}\text{ primitive}}\omega_1 \biggl(\frac{|\mathbf{c}|d}{D}\biggr)
          \frac{1}{Q}\sum_{q\ge 1} \biggl( \omega_2 \biggl(\frac{q}{Q} \biggr)
          - \omega_2 \biggl(\frac{\det(\mathbf{c}, \mathbf{n}/d)}{q Q}\biggr)\biggr)\charf{\det( \mathbf{c},\mathbf{n}/d)\equiv 0 \, (q)}\\
         -\sum_{q\ge 1} \charf{\mathbf{n}\equiv 0 \, (q)} \, \omega_1 \biggl(\frac{|\mathbf{n}|}{q D}\biggr)+ O(D^{-A})
  \end{multline}
  for any $A>0$. Here we use that $d\ll D\ll X^{1/2}$ to obtain the error term $O(D^{-A})$. This
  completes the proof for $\mathbf{n}\neq0$.

  For $\mathbf{n} = 0$, the right-hand side of \eqref{eqwithinDeltaProp} equals
  \begin{equation}
    \sum_{d\ge 1} \sum_{\mathbf{c}\text{ primitive}}\omega_1\biggl(\frac{|\mathbf{c}|d}{D}\biggr)
    \frac{1}{Q}\sum_{q\ge 1}\omega_2 \biggl(\frac{q}{Q}\biggr) +O(D^{-A})
    =  D^2 + O(D^{-A}+Q^{-A})
  \end{equation}
  by Poisson summation and the normalizations \eqref{eq:cq5pgpuzx6}.
  The desired result follows.
\end{proof}

\section{Setup for proof of Theorem \ref{theorem:cq5o3jikeq} in the balanced case} \label{sec:outline}

Let $X = X_1X_2$.  Until \S\ref{sec:cuhe1tf31a}, we shall restrict ourselves to the balanced case
\begin{equation}\label{eq:cuhe1tmv46}
  X^{-\frac{1}{20}}\ll \frac{X_1}{X_2}, \frac{X_2}{X_1}\ll X^{\frac{1}{20}},
\end{equation}
for the contents of \S\ref{sec:cuhe1tf31a} will imply the desired result otherwise with a remainder of
$ O(X^2\min(X_1/X_2, X_2/X_1)) $ (in light of the fact that
$\eta = 1/236$, less than the saving in the exponent from \S\ref{sec:cuhe1tf31a} if \eqref{eq:cuhe1tmv46} does not hold).

We begin with some reductions to simplify the notation for the remainder of the paper.
By Fourier inversion, we have that
\begin{equation}
  \Phi^\infty(x_1^\infty, x_2^\infty) = \int_{K_\infty^2} \widehat{\Phi^\infty}(y_1^\infty, y_2^\infty) \psi(-x_1^\infty y_1^\infty - x_2^\infty y_2^\infty)
  \,d y_1^\infty \,d y_2^\infty.
\end{equation}
The decay of $\widehat{\Phi^\infty}$ implied by the derivative bounds of $\Phi^\infty$ implies that
we may suppose (at an acceptable cost absorbed into the $\Omega^{O(1)}$ factor in the final
remainder term) that $\Phi^\infty(x_1^\infty, x_2^\infty) = \phi_1(x_1^\infty)\phi_2(x_2^\infty)$ for some $\phi_1, \phi_2\in C_c^\infty(K_\infty\setminus K_\infty^0)$
upon performing a dyadic partition of unity on $\alpha_1$, $\alpha_2$.

We are therefore reduced to showing that 
\begin{equation}
  \Sigma = \sum_{\alpha_1, \alpha_2\in \mathcal{O}_K \\ \ensuremath{\boldsymbol\ell}(\alpha_1\alpha_2) = 0 \\ \alpha_1 \equiv \beta_1' (M) \\ \alpha_2 \equiv \beta_2'(M) }
  \phi_1 \biggl(\frac{\alpha_1}{X_1}\biggr) \phi_2 \biggl(\frac{\alpha_2}{X_2}\biggr)
  = X^2 \sigma_\infty\prod_{p} \sigma_p + O((M\Omega)^{O(1)}X^{2 - \eta}),
\end{equation}
with
\begin{equation}
  \sigma_\infty = \int_{K_\infty^2} \phi_1( x_1^\infty)\phi_2(x_2^\infty) \delta(\ensuremath{\boldsymbol\ell}(x_1^\infty x_2^\infty))
  \, d x_1^\infty \,d x_2^\infty,
\end{equation}
\begin{equation}
  \sigma_p = \int_{\mathcal{O}_{K, p}} \mathbbm{1}_{ \substack{\beta_1 \equiv \beta_1'\, (M) \\ \beta_2 \equiv \beta_2'\, (M) } }
  \delta(\ensuremath{\boldsymbol\ell}(\beta_1\beta_2)) \, d \beta_1 \, d \beta_2.
\end{equation}

Our proof begins with an application of the 2-dimensional $\delta$-method to detect
the condition $\ensuremath{\boldsymbol\ell} = 0$, followed by an application of Poisson summation.

For the application of the $\delta$-method, we fix two even $\omega_1, \omega_2\in C_c^\infty(\mathbb{R}\setminus \set{0})$ with
\begin{equation}
  1 = \int_{\mathbb{R}^2} \omega_1(|\mathbf{x}|) \, d \mathbf{x}
  = \frac{1}{2}\int_{\mathbb{R}} \omega_2(x) \,d x .
\end{equation}
These will remain fixed throughout the argument, with no dependence on any of the other
parameters.

We now introduce a parameter $1\le L\ll X^{\frac{1}{100}}$ (ultimately, it will be taken to be
$X^{2\eta/3} $) and write
\begin{equation*}
  D := \frac{X^{\frac{1}{3}}}{L}.
\end{equation*}
Then, an application of the 2-dimensional $\delta$-method (Theorem \ref{theorem:cq5o3o8yaf}) and Poisson summation (Proposition \ref{proposition:cq6r4erq4l}) yields the
following.
\begin{proposition}\label{proposition:crb36eebaq}
  We have that
  \begin{equation}
    \Sigma = -2\Sigma_1^{\set{}} + \Sigma_2^{\set{}} - 2\Sigma_1^{\set{1,2}} + \Sigma_2^{\set{1,2}}
    + \sum_{ j=1,2 } (-2\Sigma_1^{\set{j}} + \Sigma_2^{\set{j}}),
  \end{equation}
  where for $S\subset\set{1, 2}$, we write
  \begin{equation}\label{eq:crb5jnfhh8}
    \Sigma_{1}^S = \frac{X^4}{D^2} \sum_{q\ge 1 \\ M | q } \frac{1}{q^5}
    \sum_{\alpha_1, \alpha_2 \\ \alpha_i = 0\iff i\not\in S} (S_1I_1)(\alpha_1, \alpha_2; q),
  \end{equation}
  with
  \begin{equation}\label{eq:crb7tdqkjb}
    S_1(\alpha_1, \alpha_2; q) = \frac{1}{q^3} \sum_{\beta_1, \beta_2\in \mathcal{O}_K/q \mathcal{O}_K \\ \ensuremath{\boldsymbol\ell}(\beta_1\beta_2) \equiv 0(q/(q, M)) \\ \beta_1 \equiv \beta_1' ( (q, M)) \\ \beta_2 \equiv \beta_2' ((q, M)) } \psi_q(\alpha_1\beta_1 + \alpha_2\beta_2),
  \end{equation}
  \begin{equation}\label{eq:crb7ay3tq3}
    I_1(\alpha_1, \alpha_2; q) =   \int_{K_\infty^2} \omega_1 \biggl(\frac{|\ensuremath{\boldsymbol\ell}(x_1^\infty x_2^\infty)|MX}{qD}\biggr)\phi_1(x_1^\infty) \phi_2(x_2^\infty) \psi \biggl( -\frac{X_1x_1^\infty\alpha_1}{q}
    - \frac{X_2x_2^\infty\alpha_2}{q} \biggr) \,d x_1^\infty \,d x_2^\infty,
  \end{equation}
  and
  \begin{equation}\label{Sigma2SDef}
    \Sigma_2^S = \frac{X^4}{D^2} \sum_{ d\ge 1 \\ \mathbf{c}\in \mathbb{Z}^2 \\ (c_1, c_2) = 1 }
    \frac{1}{d^5} \omega_1 \biggl(\frac{|\mathbf{c}|d}{D}\biggr)
    \frac{d}{\sqrt{DX}} \sum_{q\ge 1 \\ M | q } \frac{1}{q^5}
    \sum_{\alpha_1, \alpha_2  \\ \alpha_i = 0\iff i\not\in S} (S_2I_2)(\alpha_1, \alpha_2, \mathbf{c}; d, q),
  \end{equation}
  with
  \begin{equation}\label{S2Def}
    S_2(\alpha_1, \alpha_2, \mathbf{c}; d, q) = \frac{1}{d^3q^3}
    \sum_{ \beta_1, \beta_2\in \mathcal{O}_K/ dq \mathcal{O}_K \\
      \ensuremath{\boldsymbol\ell}(\beta_1\beta_2) \equiv 0(d) \\
      \det(\mathbf{c}, \ensuremath{\boldsymbol\ell}(\beta_1\beta_2))\equiv 0(dq/(q, M))\\
      \beta_1 \equiv \beta_1' ((dq, M)) \\ \beta_2 \equiv \beta_2' ((dq, M))}
    \psi_{dq}(\alpha_1\beta_1 + \alpha_2\beta_2),
  \end{equation}
  \begin{multline}\label{I2Def}
    I_2(\alpha_1, \alpha_2; \mathbf{c}, d, q) 
    = \int_{K_\infty^2} \biggl( \omega_2 \biggl(\frac{d q}{M\sqrt{DX}}\biggr)
    - \omega_2 \biggl(\frac{M\sqrt{X}\det(\mathbf{c}, \ensuremath{\boldsymbol\ell}(x_1^\infty x_2^\infty))}
    {q\sqrt{D}}\biggr)\biggr)\phi_1(x_1^\infty)\phi_2(x_2^\infty) \\ 
    \psi \biggl( -\frac{X_1x_1^\infty\alpha_1}{d q}- \frac{X_2x_2^\infty\alpha_2}{d q} \biggr)
    \,d x_1^\infty \,d x_2^\infty.
  \end{multline}
\end{proposition}

\subsection{Proof of Theorem \ref{theorem:cq5o3jikeq}}

In this subsection, we shall prove Theorem \ref{theorem:cq5o3jikeq} assuming bounds
we will show in \S\ref{sec:cq6r430mfy} and \S\ref{sect:Sigma2} for $\Sigma_1^S$ and
$\Sigma_2^S$, respectively. The remainder of this paper will then be devoted to showing these
bounds. We refer the reader to the next subsection, \S\ref{sec:cuhge4dbyc}, for a explicit listing
of the main propositions and what they bound, which may make the present subsection
easier to read.

Putting together Propositions \ref{proposition:cran7rqg6g} and \ref{prop.Sigma2Final} to estimate $\Sigma_1^{\set{}}, \Sigma_2^{\set{}}$, we obtain that
\begin{equation}
  -2\Sigma_1^{\set{}} + \Sigma_2^{\set{}} = X^2 \sigma_\infty\prod_p\sigma_p
  + O \biggl( (M\Omega)^{O(1)}
  X^{2 + o(1)} \biggl( \frac{L^4}{X^{\frac{1}{3}}} +
  \frac{L^2}{X^{\frac{2}{3}}}\biggr) \biggr).
\end{equation}
By Proposition \ref{proposition:cran7w05id}, we have 
\begin{equation}\label{eq:crare1j4im}
  \Sigma_{2}^{\set{1,2}} \ll (M\Omega)^{O(1)}X^{2 + o(1)} \biggl( \frac{1}{L^{\frac{3}{2}}}
  + \frac{L^{\frac{2}{7}}}{X^{\frac{2}{21}}} \biggr),
\end{equation}
and by Proposition \ref{proposition:cq7567d3xk}, we have that 
\begin{equation}\label{eq:crare1jzos}
  \Sigma_1^{\set{1,2}}\ll   (M\Omega)^{O(1)} L^{47/8}X^{2 - 1/48+ o(1)} .
\end{equation}
At this point, we optimize to match the first two terms in \eqref{eq:crare1j4im} and \eqref{eq:crare1jzos}, taking
\begin{equation}
  L = X^{\frac{2}{3}\eta}
\end{equation}
with
\begin{equation}
  \eta = \frac{1}{236} = \frac{1/48}{3/2 + 47/8}\cdot \frac{3}{2}.
\end{equation}

Putting this together with Propositions \ref{proposition:cq9d99egtf} and \ref{proposition:crare1wa9u}, which imply that
\begin{align}
  \Sigma_1^{\set{j}} &\ll \frac{X^{3 + o(1)}}{D X_2^2} \ll X^{2 - 1/10 + o(1)},\\
  \Sigma_2^{\set{j}} &\ll X^{2 + o(1)}\frac{X^{\frac{3}{2}}}{X_2^3D} \ll X^{2 - 1/20 + o(1)},
\end{align}
yields that 
\begin{equation}
  \sum_{j\le 2}(-2\Sigma_1^{\set j} + \Sigma_2^{\set j})
  - 2\Sigma_1^{\set{1,2}} + \Sigma_2^{\set{1,2}}
  \ll (M\Omega)^{O(1)}X^{2 - \eta + o(1)}.
\end{equation}
from which Theorem \ref{theorem:cq5o3jikeq} follows.

\subsection{Structure of paper}\label{sec:cuhge4dbyc}

The paper is organized as follows.
\begin{enumerate}
\item[\S\ref{sec:prelim_lemmas}] contains certain minor lemmas which will be used at various points throughout
  the proof, as well as the statement of Poisson summation we shall use. We also record
  in \S\ref{sec:cuhk0c94z7} all facts about the factorization of cubic Dedekind zeta functions and the bounds
  on the resulting $\GL_2 $ Fourier coefficients.
\item[\S\ref{sec:delta-method}] contains a statement and proof of the two-dimensional delta method we use,
  providing a perturbation of the method of \cite{2024arXiv2411.11355} for our purposes.
\item[\S\ref{sec:cq72jm8eaf}] contains estimates on $S_1$ that will be used to handle the nonzero frequencies
  $\Sigma_1^{\set{1,2}}$ as well as bounds pertaining to the zero frequency which will be used to bound
  $\Sigma_1^{\set{}} $.
\item[\S\ref{sec:cq72jm8woo}] contains upper bounds on $S_2$ germane to the estimation of the nonzero frequency
  contribution $\Sigma_2^{\set{1, 2}}$, as well as a related estimate, Proposition \ref{prop:S20}, which will play
  a significant role in the estimation of the zero frequency $\Sigma_2^{\set{}}$ via \S\ref{sec:cuhge2i6da}.
\item[\S\ref{sec:oscillatory}] contains upper bounds on both the oscillatory integrals and their derivatives
  which will be used to bound the nonzero frequency contributions $\Sigma_j^{\set{1,2}}$.
\item[\S\ref{sec:cuhge2i6da}] contains preparatory lemmas leading up to the computation of the constants
  in front of the main terms of $\Sigma_{j}^{\set{}}$ and the secondary main term of $\Sigma_1^{\set{}}$ (which leads
  to the main term of Theorem \ref{theorem:cq5o3jikeq})
\item[\S\ref{sec:cuhge2k5o6}] bounds the nonzero frequency contribution $\Sigma_1^{\set{1,2}}$ in Proposition
  \ref{proposition:cq7567d3xk}. It is in this section that we apply the argument that culminates in an application
  of the moment argument described in the proof sketch in Proposition \ref{prop:moment_rs_bound}.
  Proposition \ref{prop:moment_rs_bound} is where all of the facts on the factorization of Dedekind $\zeta $-function
  recorded \S\ref{sec:cuhk0c94z7} are used.
\item[\S\ref{sec:cuhge2lp8l}] bounds the partial zero frequency contribution $\Sigma_{1}^{\set{j}}$ in
  Proposition \ref{proposition:cq9d99egtf}.
\item[\S\ref{sec:crb0xrxfh5}] contains an asymptotic for $\Sigma_1^{\set{}}$ in Proposition \ref{proposition:cran7rqg6g}, giving a main term of
  size $X^3/D$ and a secondary main term of size $X^{2}$, the former of which cancels with
  the main term of $\Sigma_2^{\set{}}$.
\item[\S\ref{sec:cuhge2kjb1}] contains a bound on the nonzero frequency contribution $\Sigma_2^{\set{1,2}}$ in
  Proposition \ref{proposition:cran7w05id}.
\item[\S\ref{sec:cuhge2ozhj}] bounds the partial zero frequency contribution $\Sigma_2^{\set{j}}$ in
  Proposition \ref{proposition:crare1wa9u}. 
\item[\S\ref{sec:crao4uslaz}] contains an asymptotic for $\Sigma_2^{\set{}}$ in Proposition \ref{prop.Sigma2Final}, matching
  the appropriate constant in front of $X^3/D$ coming from Proposition \ref{proposition:cran7rqg6g}.
\end{enumerate}

We provide below a dependency graph containing most of the main key bounds and their
dependencies for ease of reading and tracking.

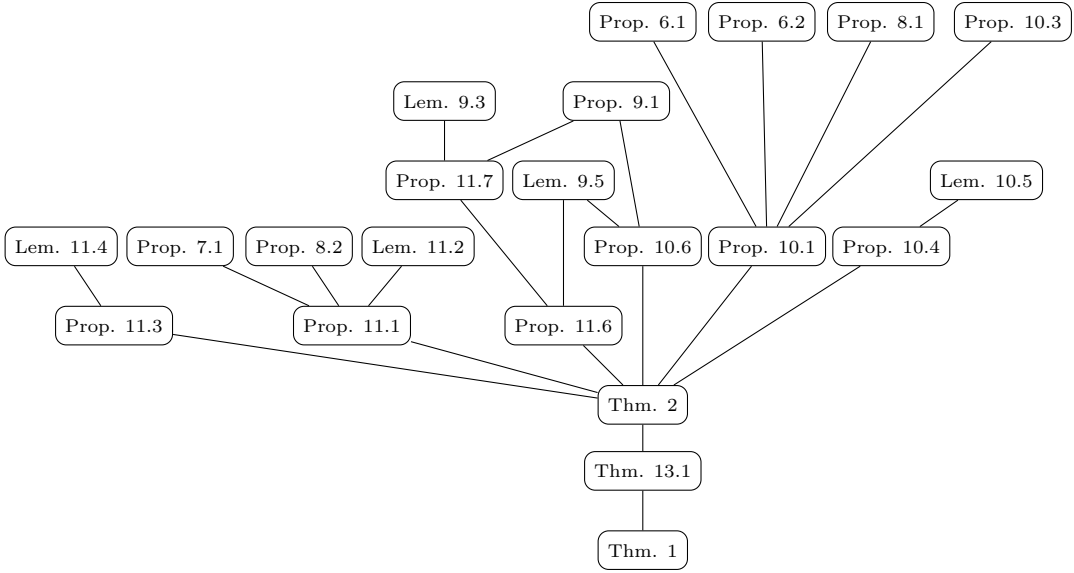
\vspace{20pt}
\begin{tikzpicture}[rounded corners, every node/.style={draw, rectangle, font=\tiny, text height=1.3ex, text depth=.25ex}]
  
  \node[fill=green!00] (cor_div_sum) at (0pt,0pt) {Thm. \ref{theorem:cq5oju0wf8}};
  \node[fill=green!00] (cor_main) at (0pt,30pt) {Thm. \ref{proposition:crb1e3s2w0}};
  \node[fill=green!00] (lem_singular) at (-30pt,140pt) {Lem. \ref{lemma:crb0xx4e28}};
  \node[fill=green!00] (propD1star) at (-10pt,170pt) {Prop. \ref{lemma:cran7rw9t5}};
  \node[fill=green!00] (propSigma2LocalFactorSum) at (-75pt,140pt) {Prop. \ref{proposition:crb7r8dzyl}};
  \node[fill=green!00] (lemSigmap) at (-75pt,170pt) {Lem. \ref{lemma:crb7r6zwci}};
  \node[fill=green!00] (propSigma1FullZero) at (0pt,115pt) {Prop. \ref{proposition:cran7rqg6g}};
  
  \node[fill=green!00] (propS1pEval) at (0pt,200pt) {Prop. \ref{proposition:cq6y0ob20l}};
  \node[fill=green!00] (propS1pkBound) at (45pt,200pt) {Prop. \ref{proposition:cq6y0pcoy5}};
  \node[fill=green!00] (propI1Bound) at (90pt,200pt) {Prop. \ref{proposition:crale1hrji}};
  \node[fill=green!00] (propMomentRS) at (140pt,200pt) {Prop. \ref{prop:moment_rs_bound}};
  
  \node[fill=green!00] (propS2Nonzero) at (-175pt,115pt) {Prop. \ref{prop.S2non0}};
  \node[fill=green!00] (propI2Bound) at (-130pt,115pt) {Prop. \ref{prop.I2}};
  
  \node[fill=green!00] (lemLargeGcd) at (130pt,140pt) {Lem. \ref{lemma:cq9d99fn96}};
  
  \node[fill=green!00] (propSigma1PartialZero) at (94pt,115pt) {Prop. \ref{proposition:cq9d99egtf}};
  \node[fill=green!00] (propSigma1Nonzero) at (47pt,115pt) {Prop. \ref{proposition:cq7567d3xk}};
  \node[fill=green!00] (propSigma2FullZeroFinal) at (-30pt,85pt) {Prop. \ref{prop.Sigma2Final}};
  \node[fill=green!00] (propSigma2Nonzero) at (-110pt,85pt) {Prop. \ref{proposition:cran7w05id}};
  \node[fill=green!00] (propSigma2PartialZero) at (-200pt,85pt) {Prop. \ref{proposition:crare1wa9u}};
  \node[fill=green!00] (lemSigma2Gcd) at (-220pt,115pt) {Lem. \ref{proposition:cravzrp2bh}};
  \node[fill=green!00] (lemLevelDist) at (-85pt,115pt) {Lem. \ref{proposition:cravzd1g45}};
  \node[fill=green!00] (thmMain) at (0pt,55pt) {Thm. \ref{theorem:cq5o3jikeq}};

  \draw (lem_singular) -- (propSigma1FullZero);
  \draw (propD1star) -- (propSigma1FullZero);
  \draw (propSigma1FullZero) -- (thmMain);
  \draw (propS1pEval) -- (propSigma1Nonzero);
  \draw (propS1pkBound) -- (propSigma1Nonzero);
  \draw (propI1Bound) -- (propSigma1Nonzero);
  \draw (propMomentRS) -- (propSigma1Nonzero);
  \draw (propSigma1Nonzero) -- (thmMain);
  \draw (lemLargeGcd) -- (propSigma1PartialZero);
  \draw (propSigma1PartialZero) -- (thmMain);
  \draw (propS2Nonzero) -- (propSigma2Nonzero);
  \draw (propI2Bound) -- (propSigma2Nonzero);
  \draw (lemLevelDist) -- (propSigma2Nonzero);
  \draw (propSigma2Nonzero) -- (thmMain);
  \draw (lemSigma2Gcd) -- (propSigma2PartialZero);
  \draw (propSigma2PartialZero) -- (thmMain);
  \draw (lemSigmap) -- (propSigma2LocalFactorSum);
  \draw (propD1star) -- (propSigma2LocalFactorSum);
  \draw (propSigma2LocalFactorSum) -- (propSigma2FullZeroFinal);
  \draw (lem_singular) -- (propSigma2FullZeroFinal);
  \draw (propSigma2FullZeroFinal) -- (thmMain);
  \draw (thmMain) -- (cor_main);
  \draw (cor_main) -- (cor_div_sum);
\end{tikzpicture}

\subsection{Proof of Proposition \ref{proposition:crb36eebaq}}\label{sec:cuhge19z1e}

In this subsection, we carry out the application of the two dimensional delta method
and the applications of Poisson summation that lead to the decomposition in Proposition
\ref{proposition:crb36eebaq}.
Recall that we are estimating
\begin{equation}
  \Sigma = \sum_{\alpha_1, \alpha_2\in \mathcal{O}_K  \\ \ensuremath{\boldsymbol\ell}(\alpha_1\alpha_2)=  0 \\
    \alpha_1 \equiv \beta_1' \, (M) \\ \alpha_2 \equiv \beta_2' \, (M)}\phi_1 \biggl(\frac{\alpha_1}{X_1}\biggr)
  \phi_2 \biggl(\frac{\alpha_2}{X_2}\biggr).
\end{equation}
Applying Theorem \ref{theorem:cq5o3o8yaf} to decompose $\mathbbm{1}[\ensuremath{\boldsymbol\ell}(\alpha_1\alpha_2) = 0]$ (with $X$, $D$ in Theorem \ref{theorem:cq5o3o8yaf} taken
to be $X=X_1X_2$, $D$ respectively, and any fixed choice of $\omega_1$, $\omega_2$ with $\omega_2$ supported compactly away from $0$), we obtain that
\begin{equation}
  \Sigma = -2\Sigma_1 + \Sigma_2,
\end{equation}
where
\begin{equation}
  \Sigma_1 = \frac{1}{D^2} \sum_{q\ge 1 }
  \sum_{\alpha_1, \alpha_2\in \mathcal{O}_K \\ \alpha_1 \equiv \beta_1' \, (M) \\ \alpha_2  \equiv \beta_2' \, (M)}
  \mathbbm{1}_{ \ensuremath{\boldsymbol\ell}(\alpha_1\alpha_2) \equiv 0\, (q)  } \, 
  \omega_1 \biggl(\frac{|\ensuremath{\boldsymbol\ell}(\alpha_1\alpha_2)|}{q D}\biggr),
\end{equation}
\begin{multline}
  \Sigma_2 = \frac{1}{D^2} \sum_{d\ge 1 \\ \mathbf{c} \in \mathbb{Z}^2 \\ (c_1, c_2) = 1 } \omega_1 \biggl(\frac{|\mathbf{c}|d}{D}\biggr)
  \frac{d}{\sqrt{DX}} \sum_{q\ge 1 } \\ 
  \sum_{\alpha_1,\alpha_2\in \mathcal{O}_K \\ \alpha_1 \equiv \beta_1' \,(M) \\ \alpha_2 \equiv \beta_2' \, (M) }
  \biggl( \omega_2 \biggl(\frac{dq}{\sqrt{DX}}\biggr)
  - \omega_2 \biggl(\frac{\det(\mathbf{c}, \ensuremath{\boldsymbol\ell}(\alpha_1\alpha_2))}{q\sqrt{DX}}\biggr)\biggr)
  \mathbbm{1}_{ \substack{\det(\mathbf{c}, \ensuremath{\boldsymbol\ell}(\alpha_1\alpha_2)) \equiv 0\, (dq) \\
      \ensuremath{\boldsymbol\ell}(\alpha_1\alpha_2) \equiv 0\, (d)} }.
\end{multline}
Then, applying Poisson summation (Proposition \ref{proposition:cq6r4erq4l}) to both the
sums over $\alpha_1$, $\alpha_2$ yields that
\begin{equation}
  \Sigma_1 = \frac{1}{D^2} \sum_{q\ge 1 \\ M|q} \frac{1}{q^5} \sum_{\alpha_1, \alpha_2\in \mathcal{O}_K }
  S_1(\alpha_1, \alpha_2; q) I_1(\alpha_1, \alpha_2; q),
\end{equation}
and
\begin{equation}
  \Sigma_2 = \frac{1}{D^2} \sum_{d\ge 1 \\ \mathbf{c} \in \mathbb{Z}^2 \\ (c_1, c_2) = 1 }\frac{1}{d^5}
  \omega_1 \biggl(\frac{|\mathbf{c}|d}{D}\biggr) \frac{d}{\sqrt{DX}}
  \sum_{q\ge 1 \\M|q} \frac{1}{q^5}
  \sum_{\alpha_1, \alpha_2\in \mathcal{O}_K } S_2(\alpha_1, \alpha_2; \mathbf{c}, d, q) I_2(\alpha_1, \alpha_2; \mathbf{c}, d, q),
\end{equation}
where the $S_j$ and $I_j$ were defined in \eqref{eq:crb7tdqkjb}-\eqref{I2Def}. Note that we have rewritten $qM$ as $q$ above. Splitting based on which of the $\alpha_j$ are $0$ yields the desired result.

\section{Estimating $S_1$}\label{sec:cq72jm8eaf}

In this section, we provide estimates for $S_1(\alpha_1,\alpha_2; q)$ in \eqref{eq:crb7tdqkjb}. It is here that we
obtain what will turn out to be cusp for Fourier coefficients we mentioned in our sketch
at \eqref{eq:cuhgf3r9nr}. This will be exploited in \S\ref{sec:cuhge2k5o6}.

The estimation of $S_1(\alpha_1, \alpha_2; q)$ reduces to the case of $q = p^k$ for $k\ge 1$, for
$S_1(\alpha_1, \alpha_2; q)$ is a multiplicative function in $q$. These cases we resolve with the following
collection of results.

For any $\alpha=n_0+n_1\zeta+n_2\zeta^2+n_3\zeta^3\in\mathcal{O}_K$, we denote
\begin{equation}
  f_\alpha(x) = n_3x^3 + n_2x^2 + n_1x + n_0
\end{equation}
and let $\Delta_\alpha = \mathrm{Disc}(f_\alpha) $. 

The first is a result we will be applying in the general case
\begin{proposition}\label{proposition:cq6y0ob20l}
  Suppose that $p$ is prime such that $p\nmid M$. Then, we have that
  \begin{equation}
    S_1( \alpha_1, \alpha_2; p) = a_{\alpha_1\alpha_2}(p) + r_{\alpha_1\alpha_2}(p)
  \end{equation}
  where
  \begin{equation}\label{eq:cq6y0x7va3}
    a_\alpha(p)
    = -1-\charf{p|\langle\alpha,1\rangle} + \#\set{x(p) : f_\alpha(x)\equiv 0(p)}
  \end{equation}
  and $r_{\alpha}(p)$ satisfies
  \begin{equation}
    r_{\alpha}(p) \ll \frac{1}{p} + \charf{p | N(\alpha)} +
    p\charf{p^2 | N(\alpha)} + p^3\charf{p  | \alpha}.
  \end{equation}
\end{proposition}
In all other cases, we can afford to show the following upper bounds.
\begin{proposition}\label{proposition:cq6y0pcoy5}
  Suppose that $p$ is prime and $k\ge 1$.
  Furthermore, suppose that $p^{s_1} || \alpha_1 $, $p^{s_2} || \alpha_2 $,  $p^{s_3} || \alpha_1\alpha_2p^{-s_1 - s_2} $. Write $s^* = s_1 + s_2 + s_3 $,
  $\alpha = \alpha_1\alpha_2 $, $\alpha' = \alpha_1\alpha_2p^{-s^*} $.
  Then, we have that
  \begin{equation}
    S_1(\alpha_1, \alpha_2; p^k)
    \ll (p^k, M)^8 k p^{3\min(s^*, k)} (p^{k}, \Delta_{\alpha'})^{2/3} \sum_{t\le k - s^* }
    p^{t} \mathbbm{1}_{ p^{2t} | N(\alpha')  }.
  \end{equation}

\end{proposition}

At the zero frequency $\alpha_1=\alpha_2=0$, we renormalize slightly, defining.
\begin{equation}\label{N1TDef}
  \tilde{N}_1(q) := \frac{1}{q^3}S(0,0;q)=\frac{1}{q^6}\sum_{\beta_1, \beta_2\in \mathcal{O}_K/q\mathcal{O}_K \\ \ensuremath{\boldsymbol\ell}(\beta_1\beta_2)\equiv 0(q) \\ \beta_i\equiv \beta_i'\, ((q,M))} 1. 
\end{equation}
We will require only the below bounds at the zero frequency.
\begin{proposition}\label{prop.N1}
  Let $p$ be a prime and $k\geq m_p$. Then we have
  \begin{equation}
    \tilde{N}_1(p^k)\ll p^{2m_p}.
  \end{equation}
  Moreover, we have
  \begin{equation}
    |\tilde{N}_1(p^{k+1})-\tilde{N}_1(p^k)|\ll p^{4m_p-2k-2}.
  \end{equation}
\end{proposition}

The proofs of Propositions \ref{proposition:cq6y0ob20l}, \ref{proposition:cq6y0pcoy5} and \ref{prop.N1} begin with the observation that by
orthogonality,
\begin{multline}
  S_1(\alpha_1, \alpha_2; q) \\ 
  = 
  \frac{1}{(q, M)^8} \sum_{\gamma_1, \gamma_2 ((q, M)) } \psi_{(q, M)}(-\gamma_1\beta_1' - \gamma_2\beta_2')
  S_1^*\biggl( \alpha_1 + \frac{q}{(q, M)}\gamma_1, \alpha_2 + \frac{q}{(q, M)}\gamma_2; q \biggr),
\end{multline}
where
\begin{equation}
  S_1^*(\alpha_1, \alpha_2; q)
  = \frac{1}{q^3}\sum_{\beta_1, \beta_2 (q) \\ \ensuremath{\boldsymbol\ell}(\beta_1\beta_2)\equiv 0(q/(q, M))} \psi_{q}(\alpha_1\beta_1 + \alpha_2\beta_2).
\end{equation}
Then, by a couple more applications of orthogonality, we have 
\begin{multline}\label{multline:crb5zuo9x1}
  S_1^*(\alpha_1, \alpha_2; q) = \frac{1}{q^5}\sum_{x, y\, (q)}\sum_{\beta_1, \beta_2 \, (q)}
  \psi_q(\alpha_1\beta_1 + \alpha_2\beta_2 - (q, M)(x + y\zeta)\beta_1\beta_2)\\
  =  \frac{1}{q}\sum_{x, y\, (q)}\sum_{\beta_1 \, (q)\\ (q, M)(x + y\zeta)\beta_1\equiv \alpha_2\, (q)} \psi_q(\alpha_1\beta_1) \\ 
  = \frac{(q, M)^2}{q} \sum_{x, y\, (q) \\ (q, M) | x, y}
  \sum_{\beta_1\, (q) \\ (x + y\zeta)\beta_1 \equiv \alpha_2\, (q)} \psi_{q}(\alpha_1\beta_1).
\end{multline}
Armed with this, we shall now show the three propositions. 
\begin{proof}[Proof of Proposition \ref{proposition:cq6y0ob20l}]
  Because $p\nmid M$, $S_1^* = S_1$, so we shall proceed starting from \eqref{multline:crb5zuo9x1}.
  Scaling $\beta_1$ by $t\in (\mathbb{Z}/p \mathbb{Z})^\times$ and $x, y$ by $t^{-1}$ and averaging over $t$, we obtain 
  \begin{align}\label{multline:cq6y00ffq8}
    S_1(\alpha_1, \alpha_2; p) =& \frac{1}{p(p - 1)}\mathop{\sum_{x,y\, (p)}\sum_{\beta_1 \, (p)}}_{(x + y\zeta)\beta_1\equiv \alpha_2 \, (p)}
    \sumCp_{t \, (p)} \psi_p(t\alpha_1\beta_1)\\
    =&\ \frac{1}{(p - 1)}\mathop{\sum_{x,y\,(p)}\sum_{\beta_1 \, (p)}}_{\substack{\langle \alpha_1, \beta_1 \rangle\equiv 0 \, (p) \\ (x + y\zeta)\beta_1 \equiv \alpha_2 \, (p)}} 1 -
    \frac{1}{p(p - 1)}\mathop{\sum_{x,y\,(p)}\sum_{\beta_1 \, (p)}}_{(x + y\zeta)\beta_1\equiv \alpha_2 \, (p)} 1.
  \end{align}

  We claim that the second sum is
  \begin{equation}\label{eq:cq6y00fq19}
    \frac{1}{p(p - 1)}\mathop{\sum_{x,y\,(p)}\sum_{\beta_1 \, (p)}}_{(x + y\zeta)\beta_1\equiv \alpha_2 \, (p)} 1
    = 1 + O \biggl(\frac{1}{p} + \charf{p | N(\alpha_2)} + p^2\charf{p | \alpha_2}\biggr).
  \end{equation}
  To see this precisely, note that the contribution of $x\equiv y\equiv 0(p)$ is given by
  \begin{equation}
    \frac{1}{p(p - 1)}\sum_{\beta_1 \, (p)} \charf{p|\alpha_2} \ll p^2\charf{p|\alpha_2}.
  \end{equation}
  On the other hand, the contribution of $x^4+y^4\equiv0 \, (p)$ with $(x, y)\not\equiv (0, 0) \, (p)$ is given by
  \begin{equation}
    \frac{1}{p(p - 1)}\mathop{\sumCp_{x,y\, (p)}\sum_{\beta_1 \, (p)}}_{\substack{(x/y)^4+1\equiv 0 \, (p)\\ (x + y\zeta)\beta_1\equiv \alpha_2 \, (p)}} 1\ll \charf{p|N(\alpha_2)}.
  \end{equation}
  Since $(x+y\zeta,p)=1 \Leftrightarrow x^4+y^4\not\equiv 0 \, (p)$, we have
  \begin{align}
    \frac{1}{p(p - 1)}\mathop{\sum_{x,y\,(p)}\sum_{\beta_1 \, (p)}}_{(x + y\zeta)\beta_1\equiv \alpha_2 \, (p)} 1
    =&\  \frac{1}{p(p - 1)}\sum_{x,y \, (p)\\ x^4+y^4\not\equiv0 \, (p)} 1 + O \biggl(\charf{p | N(\alpha_2)} + p^2\charf{p | \alpha_2}\biggr)\nonumber\\
    =&\ 1 + O \biggl(\frac{1}{p} + \charf{p | N(\alpha_2)} + p^2\charf{p | \alpha_2}\biggr).
  \end{align}

  It remains now to estimate the first sum in \eqref{multline:cq6y00ffq8}, which is given by
  \begin{equation}\label{eq:cq6y0utwor}
    \frac{1}{p - 1}\mathop{\sum_{x,y\, (p)}\sum_{\beta_1\, (p)}}_{\substack{\langle \alpha_1, \beta_1 \rangle\equiv 0\, (p) \\ (x + y\zeta)\beta_1\equiv \alpha_2\, (p)}} 1.
  \end{equation}
  The contribution of $y\equiv 0 \, (p)$ to \eqref{eq:cq6y0utwor} is
  \begin{align}\label{eq:cq6y00f62q}
    \frac{1}{p - 1}\mathop{\sum_{x\, (p)}\sum_{\beta_1\, (p)}}_{\substack{\langle \alpha_1,\beta_1\rangle\equiv0 \, (p)\\ x\beta_1\equiv \alpha_2 \, (p)}} 1
    =&\ \frac{1}{p - 1}\mathop{\sumCp_{x\, (p)}\sum_{\beta_1 \, (p)}}_{\substack{\langle \alpha_1,\beta_1\rangle\equiv0 \, (p) \\ x\beta_1\equiv \alpha_2 \, (p) }} 1
    + \frac{1}{p-1}\sum_{\beta_1(p) \\ \langle \alpha_1,\beta_1\rangle\equiv0 \, (p)} \charf{p|\alpha_2}\nonumber\\
    =&\ \frac{1}{p - 1}\sumCp_{\substack{x \, (p)\\ \langle \alpha_1\alpha_2,x\rangle\equiv0 \, (p)}} 1+O\left(p^3\charf{p|\alpha_2}\right)=\charf{p|\langle\alpha_1,\alpha_2\rangle}+O\left(p^3\charf{p|\alpha_2}\right).
  \end{align}
  Hence we have
  \begin{align}\label{multline:cq6y00g98s}
    \frac{1}{p - 1}\mathop{\sum_{x,y\, (p)}\sum_{\beta_1 \, (p)}}_{\substack{\langle \alpha_1, \beta_1 \rangle\equiv 0 \, (p) \\ (x + y\zeta)\beta_1\equiv \alpha_2 \, (p)}} 1
    =&\ \frac{1}{p - 1}\mathop{\sum_{x\, (p)}\sumCp_{y\, (p)}\sum_{\beta_1 \, (p)}}_{\substack{\langle \alpha_1, \beta_1 \rangle\equiv 0 \, (p) \\ (x + y\zeta)\beta_1\equiv \alpha_2 \, (p)}} 1
    + \charf{p|\langle\alpha_1,\alpha_2\rangle}+O\left(p^3\charf{p|\alpha_2}\right)\nonumber\\ 
    =&\ \mathop{\sum_{x\, (p)}\sum_{\beta_1 \, (p)}}_{\substack{\langle \alpha_1, \beta_1 \rangle\equiv 0 \, (p) \\ (x + \zeta)\beta_1\equiv \alpha_2 \, (p)}} 1+ \charf{p|\langle\alpha_1,\alpha_2\rangle}
    + O \biggl( p^3\charf{p | \alpha_2} \biggr),
  \end{align}
  By the symmetry $S_1(\alpha_1, \alpha_2; p^k) = S_1(\alpha_2, \alpha_1; p^k)$, we may suppose without loss of generality that $N((\alpha_2, p)) | N((\alpha_1, p))$.
  Then, the contribution of $x$ with $(x + \zeta,p)>1$ $\Leftrightarrow$ $x^4 + 1\equiv 0(p)$ is
  \begin{equation}\label{eq:cq6y00hydq}
    \mathop{\sum_{x\, (p)}\sum_{\beta_1 \, (p)}}_{\substack{x^4 + 1\equiv 0 \, (p)\\ \langle \alpha_1, \beta_1 \rangle\equiv 0 \, (p) \\ (x + \zeta)\beta_1\equiv \alpha_2 \, (p)}} 1\le \mathop{\sum_{x\, (p)}\sum_{\beta_1 \, (p)}}_{\substack{x^4 + 1\equiv 0 \, (p) \\ (x + \zeta)\beta_1\equiv \alpha_2 \, (p)}} 1
    \le 4p \charf{p | N(\alpha_2)}\le 4p\charf{p^2 | N(\alpha_1\alpha_2)}.
  \end{equation}
  As for the remaining contribution, we have that
  \begin{equation}\label{eq:cq6y0zpieo}
    \mathop{\sum_{x\, (p)}\sum_{\beta_1 \, (p)}}_{\substack{(x+\zeta,p)=1  \\ \langle \alpha_1, \beta_1 \rangle\equiv 0 \, (p) \\ (x+ \zeta)\beta_1\equiv \alpha_2 \, (p)}} 1
    = \sum_{x \, (p) \\  (x+\zeta,p)=1  \\ \langle \alpha_1\alpha_2,\overline{x+\zeta} \rangle\equiv 0 \, (p) } 1
    = \sum_{u, v, w, x \, (p) \\ (x+\zeta,p)=1 \\ (x + \zeta)(u + v\zeta + w\zeta^2)\equiv \alpha_1\alpha_2 \, (p)} 1.
  \end{equation}
  The contribution of $\alpha_1\alpha_2\equiv 0(p)$ is then at most
  \begin{equation}\label{eq:cq6y00mb7q}
    \le p\charf{p | \alpha_1\alpha_2},
  \end{equation}
  so we shall from now on focus on the case of $p\nmid \alpha_1\alpha_2$.
  In this case, we have that \eqref{eq:cq6y0zpieo} equals 
  \begin{multline}\label{multline:cq6y0z6u7k}
    \sum_{u, v, w, x \, (p) \\ (x+\zeta,p)=1 \\ 
      (x + \zeta)(u + v\zeta + w\zeta^2)\equiv \alpha_1\alpha_2 \, (p)} 1\\
    =  \sum_{u, v, w, x \, (p) \\ (x + \zeta)(u + v\zeta + w\zeta^2) \equiv \alpha_1\alpha_2 \, (p)} 1
    + O\biggl( \sum_{u, v, w, x \, (p) \\ (x + \zeta) (u + v\zeta + w\zeta^2)\equiv \alpha_1\alpha_2 \, (p)} \charf{p | N(\alpha_1\alpha_2)}\biggr).
  \end{multline}
  Note that
  \begin{equation}\label{eq:cq6y00oiaj}
    \sum_{u, v, w, x \, (p) \\ (x + \zeta)( u + v\zeta + w\zeta^2)\equiv \alpha_1\alpha_2 \, (p)} 1 =
    \#\set{x(p) : n_0 + n_1x + n_2x^2 + n_3x^3\equiv 0 \, (p)}\le 3,
  \end{equation}
  writing $\alpha_1\alpha_2 =n_0 + n_1\zeta + n_2\zeta^2 + n_3\zeta^3$, with the upper bound of $3$ following
  from the assumption that $p\nmid \alpha_1\alpha_2$.
  Therefore, it follows that \eqref{multline:cq6y0z6u7k} equals
  \begin{equation}
    \#\set{x(p) : n_0 + n_1x + n_2x^2 + n_3x^3\equiv 0 \, (p)} + O(\charf{p | N(\alpha_1\alpha_2)}).
  \end{equation}
  Combining \eqref{multline:cq6y00ffq8}, \eqref{eq:cq6y00fq19}, \eqref{multline:cq6y00g98s}, \eqref{eq:cq6y00hydq}, \eqref{eq:cq6y0zpieo}, \eqref{eq:cq6y00mb7q}, \eqref{multline:cq6y0z6u7k}, and \eqref{eq:cq6y00oiaj}, the desired result follows.
\end{proof}

\begin{proof}[Proof of Proposition \ref{proposition:cq6y0pcoy5}]

  By the symmetry $S_1(\alpha_1, \alpha_2; p^k) = S_1(\alpha_2, \alpha_1; p^k)$, so we may suppose without loss of
  generality that
  \begin{equation}\label{S1symmetry}
    v_p(N(\alpha_2))\leq v_p(N(\alpha_1)).
  \end{equation}
 Starting from \eqref{multline:crb5zuo9x1}, we put absolute value inside the $x,y$-sum to get
  \begin{equation}\label{eq:cq61srf78y}
    (p^k,M)^{-2}|S_1^*(\alpha_1,\alpha_2;p^k)|\leq \frac{1}{p^k}\sum_{x ,y\, (p^k)} \bigg|\sum_{\beta_1 \, (p^k) \\ (x +y\zeta)\beta_1 \equiv \alpha_2 \, (p^k) }
    \psi_{p^k}(\alpha_1\beta_1)\bigg|.
  \end{equation}
  As in the proof of Proposition \ref{proposition:cq6y0ob20l}, we average over dilations of $\beta_1$ by $(\mathbb{Z}/p^k \mathbb{Z})^{\times}$.
  We may also pull out the factors of $p$ from $(x,y)$, after which we obtain that \eqref{eq:cq61srf78y} is
  \begin{align}\label{eq:RHS_primepow_div}
    \leq p^{3k}\charf{p^k | \alpha_2}
    + \frac{p}{p-1}\sum_{0\le s < k \\ p^s | \alpha_2} \sum_{\eta=0,1} \frac{1}{p^{k + \eta}}
    \sum_{x, y \, (p^{k-s}) \\ (x, y, p) = 1} \sum_{\beta_1 \, (p^k) \\
    \langle \alpha_1, \beta_1 \rangle \equiv 0 \, (p^{k - \eta})\\
    (x + y\zeta)\beta_1 \equiv \alpha_2p^{-s} \, (p^{k - s}) }
    1.
  \end{align}

  We shall now bound the contribution of some fixed $\eta=0,1$ and $0\le s < k$ to \eqref{eq:RHS_primepow_div}, which is 
  \begin{multline}\label{eq:crastl6787}
    \ll \frac{1}{p^{k+\eta}} \sum_{x, y \, (p^{k - s}) \\ (x,y, p) = 1 }
    \sum_{\beta_1 \, ( p^{k}) \\ \langle \alpha_1, \beta_1 \rangle \equiv 0\, (p^{k -\eta}) \\ (x + y\zeta)\beta_1 \equiv \alpha_2p^{-s}\,(p^{k-s}) }1 \\ 
    \le p^{3s-\eta}\mathop{\sum_{x \, (p^{k - s})}
    \sum_{\beta_1 \, ( p^{k - s})}}_{\substack{\langle \alpha_1, \beta_1 \rangle \equiv 0\, (p^{k - \max\{\eta,s\}}) \\ (x + \zeta)\beta_1 \equiv \alpha_2p^{-s} \,(p^{k-s})}}1
    + p^{3s-\eta} \mathop{\sum_{y \, (p^{k - s})}
    \sum_{\beta_1 \, ( p^{k - s})}}_{\substack{\langle \alpha_1, \beta_1 \rangle \equiv 0\, (p^{k - \max\{\eta,s\}}) \\ (1 + y\zeta)\beta_1 \equiv \alpha_2p^{-s}\, (p^{k-s}) }}1,
  \end{multline}
  where the second bound has followed by separating the cases $(y, p) = 1$ and $(x, p)= 1$.
  It thus suffice to bound only the first sum on the RHS of \eqref{eq:crastl6787}, as the treatment of its counterpart is identical. 

  We shall split based on the power $p^t= (x^4 + 1, p^{k - s})$.
  For a given $t \leq k-s$, we have that the condition $(x + \zeta)\beta_1 \equiv \alpha_2p^{-s} \,(p^{k - s})$ implies
  $p^t|N(\alpha_2p^{-s})$ and determines $\beta_1$ for a given $x$ up to $p^t$ possibilities.
  Furthermore, the conditions in the sum over
  $\beta_1$ imply that
  \begin{multline}
    \biggl\langle \frac{\alpha_1\alpha_2p^{-s}}{x + \zeta}, \beta_1 \biggr\rangle \equiv 0\, (p^{k - t-\max\{\eta,s\}})
    \\
    \implies f_{\alpha_1\alpha_2p^{-s}}(x) := \langle \alpha_1\alpha_2p^{-s}, (x + \zeta^3)(x + \zeta^5)(x + \zeta^7) \rangle \equiv 0\,(p^{k -\max\{\eta,s\}}).
  \end{multline}
  Therefore, the contribution of some $0\le t\le k - s$ with $p^t|N(\alpha_2p^{-s})$ to the first sum on the RHS of \eqref{eq:crastl6787} is at most
  \begin{multline}\label{eq:crastnglb3}
    p^{3s+t-\eta} \sum_{x \, (p^{k - s}) \\ p^t|x^4+1 } \mathbbm{1}[f_{\alpha_1\alpha_2p^{-s}}(x)\equiv 0 \, (p^{k -\max\{\eta,s\}}) ] \\ 
    \leq 3p^{3s+t} p^{\min(s^*, k) - s} (p^{k - \min(s^*, k)}, \Delta_{\alpha'})^{2/3} \le 3p^{\min(s^*, k) + 2s+ t} (p^{k}, \Delta_{\alpha'})^{2/3}.
  \end{multline}
  by an application of Lemma \ref{lem.RoosofPoly}.
  Finally, $p^t|N(\alpha_1p^{-s})$ implies $p^{2t}|N(\alpha_1\alpha_2p^{-s})$ by \eqref{S1symmetry}.
  Thus, we have shown an upper bound of
  \begin{equation}\label{eq:cuhk3q3fyt}
    \ll (p^{k}, \Delta_{\alpha'})^{2/3}p^{\min(s^*, k)}\sum_{s + t\le k } p^{2s + t} \mathbbm{1}_{ \substack{p^s | \alpha_1\alpha_2 \\ p^{2t } | N(\alpha_1\alpha_2p^{-s}) }}.
  \end{equation}
  Now, consider some $s, t $ for which $p^s | \alpha_1\alpha_2 $ and $p^{2t} | N(\alpha_1\alpha_2p^{-s}) $. Then, if $p^{s + 1} | \alpha_1\alpha_2 $, we have
  that
  \begin{equation}
    p^{2s + t} \mathbbm{1}_{ \substack{p^s | \alpha_1\alpha_2 \\ p^{2t} | N(\alpha_1\alpha_2p^{-s})} }
    \le p^{2(s + 1) + (t-  2)}
    \mathbbm{1}_{ \substack{p^{s + 1} | \alpha_1\alpha_2 \\ p^{2(t - 2)} | N(\alpha_1\alpha_2p^{-s - 1})} }.
  \end{equation}
  Therefore, we have that \eqref{eq:cuhk3q3fyt} is at most
  \begin{equation}
    k(p^k, \Delta_{\alpha'})^{2/3} p^{3\min(s^*, k)} \sum_{t\le k - \min(s^*, k)}
    p^{t} \mathbbm{1}_{ p^{2t } | N(\alpha')},
  \end{equation}
  as desired. 
\end{proof}


\begin{proof}[Proof of Proposition \ref{prop.N1}]
  Write $m=m_p:=v_p(M)$ for simplicity. We first prove the second statement. Starting with \eqref{multline:crb5zuo9x1}, we have
  \begin{equation}
    \tilde{N}_1(p^{k+1})=\frac{1}{p^{4k+4+2m}}\sum_{\gamma_2\, (p^m)}\psi_{p^m}(-\beta_2'\gamma_2)\mathop{\sum_{x,y\, (p^{k+1-m})}\sum_{\beta_1 \, (p^{k+1})}}_{\substack{p^m(x+y\zeta)\beta_1\equiv p^{k+1-m}\gamma_2 \, (p^{k+1})\\ \beta_1\equiv \beta_1'\, (p^m)}} 1.
  \end{equation}
  The contribution of $p|x,y$ is given by
  \begin{multline}
    \frac{1}{p^{4k+4+2m}}\sum_{\gamma_2\, (p^m)}\psi_{p^m}(-\beta_2'\gamma_2)\mathop{\sum_{x,y\, (p^{k-m})}\sum_{\beta_1 \, (p^{k+1})}}_{\substack{p^{m+1}(x+y\zeta)\beta_1\equiv p^{k+1-m}\gamma_2 \, (p^{k+1})\\ \beta_1\equiv \beta_1'\, (p^m)}} 1 \\
    = \frac{1}{p^{4k+2m}}\sum_{\gamma_2\, (p^m)}\psi_{p^m}(-\beta_2'\gamma_2)\mathop{\sum_{x,y\, (p^{k-m})}\sum_{\beta_1 \, (p^{k})}}_{\substack{p^{m}(x+y\zeta)\beta_1\equiv p^{k-m}\gamma_2 \, (p^{k})\\ \beta_1\equiv \beta_1'\, (p^m)}} 1=\tilde{N}_1(p^k).
  \end{multline}
  Hence we deduce
  \begin{align}
    D_1(p^k):=&\ \tilde{N}_1(p^{k+1})-\tilde{N}_1(p^k) \nonumber\\
    =&\ \frac{1}{p^{4k+4+2m}}\sum_{\gamma_2\, (p^m)}\psi_{p^m}(-\beta_2'\gamma_2)\sum_{x,y\, (p^{k+1-m})\\(x,y,p)=1}\sum_{\beta_1 \, (p^{k+1})\\ p^m(x+y\zeta)\beta_1\equiv p^{k+1-m}\gamma_2 \, (p^{k+1})\\ \beta_1\equiv \beta_1'\, (p^m)} 1.
  \end{align}

  Since $m\leq k$, $\gamma_2$ is completely determined by the other variables, giving us
  \begin{equation}
    D_1(p^k) \leq \frac{1}{p^{4k+2m+4}}\sum_{x,y\, (p^{k+1-m})\\ (x,y,p)=1}\sum_{\beta_1 \, (p^{k+1})\\ p^m(x+y\zeta)\beta_1\equiv 0 \, (p^{k+1-m})} 1.
  \end{equation}

  For $k+1\leq 2m$, we have
  \begin{equation}
    D_1(p^k)\leq \frac{1}{p^{4k+2m+4}}p^{2(k+1-m)}p^{4(k+1)}=p^{2k-4m+2}\leq p^{6m-3k-3}.
  \end{equation}

  For $k+1>2m$, we have
  \begin{align}
    D_1(p^k)\leq&\  \frac{1}{p^{4k-8m+4}}\sum_{x,y\, (p^{k+1-2m})\\(x,y,p)=1}\sum_{\beta_1 \, (p^{k+1-2m})\\ (x+y\zeta)\beta_1\equiv 0 \, (p^{k+1-2m})} 1\nonumber\\
    =&\ \frac{1}{p^{4k-8m+4}}\sum_{x,y\, (p^{k+1-2m})\\(x,y,p)=1}N((x+y\zeta,p^{k+1-2m})).
  \end{align}
  Changing the variable $x\mapsto xy$ if $(y,p)=1$ and $y\mapsto xy$ if $(x,p)=1$, we obtain
  \begin{equation}
    D_1(p^k)\leq \frac{1}{p^{3k-6m+3}}\sumCp_{x\, (p^{k+1-2m})}\left(N((x+\zeta,p^{k+1-2m}))+N((1+x\zeta,p^{k+1-2m}))\right).
  \end{equation}
  Lemma \ref{lem.sumNorm} then implies the desired bound for $D_1(p^k)$, which concludes the proof of the second statement.

  Finally, we bound $\tilde{N}_1(p^k)$. For $k=m$, we have
  \begin{equation}
    \tilde{N}_1(p^m)=\frac{1}{p^{6m}}\charf{\ensuremath{\boldsymbol{\ell}(\beta_1'\beta_2'\equiv0\, (p^m))}}\leq p^{-6m}.
  \end{equation}
  For any $k>m$, we have
  \begin{equation}
    \tilde{N}_1(p^k)=\tilde{N}_1(p^m)+\sum_{j=m}^{k-1}D_1(p^j)\ll p^{-6m}+\sum_{j=m}^{k-1}p^{4m-2j-2}\ll p^{2m}.
  \end{equation}
  This concludes the proof.
\end{proof}

\section{Bounding $S_2$}\label{sec:cq72jm8woo}

Since we will choose the parameter $D$ in a way that shifts the burden of the trivial
bound  from $\Sigma_2$ onto $\Sigma_1$, we require only an upper bound for $S_2$ for $\alpha_1\alpha_2\neq0$.
We shall show a somewhat less complete bound than we did in \S\ref{sec:cq72jm8eaf},
opting to use the divisor bound from the sums over $d, q$ in \S\ref{sec:cuhge2kjb1} to yield the
final bound on $\Sigma_2^{\set{1,2 }}$.

Recall the definition of $S_2$ in \eqref{S2Def}, that
\begin{equation}
  S_2(\alpha_1,\alpha_2;\mathbf{c}, d,q)=\frac{1}{d^3q^3}\sum_{\beta_1, \beta_2\in \mathcal{O}_K/dq\mathcal{O}_K \\
    \det(\mathbf{c}, \ensuremath{\boldsymbol\ell}(\beta_1\beta_2))\equiv 0 \, (d q/(q, M)) \\
    \ensuremath{\boldsymbol\ell} (\beta_1\beta_2)\equiv 0 \, (d) \\
    \beta_1 \equiv \beta_1' \, ((dq, M)) \\
    \beta_2 \equiv \beta_2' \, ((dq, M))
  } \psi_{d q}(\alpha_1\beta_1 + \alpha_2\beta_2).
\end{equation}

We begin by using orthogonality to detect $\beta_i \equiv \beta_i' \, ((dq, M))$, which yields that
\begin{multline}\label{S2DetectM}
  S_2(\alpha_1, \alpha_2;\mathbf{c}, d, q)  
  = \frac{1}{(dq, M)^8}\sum_{\gamma_1, \gamma_2 \, ((dq, M)) }\psi_{(dq, M)}(-\gamma_1\beta_1'-\gamma_2\beta_2')
  S_2^*(\tilde \alpha_1, \tilde \alpha_2; \mathbf{c}, d, q) 
\end{multline}
where
\begin{equation}
  S_2^*(\alpha_1, \alpha_2; \mathbf{c}, d, q)
  = \frac{1}{d^3q^3}\sum_{\beta_1, \beta_2\in \mathcal{O}_K/dq\mathcal{O}_K \\
    \det(\mathbf{c}, \ensuremath{\boldsymbol\ell}(\beta_1\beta_2)) \equiv 0 \, (dq / (q, M)) \\
    \ensuremath{\boldsymbol\ell}(\beta_1\beta_2) \equiv 0 \, (d)} \psi_{dq}(\alpha_1\beta_1 + \alpha_2\beta_2),
\end{equation}
with 
\begin{equation}
  \tilde \alpha_i =  \alpha_i + \frac{dq}{(dq, M)}\gamma_i.
\end{equation}
As we do not care for losses of $M^{O(1)}$, we will be fixing $\gamma_1, \gamma_2$ from now on and
just bounding $S_2^*$.
Write
\begin{equation}
  \gamma_{\mathbf{c}}(x, y, z; q) := (q,M)x(c_2 - c_1\zeta) + q(y + z\zeta).
\end{equation}
By orthogonality, we have
\begin{align}\label{S2AfterOrthogonality}
  S_2^*(\alpha_1,\alpha_2;\mathbf{c}, d,q) 
  =&\ \frac{1}{d^6q^4}\sum_{x\, (dq)\\ y, z \, (d)}\sum_{\beta_1, \beta_2\in \mathcal{O}_K/dq\mathcal{O}_K} \nonumber
  \psi_{d q}(\alpha_1\beta_1 + \alpha_2\beta_2- \gamma_{\mathbf{c}}(x, y, z; q)\beta_1\beta_2 )\nonumber\\
  =&\ \frac{1}{d^2} \sum_{x \, (dq) \\ y, z \, (d) } \sum_{\beta_1 \in \mathcal{O}_K/ dq\mathcal{O}_K \\ \gamma_{\mathbf{c}}(x,y,z; q)\beta_1 \equiv \alpha_2 \, (dq)}\psi_{dq} (\alpha_1\beta_1).
\end{align}

We shall show the following bound on $S_2^*$.
\begin{proposition} \label{prop.S2non0}
  We have that
  \begin{multline}
    |S_2^*(\alpha_1, \alpha_2; \mathbf{c}, d, q)| \\
    \ll (dq)^{o(1)}
    \sum_{g'h' | d \\ g''h'' | q \\ g'g'' | \alpha_1\alpha_2 \\ (h'h'')^2 | N(\alpha_1\alpha_2/g) }
    \frac{g^4h}{(dq)^2}\sum_{x, y, z \, (dq) \\g | \gamma_{\mathbf{c}}(x,y ,z; q)  }
    \sum_{r | dq } \frac{1}{r}\mathbbm{1}[Y_\mathbf{c}(\alpha_1\alpha_2, x, y, z; q)  \equiv 0 \, (dq/r)],
  \end{multline}
  where we write $g = g'g'', h = h'h''$ and
  \begin{align}
    Y_\mathbf{c}(\alpha, x, y, z; q) &= \langle \alpha, \gamma_{\mathbf{c}}(x, y, z; q)^* \rangle.
  \end{align}
  Here $\gamma^* = N(\gamma)/\gamma$ for $\gamma\in \mathcal{O}_K$, so $Y_\mathbf{c}(\alpha, x,y ,z; q)$ is a cubic form in $x, y, z$. In particular, we have
  \begin{multline}
    |S_2(\alpha_1, \alpha_2; \mathbf{c}, d, q)| \\
    \ll (dq)^{o(1)}
    \sum_{g'h' | d \\ g''h'' | q \\ g'g'' | M\alpha_1\alpha_2 \\ h'^2h''^2 | M^2N(\frac{\alpha_1\alpha_2}{g'g''}) }
    \frac{g^4h}{(dq)^2}\sum_{x, y, z \, (dq) }
    \sum_{r | dq }
    \frac{1}{r}\mathbbm{1}[MY_\mathbf{c}(\alpha_1\alpha_2, x, y, z; q)  \equiv 0\, (dq/r)].
  \end{multline}
\end{proposition}

\begin{proof}

  As $S_2^*$ is multiplicative in $d, q$, we may suppose throughout the proof that
  $q$ and $d$ are both powers of some prime $p$.
  
  Starting from \eqref{S2AfterOrthogonality}, we scale and average over $(t, dq) = 1$ and then apply orthogonality, in the same manner as in the proof of Proposition \ref{proposition:cq6y0ob20l}. This gives us
  \begin{equation}
    S_2^*(\alpha_1, \alpha_2; \mathbf{c}, d, q) = \frac{dq}{\varphi(dq)} \sum_{r | dq }
    \frac{\mu(r)}{(dq)^2r} \sum_{x, y, z \, (dq) } \sum_{\beta_1 \, (dq) \\ \langle \alpha_1, \beta_1 \rangle \equiv 0 \, (dq/r) \\ \gamma_{\mathbf{c}}(x, y, z; q)\beta_1 \equiv \alpha_2 \, (dq) } 1.
  \end{equation}
  Note that
  \begin{equation}\label{eq:cratrvm311}
    \begin{cases}\langle \alpha_1, \beta_1 \rangle \equiv 0\, (dq/r)\\ \gamma_\mathbf{c}(x, y, z; q)\beta_1 \equiv \alpha_2\, (dq) \end{cases}
    \implies \langle \alpha_1\alpha_2, \gamma_{\mathbf{c}}(x, y, z; q)^* \rangle \equiv 0\, (dq/r).
  \end{equation}
  By the symmetry of $\alpha_1, \alpha_2$ in $S_2$, we shall suppose without loss of generality that
  \begin{equation}\label{eq:crat5mf5qz}
    v_p(N(\alpha_1))\ge v_p(N(\alpha_2)).
  \end{equation}
  Suppose that $(\gamma_\mathbf{c}(x, y, z; q), dq) = (\eta)g$ for some $g\ge 1$ and some primitive (not divisible by
  any rational prime) $\eta\in \mathcal{O}_K$ with $N(\eta) = h$, say.
  Then, the congruence condition on $\alpha_2$ implies that $g\eta | \alpha_2$.
  So by \eqref{eq:crat5mf5qz}, we have
  $h^2 | N(\alpha_1\alpha_2/g)$. We furthermore have that $h g | dq$, and we may split so that $h = h'h''$
  and $g = g'g''$ with $h'g' | d, h''g'' | q$.
  
  Note that the number of $\beta_1$ satisfying the congruence condition is bounded by $g^4h$. Combined with \eqref{eq:cratrvm311}, we obtain that
  \begin{multline}
    |S_2^*(\alpha_1,\alpha_2, \mathbf{c}, d, q)| \\ 
    \ll \frac{dq}{\varphi(dq)}
    \sum_{h'g' | d \\ h''g'' | q \\ g'g'' | \alpha_1\alpha_2 \\   h'h'' | N(\alpha_1\alpha_2/g)} 
    \frac{g^4h}{(dq)^2} \sum_{x, y, z \, (dq) \\ g | \gamma_{\mathbf{c}}(x,y ,z; q) }
    \sum_{r | dq } \frac{1}{r}\mathbbm{1}[\langle \alpha_1\alpha_2, \gamma_{\mathbf{c}}(x,y ,z; q)^* \rangle \equiv 0 \, (dq/r)]
  \end{multline}
  with the notation $g=g'g''$ and $h=h'h''$, as desired.
\end{proof}

For clarity, we renormalize
at this point, taking 
\begin{equation}\label{N2Def}
  \tilde{N}_2(\mathbf{c},d;q)=\frac{1}{d^3q^4}S_2(0,0;d,q)=\frac{1}{d^6q^7}
  \sum_{\beta_1,\beta_2\in \mathcal{O}_K/ dq\mathcal{O}_K \\ \det(\mathbf{c},\ensuremath{\boldsymbol\ell}(\beta_1\beta_2))\equiv 0 \, (dq/(q, M)) \\ \ensuremath{\boldsymbol\ell}(\beta_1\beta_2)\equiv 0 \, (d) \\ \beta_i\equiv \beta_i'\, ((dq,M))} 1.
\end{equation}
Notice that $\tilde{N}_2(\mathbf{c},d,q)$ is multiplicative in $d$ and $q$. Recall the notation that we take
$m_p = v_p(M)$.

\begin{proposition}\label{prop:S20}
  Let $\mathbf{c}=(c_1,c_2)$ with $c_1,c_2$ being positive integers coprime with each other.
  Let $p$ be prime and take integers $h\ge 0$ and $k\ge m_p$. Then we have the bound
  \begin{equation}\label{eq:crb7a0eq6k}
    \tilde{N}_2(\mathbf{c},p^h;p^k)\ll (h+k+1)p^{2m_p}.
  \end{equation}
  Write $p^\ell=(c_1^4+c_2^4, p^{h + k+1})$. Then we have
  \begin{equation}\label{eq:cuhe1uxr7w}
    |\tilde{N}_2(\mathbf{c},p^h;p^{k+1})-\tilde{N}_2(\mathbf{c},p^h;p^k)|\ll (h+k+1)p^{4m_p-2h-3k-3+\ell}.
  \end{equation}
\end{proposition}
\begin{proof}
  Write $m=m_p$ for simplicity. We first prove the second statement.
  From \eqref{S2DetectM} to \eqref{S2AfterOrthogonality} (summing over $\gamma_1$ and relabelling $\beta_1$ and $\gamma_2$ as $\beta$ and $\gamma$ respectively), we have that
  \begin{equation}\label{eq.prop9.2temp}
    \tilde{N}_2(\mathbf{c},p^h; p^{k+1}) = \frac{1}{p^{4m+5h+4k+4}}\sum_{\gamma\, (p^m)}\psi_{p^m}(-\gamma\beta_2')\mathop{\sum_{x\, (p^{h+k+1})}\sum_{y, z\, (p^h)}\sum_{\beta\, (p^{h+k+1})}}_{\substack{\gamma_{\mathbf{c}}(x,y,z; p^{k+1})\beta\equiv p^{h+k+1-m}\gamma\, (p^{h+k+1})\\ \beta\equiv \beta_1'\, (p^m)}} 1.
  \end{equation}
  Notice that the contribution of $p|x$ to above is given by
  \begin{align}
    &\ \frac{1}{p^{4m+5h+4k+4}}\sum_{\gamma\, (p^m)}\psi_{p^m}(-\gamma\beta_2')\mathop{\sum_{x\, (p^{h+k})}\sum_{y, z\, (p^h)}\sum_{\beta\, (p^{h+k+1})}}_{\substack{\gamma_{\mathbf{c}}(px,y,z; p^{k+1})\beta\equiv p^{h+k+1-m}\gamma\, (p^{h+k+1})\\ \beta\equiv \beta_1'\, (p^m)}} 1\nonumber\\
    &\ = \frac{1}{p^{4m+5h+4k}}\sum_{\gamma\, (p^m)}\psi_{p^m}(-\gamma\beta_2')\mathop{\sum_{x\, (p^{h+k})}\sum_{y, z\, (p^h)}\sum_{\beta \, (p^{h+k})}}_{\substack{\gamma_{\mathbf{c}}(x,y,z;p^k)\beta \equiv p^{h+k-m}\gamma \, (p^{h+k})\\ \beta \equiv \beta_1'\, (p^m)}} 1=\tilde{N}_2(\mathbf{c},p^h;p^k).
  \end{align}
  Hence the difference is given by
  \begin{align}
    D(\mathbf{c},p^h;p^k):=&\ \tilde{N}_2(\mathbf{c},p^{h};p^{k+1})-\tilde{N}_2(\mathbf{c},p^h;p^k)\nonumber\\
    =&\ \frac{1}{p^{4m+5h+4k+4}}\sum_{\gamma\, (p^m)}\psi_{p^m}(-\gamma\beta_2')\mathop{\sumCp_{x\, (p^{h+k+1})}\sum_{y, z\, (p^h)}\sum_{\beta \,(p^{h+k+1})}}_{\substack{\gamma_{\mathbf{c}}(x,y,z; p^{k+1})\beta\equiv p^{h+k+1-m}\gamma\, (p^{h+k+1})\\ \beta\equiv \beta_1'\, (p^m)}}  1.
  \end{align}
  Note that $\gamma$ is completely determined by the other variables, giving us
  \begin{align}
    D(\mathbf{c},p^h;p^k)\ll &\ \frac{1}{p^{4m+5h+4k+4}}\mathop{\sumCp_{x\, (p^{h+k+1})}\sum_{y, z\, (p^h)}\sum_{\beta \, (p^{h+k+1})}}_{\substack{\gamma_{\mathbf{c}}(x,y,z; p^{k+1})\beta\equiv 0\, (p^{h+k+1-m})}}  1,
  \end{align}
  which is equal to 
  \begin{equation}\label{abiggerCase}
    \frac{1}{p^{3h+4k+4-m}}\mathop{\sumCp_{x\, (p^{h+k+1-m})}\sum_{\beta\, (p^{h+k+1-m})}}_{\substack{p^m x(c_2-c_1\zeta)\beta\equiv 0\, (p^{h+k+1-m})}}1
  \end{equation}
  if $m>h$, and
  \begin{equation}\label{tempaleqCase}
    \frac{1}{p^{5h+4k+4-3m}}\mathop{\sumCp_{x\, (p^{h+k+1-m})}\sum_{y, z\, (p^{h-m})}\sum_{\beta\, (p^{h+k+1-m})}}_{\substack{\gamma_{\mathbf{c}}(x,y,z; p^{k+1})\beta\equiv 0\, (p^{h+k+1-m})}}1.
  \end{equation}
  if $m\leq h$.

  For $m>h$, counting the number of $\beta$, Lemma \ref{lem.Only1Prime} with $(c_1,c_2)=1$ implies that the sum in \eqref{abiggerCase} is
  \begin{multline}\label{tempabiggerConclusion}
    \leq \frac{1}{p^{2h+3k+3}}\sum_{\beta\, (p^{h+k+1-m})\\ (c_2-c_1\zeta)\beta\equiv 0\, (p^{\max\{0,h+k+1-2m\}})}1\leq \frac{1}{p^{2h+3k+3-4m}}N((c_2-c_1\zeta,p^{\max\{0,h+k+1-2m\}}))\\
    =\frac{1}{p^{2h+3k+3-4m}}(c_1^4+c_2^4,p^{\max\{0,h+k+1-2m\}})\leq p^{4m-2h-3k-3+\ell}.
  \end{multline}

  For $m\leq h$, counting the number of $\beta$, \eqref{tempaleqCase} is equal to
  \begin{equation}
    \frac{1}{p^{5h+4k+4-3m}}\sumCp_{x\, (p^{h+k+1-m})}\sum_{y, z\, (p^{h-m})} N((\gamma_{\mathbf{c}}(x,y,z;p^{k+1}), p^{h+k+1-m})).
  \end{equation}

  \underline{\textbf{Case 1:}} Suppose $(c_2,p)=1$, then we can perform a change of variable
  $x\mapsto \overline{c_2}(x-p^{k+1-m}y)$ and $z\mapsto z-c_1\overline{c_2}y$ to get
  \begin{multline}\label{S200temp1}
    D(\mathbf{c},p^h;p^k)\ll \frac{1}{p^{4h+4k+4-2m}}\sumCp_{x\, (p^{h+k+1-m})}\sum_{z\, (p^{h-m})}\\
    N((p^m x+(p^{k+1}z-c_1\overline{c_2}p^m x)\zeta,p^{h+k+1-m})).  
  \end{multline}
  Write $c_1=p^bc_1'$ for some $b\geq0$ and $(c_1',p)=1$. 

  If $b\geq 1$, we have
  \begin{equation}\label{Case1.1}
    D(\mathbf{c},p^h;p^k)\ll p^{2m-4h-4k-4}\sumCp_{x\, (p^{h+k+1-m})}\sum_{z\, (p^{h-m})}p^{4m}\ll p^{4m-2h-3k-3}.
  \end{equation}

  On the other hand, if $b=0$, we have
  \begin{multline}
    D(\mathbf{c},p^h;p^k)\ll \frac{1}{p^{4h+4k+4-6m}}\sumCp_{x\, (p^{h+k+1-m})}\sum_{g=0}^{h-m}\sumCp_{z\, (p^{h-m-g})}\\
    N((x+(p^{k+g+1-m}z-\overline{c_2}c_1x)\zeta,p^{h+k+1-2m})).
  \end{multline}
  Lemma \ref{lem.Only1Prime} then implies that this is equal to
  \begin{align}\label{Case1.3temp}
    \sum_{g=0}^{h-m}\frac{1}{p^{4h+4k+4-6m}}\sumCp_{x \, (p^{h+k+1-m})}\sumCp_{z\, (p^{h-m-g})}(x^4+(p^{k+g+1-m}z-\overline{c_2}c_1x)^4,p^{h+k+1-2m}).
  \end{align}
  Expanding the fourth power, Hensel's lemma (similar to Lemma \ref{lem.sumNorm}) then implies that
  \begin{align}\label{Case1.2}
    D(\mathbf{c},p^h;p^k)\ll &\ \sum_{g=0}^{h-m}\frac{h+k+1}{p^{4h+4k+4-6m}}\sumCp_{z\, (p^{h-m-g})}p^{h+k+1-m+\min\{\ell,k+g+1-m\}}\nonumber\\
    \ll&\ (h+k+1)p^{4m-2h-3k-3+\ell}.
  \end{align}

  \underline{\textbf{Case 2:}} Suppose $p|c_2$, then $(c_1,p)=1$ as $(c_1,c_2)=1$. Write $p^b||c_2$ for some
  $b\geq1$. Performing a change of variable $x\mapsto \overline{c_1}(x+p^{k+1-m}z)$ and $y\mapsto y-\overline{c_1}c_2z$ yields
  \begin{align}\label{Case2Prop9.2}
    D(\mathbf{c},p^h;p^k)\ll &\ \frac{1}{p^{4h+4k+4-2m}}\sumCp_{x \, (p^{h+k+1-m})}\sum_{y \, (p^{h-m})}N((\overline{c_1}c_2p^mx+p^{k+1}y-p^mx\zeta,p^{h+k+1-m}))\nonumber\\
    =&\ \frac{1}{p^{4h+4k+4-2m}}\sumCp_{x \, (p^{h+k+1-m})}\sum_{y \, (p^{h-m})}p^{4m}\ll p^{4m-2h-3k-3}.  
  \end{align}

  Observe that \eqref{tempabiggerConclusion}, \eqref{Case1.1} and \eqref{Case2Prop9.2} are all smaller than \eqref{Case1.2}. Hence \eqref{Case1.2}
  holds in all cases.

  Finally, we have to show the bound for $\tilde{N}_2$. Starting again from \eqref{eq.prop9.2temp}, with $\gamma$
  completely determined by the other variables, we have
  \begin{align}
    \tilde{N}_2(\mathbf{c},p^h;p^m)\leq \frac{1}{p^{8m+5h}}\mathop{\sum_{x\, (p^{h+m})}\sum_{y,z\, (p^h)}\sum_{\beta\, (p^{h+m})}}_{(x(c_2-c_1\zeta)+y+z\zeta)\beta\equiv0\, (p^h)}1= \frac{1}{p^{3m+5h}}\mathop{\sum_{x, y,z\, (p^h)}\sum_{\beta\, (p^{h})}}_{(x(c_2-c_1\zeta)+y+z\zeta)\beta\equiv0\, (p^h)}1.
  \end{align}
  Performing change of variables $y\mapsto y-c_2x$ and $z\mapsto z+c_1x$ gives us
  \begin{equation}
    \tilde{N}_2(\mathbf{c},p^h;p^m)\leq \frac{1}{p^{3m+4h}}\mathop{\sum_{y,z\, (p^h)}\sum_{\beta\, (p^h)}}_{(y+z\zeta)\beta\equiv0\, (p^h)}1=\frac{1}{p^{3m+4h}}\sum_{y,z\, (p^h)}N((y+z\zeta,p^h)).
  \end{equation}
  Pulling out the powers of $p$ from $y$ and $z$, we have
  \begin{align}
    \tilde{N}_2(\mathbf{c},p^h;p^m)\leq &\ \frac{1}{p^{3m+4h}}\sum_{0\leq s,t\leq h}\sumCp_{y\, (p^{h-s})}\sumCp_{z\, (p^{h-t})}N((p^sy+p^tz\zeta,p^h))\nonumber\\
    \leq&\ \frac{1}{p^{3m+4h}}\sum_{0\leq s, t\leq h \\ s\neq t}p^{2h-s-t+4\min\{s,t\}}+\frac{1}{p^{3m+4h}}\sum_{0\leq s\leq h} \sumCp_{y,z\, (p^{h-s})} p^{4s}N((y+z\zeta,p^{h-s})).
  \end{align}
  Note that
  \begin{align}
    \frac{1}{p^{3m+4h}}\sum_{0\leq s, t\leq h\\ s\neq t}p^{2h-s-t+4\min\{s,t\}}\leq &\ \frac{2}{p^{3m+2h}}\sum_{0\leq s<t\leq h}p^{2s}\ll \frac{h+1}{p^{3m}}.
  \end{align}
  On the other hand, a change of variable $y\mapsto yz$ gives
  \begin{align}
    \frac{1}{p^{3m+4h}}\sum_{0\leq s\leq h} p^{4s}\sumCp_{y,z\, (p^{h-s})} N(y+z\zeta,p^{h-s})\leq &\ \sum_{0\leq s\leq h}\frac{1}{p^{3m+3h-3s}} \sumCp_{y\, (p^{h-s})} N(y+\zeta,p^{h-s}).
  \end{align}
  Lemma \ref{lem.sumNorm} gives us that this is $\ll (h+1)p^{-3m}$. Hence we have the bound
  \begin{equation}
    \tilde{N}_2(\mathbf{c},p^h;p^m)\ll \frac{h+1}{p^{3m}}.
  \end{equation}
  Together with \eqref{Case1.2} which holds for all cases, we have for $k\geq m$,
  \begin{align}
    \tilde{N}_2(\mathbf{c},p^h;p^k)= \tilde{N}_2(\mathbf{c},p^h;p^m)+\sum_{j=m}^{k - 1}D(\mathbf{c},p^h;p^j)\ll (h+k+1)p^{2m}.
  \end{align}
  This concludes the proof.
\end{proof}

\section{Oscillatory integral estimation}\label{sec:oscillatory}

In this section, we carry out the estimation of the oscillatory integrals $I_1, I_2$ defined
in \eqref{eq:crb7ay3tq3}, \eqref{I2Def}, respectively. 
Their treatments will be very similar to the applications of orthogonality that went into
the treatments of the exponential sums $S_1$, $S_2$ in \S\ref{sec:cq72jm8eaf}, \ref{sec:cq72jm8woo}, though matters will be
simpler in this archimedean analogue for we only seek an upper bound.

\subsection{Estimation of $I_1$}
The main result of this subsection is the following.
\begin{proposition}\label{proposition:crale1hrji}
  For any $A, k\ge 0$, we have that
  \begin{equation}\label{eq:cq9j69pj16}
    q^k\frac{d^k}{dq^k} I_1(\alpha_1, \alpha_2; q)
    \ll_{A,k} 
    |\tilde\alpha_1\tilde\alpha_2|_{\sup}^{-\frac{5}{6}} \, T^{-13/3+o(1)}\biggl( 1 + |\tilde{\alpha}_1|_{\sup}
    + |\tilde{\alpha}_2|_{\sup}\biggr)^{-A},
  \end{equation}
  where we write
  \begin{equation}
    \tilde \alpha_i = \frac{\alpha_i D X_i}{MX} \quad \text{ and } \quad T = \frac{MX}{q D}.
  \end{equation}
\end{proposition}
\begin{proof}
  Differentiating under the integral, it is sufficient to bound, for any
  $\omega^*\in C_c^\infty(\mathbb{R}^2\setminus\set{0}) $ and $\Phi_1, \Phi_2\in C_c^\infty(K_\infty\setminus\set{0})$, the integral
  \begin{equation}\label{eq:cq9j672xse}
    I_T(\omega^*, \Phi_1, \Phi_2) := \int_{K_\infty^2} \omega^* ( T\ensuremath{\boldsymbol\ell}(x_1^\infty x_2^\infty )) \Phi_1(x_1^\infty)\Phi_2(x_2^\infty)
    e(-T(\langle x_1^\infty, \tilde\alpha_1 \rangle + \langle x_2^\infty, \tilde\alpha_2 \rangle)) \,d x_1^\infty \,d x_2^\infty.
  \end{equation}
  Repeated integration by parts implies that
  \begin{equation}
    I_T(\omega^*, \Phi_1, \Phi_2)\ll (1 + |\tilde{\alpha}_1|_{\sup}
    +|\tilde{\alpha}_2|_{\sup})^{-A},
  \end{equation}
  This immediately gives the statement if $T\ll X^\varepsilon$. So we'll
  suppose from now on that $T\gg X^\varepsilon$ and $|\tilde\alpha_i|_{\sup}\ll T^{o(1)}$.
  
  We shall suppose without loss of generality that $|\tilde\alpha_1|_{\sup} \le |\tilde\alpha_2|_{\sup}$. Recalling
  $\ensuremath{\boldsymbol\ell}(x_\infty) = (\langle x_\infty, 1 \rangle, \langle x_\infty, \zeta \rangle)$, Fourier inversion implies that \eqref{eq:cq9j672xse}
  equals
  \begin{multline}
    \int_{\mathbb{R}^2} \widehat{\omega^*}(y_1, y_2) \int_{K_\infty^2} \Phi_1(x_1^\infty) \Phi_2(x_2^\infty)\\
    \cdot e(-T(\langle x_1^\infty, \tilde\alpha_1 \rangle  + \langle x_2^\infty, \tilde\alpha_2 \rangle - \langle x_1^\infty x_2^\infty, y_1 + y_2\zeta \rangle  ))
    \,d x_1^\infty \,d x_2^\infty \,d y_1 \,d y_2 \\
    = \int_{\mathbb{R}^2} \widehat{\omega^*}(y_1, y_2) \int_{K_\infty} \Phi_1(x_1^\infty) \widehat{\Phi}_2( T( \tilde\alpha_2 - x_1^\infty(y_1 + y_2\zeta)))
    e(-T \langle x_1^\infty, \tilde\alpha_1 \rangle)
    \,d x_1^\infty \,d y_1 \,d y_2.
  \end{multline}
  We suppose from now on that 
  \begin{equation}\label{alpha2assumption}
    |\tilde\alpha_2|_{\sup} \gg \frac{1}{T^{1 - 2\varepsilon}}, 
  \end{equation}
  for otherwise the statement immediately holds with $|\tilde\alpha_1\tilde\alpha_2|_{\sup}\leq |\tilde\alpha_1|_{\sup}|\tilde\alpha_2|_{\sup}\ll T^{-2+4\varepsilon}$.
  Let $\eta_\infty = T(\tilde\alpha_2 - x_1^\infty(y_1 + y_2\zeta))$ so that by the change of variables
  \begin{equation}
    \,d x_1^\infty = \frac{\, d \eta_\infty}{T^4(y_1^4 + y_2^4)},
  \end{equation}
  we have that $I_T(\omega^*, \Phi_1, \Phi_2)$ is equal to 
  \begin{equation}
    \frac{1}{T^4}
    \int_{\mathbb{R}^2} \frac{\widehat{\omega^*}(y_1, y_2)}{y_1^4 + y_2^4}
    \int_{K_\infty} \widehat{\Phi}_2(\eta_\infty) \Phi_1 \biggl(\frac{\tilde\alpha_2 - \eta_\infty/T}{y_1 + y_2\zeta}\biggr)
    e\bigg(- \bigg\langle \frac{ T\tilde \alpha_2 - \eta_\infty }{y_1 + y_2\zeta},
    \tilde\alpha_1 \bigg\rangle\bigg)\, d \eta_\infty \, dy_1 \, dy_2.
  \end{equation}
  The contribution of $|\eta_\infty|_{\sup} > T^{\varepsilon}$ is $\ll_A T^{-A}$ from the decay of $\widehat\Phi_2$.
  Also, from the support of $\Phi_1$ and the fact that both embeddings of $y_1 + y_2\zeta$ are
  $\asymp |y_1^4 + y_2^4|^{1/4}\asymp |\mathbf{y}|$ in magnitude, we have that in the support of the integrand above
  with $|\eta_\infty|_{\sup}\le T^{\varepsilon}$ and \eqref{alpha2assumption}, 
  \begin{equation}
    1\asymp \biggl|\frac{\tilde\alpha_2 - \eta_\infty/T}{y_1 + y_2\zeta}\biggr|_{\sup} \asymp
    \biggl| \frac{\tilde\alpha_2}{y_1 + y_2\zeta} \biggr|_{\sup}
    \asymp \frac{|\tilde\alpha_2|_{\sup}}{|\mathbf{y}|}.
  \end{equation}
  In particular, it suffices to bound for $|\eta_\infty|_{\sup}\le T^{\varepsilon}$, some function
  $\Psi : \mathbb{R} \times \mathbb{R} \times K_\infty\to \mathbb{C}$ with fixed compact support, compactly supported away from $0$ on the third coordinate and any derivatives of order $j$ being $O_j(1)$, and
  some $\tilde\alpha_i' \, (=\tilde \alpha_i - \eta_\infty/T)$ satisfying
  \begin{equation}\label{eq:3}
    |\tilde\alpha_i'|_{\sup}\asymp |\tilde\alpha_i|_{\sup}\asymp |\tilde\alpha_i'|_{\infty}^{\frac{1}{4}}
  \end{equation}
  the quantity (after a change of variable $y_i\mapsto |\tilde\alpha_2'|_{\sup} y_i$)
  \begin{equation}\label{eq:cran51w4hy}
    \frac{|\tilde\alpha_2'|_{\sup}^{-2}}{T^4}\int_{|\mathbf{y}|\asymp 1}
    \Psi \biggl( y_1, y_2, \frac{\tilde\alpha_2'/|\tilde\alpha_2'|_{\sup}}{y_1 + y_2\zeta} \biggr)
    e \biggl( - T|\tilde\alpha_1'|_{\sup}
    \biggl\langle \frac{\tilde\alpha_1'\tilde\alpha_2'|\tilde\alpha_1'|_{\sup}^{-1}
      |\tilde\alpha_2'|_{\sup}^{-1} }{y_1 + y_2\zeta}, 1 \biggr\rangle \biggr)
    \,d y_1 \,d y_2.
  \end{equation}
  It now remains to execute the $y_1, y_2$ integrals. Taking any fixed $\phi\in C_c^\infty((1, 2))$ with
  \begin{equation}
    \int_{\mathbb{R}} \phi(t) \,d t = 1,
  \end{equation}
  applying a change of variables $(y_1, y_2)\to(y_1/t, y_2/t)$, and integrating against
  $\phi(t) \,d t$ yields that \eqref{eq:cran51w4hy} equals
  \begin{equation}
    \frac{|\tilde\alpha_2'|_{\sup}^{-2}}{T^4}\int_{\mathbb{R}} \frac{\phi(t)}{t^2}
    \int_{|\mathbf{y}|\asymp 1} \Psi \biggl( y_1/t, y_2/t, t\frac{\tilde\alpha_2'/|\tilde\alpha_2'|_{\sup} }{y_1 + y_2\zeta} \biggr) e \biggl( -Tt|\tilde\alpha_1'|_{\sup} \biggl\langle \frac{\tilde\alpha_i'\tilde\alpha_2'|\tilde\alpha_1'|_{\sup}^{-1}
      |\tilde\alpha_2'|_{\sup}^{-1} }{y_1 + y_2\zeta}, 1 \biggr\rangle\biggr) \,d y_1 \,d y_2 \,d t.
  \end{equation}
  Executing the $t$-integral and noting the decay of the Fourier transform in $t$
  along with a change of variable $y_2\mapsto -y_2$ reduces us to bounding
  \begin{equation}\label{eq:cran52wce2}
    \frac{|\tilde\alpha_2|_{\sup}^{-2}}{T^4} \int_{|\mathbf{y}|\asymp 1}
    \mathbbm{1}\biggl[u_0 y_2^3 + u_1 y_1y_2^2 + u_2 y_1^2y_2 + u_3y_1^3
    \le |\tilde\alpha_1'|_{\sup}\frac{T^{\varepsilon} }{T}\biggr] \,d y_1 \,d y_2
  \end{equation}
  where we write
  \begin{equation}
    \frac{\tilde\alpha_1'\tilde\alpha_2'}{|\tilde\alpha_1'|_{\sup}|\tilde\alpha_2'|_{\sup}}
    = u_0 + u_1\zeta + u_2\zeta^2 + u_3\zeta^3.
  \end{equation}
  The bounds \eqref{eq:3} imply that $\max_{0\le i\le 3}(|u_i|)\gg 1$ and so we obtain that
  \eqref{eq:cran52wce2} is 
  \begin{equation}
    \ll \frac{1}{T^4} |\tilde \alpha_2'|_{\sup}^{-2}|\tilde\alpha_1'|_{\sup}^{\frac{1}{3}} \,
    T^{-\frac{1}{3} + \frac{\varepsilon}{3}}
    \ll \frac{1}{T^5} |\tilde\alpha_1\tilde\alpha_2|_{\sup}^{-\frac{5}{6}} \, T^{\frac{2}{3} + \frac{\varepsilon}{3}}.
  \end{equation}
  The desired result follows.
\end{proof}

\subsection{Estimation of $I_2$}

The estimate for $I_2$ we show is as follows.
\begin{proposition}\label{prop.I2}
  For $|\mathbf{c}|d\asymp D$, we have that $I_2=0$ unless $q\ll \sqrt{DX}/d$, in which case
  we have the bound
  \begin{equation}\label{eq:cranqug9hv}
    I_2(\alpha_1, \alpha_2; \mathbf{c}, d, q) 
    \ll_A |\tilde\alpha_1|_{\sup}^{-3/2}|\tilde\alpha_2|_{\sup}^{-3/2}
    T^{-4+o(1)}\biggl( 1 + |\tilde{\alpha}_1|_{\sup}
    + |\tilde{\alpha}_2|_{\sup}\biggr)^{-A},
  \end{equation}
  where
  \begin{equation}
    T = \frac{M\sqrt{DX}}{dq} \quad \text{ and } \quad  \tilde \alpha_i = \frac{\alpha_iX_i}{M\sqrt{DX}}.
  \end{equation}
\end{proposition}
\begin{proof} 

  We have that
  \begin{equation}
    I_2(\alpha_1, \alpha_2 ; \mathbf{c}, d, q) =
    \omega_2 \biggl(\frac{1}{T}\biggr)\hat\phi_1 \biggl(T\tilde\alpha_1\biggr)
    \hat\phi_2 \biggl(T\tilde\alpha_2\biggr) -I,
  \end{equation}
  where \begin{equation}
    I :=  \int_{K_\infty^2} \omega_2(T \det(\tilde{\mathbf{c}}, \ensuremath{\boldsymbol\ell} (x_1^\infty x_2^\infty)))
    \phi_1(x_1^\infty)\phi_2(x_2^\infty)
    e (- T(\langle x_1^\infty, \tilde\alpha_1 \rangle + \langle x_2^\infty, \tilde\alpha_2 \rangle)) \,d x_1^\infty \,d x_2^\infty,
  \end{equation}
  with the notation $\tilde{\mathbf{c}} = \mathbf{c} d/D$. The first term may be absorbed into the bound \eqref{eq:cranqug9hv}, so it remains
  to bound $I$.
  Since $|\tilde{\mathbf{c}}|\asymp 1$, repeated integration by parts gives us $I\ll (1+| \tilde{\alpha}_1|_{\sup}+| \tilde{\alpha}_2|_{\sup})^{-A}$.
  We shall suppose from now on that $|\tilde\alpha_i|_{\sup}\ll X^{o(1)}$.
  
  As before, we begin with Fourier inversion. Noting that
  \begin{equation}
    \det(\tilde{\mathbf{c}}, \ensuremath{\boldsymbol\ell}(x_1^\infty x_2^\infty))
    = \langle \tilde c_2 - \tilde c_1\zeta, x_1^\infty x_2^\infty \rangle,
  \end{equation}
  we have
  \begin{multline}
    I =
    \int_{\mathbb{R}} \widehat\omega_2(y) \int_{K_\infty^2} \phi_1( x_1^\infty)\phi_2(x_2^\infty)
    e(-T (\langle x_1^\infty, \tilde\alpha_1 \rangle + \langle x_2^\infty, \tilde\alpha_2 \rangle
    + \langle x_1^\infty x_2^\infty, y(\tilde c_2 - \tilde c_1\zeta) \rangle))
    \,d x_1^\infty \,d x_2^\infty \\
    = \int_{\mathbb{R}} \widehat\omega_2(y) \int_{K_\infty} \phi_1(x_1^\infty)
    \hat\phi_2(T(\tilde\alpha_2 + y(\tilde c_2 - \tilde c_1\zeta)x_1^\infty))
    e(-T \langle x_1^\infty, \tilde\alpha_1 \rangle)\,d x_1^\infty.
  \end{multline}
  Similar to before, we make the change of variables
  \begin{equation}
    \eta_\infty = T(\tilde\alpha_2 + y( \tilde c_2 - \tilde c_1\zeta)x_1^\infty)
  \end{equation}
  so that
  \begin{equation}
    I = \frac{1}{T^4(\tilde c_1^4 + \tilde c_2^4)} \int_{\mathbb{R}} \frac{\widehat\omega_2(y)}{y^4}
    \int_{K_\infty} \phi_1 \biggl( \frac{-\tilde \alpha_2 + \eta_\infty/T}{y(\tilde c_2 - \tilde c_1\zeta)} \biggr)
    \hat\phi_2( \eta_\infty) e \biggl( -T \biggl\langle \frac{-\tilde\alpha_2 + \eta_\infty/T}{y(\tilde c_2 - \tilde c_1\zeta)}, \tilde\alpha_1\biggr\rangle \biggr) \, d \eta_\infty \,d y .
  \end{equation}
  As in the proof of Proposition \ref{proposition:crale1hrji}, we may restrict to the
  case of $|\eta_\infty|_{\sup} < T^\varepsilon, |\tilde\alpha_2|_{\sup}\gg T^{2\varepsilon - 1}$ and fix an $\eta_\infty$. Then, it is forced that
  $|y|\asymp |\tilde\alpha_2|_{\sup}$. A trivial bound then gives us
  \begin{equation}
    I\ll (1+| \tilde{\alpha}_1|_{\sup}+| \tilde{\alpha}_2|_{\sup})^{-A}T^{-4+o(1)}|\tilde{\alpha}_2|_{\sup}^{-3}.
  \end{equation}
  Noticing that the same bound holds with $\tilde{\alpha}_1$ by swapping the roles of $\alpha_1$ and $\alpha_2$, the
  desired result follows.
  
\end{proof}

\section{Local density estimates}\label{sec:cuhge2i6da}

In this section, we record a number of facts about local densities, their averages,
and associated objects, which will be of use in the following two subsections when
analyzing zero frequencies.

\subsection{Singular series estimates}\label{sec:crb7rws6oy}

Write \begin{equation}
  V_q = \set{(\beta_1, \beta_2)\in (\mathcal{O}_K/q\mathcal{O}_K)^2 : \beta_i \equiv \beta_i' \, (M) \text{ for } i=1, 2},
\end{equation}
and for $d,q\geq1$, $\mathbf{c}=(c_1,c_2)\in \mathbb{Z}^2$ with $(c_1,c_2)=1$, let 
\begin{align}
  \tilde N_1(q) = \frac{1}{q^6} \sum_{(\beta_1, \beta_2)\in V_q \\ \ensuremath{\boldsymbol\ell}(\beta_1\beta_2)  \equiv 0 \, (q/(q, M)) } 1,
  && \tilde N_2(\mathbf{c}, d; q) =
     \frac{1}{d^6q^7}\sum_{(\beta_1, \beta_2)\in V_{dq} \\ \det(\mathbf{c}, \ensuremath{\boldsymbol\ell}(\beta_1\beta_2)) \equiv 0 \, (dq/(q, M)) \\ \ensuremath{\boldsymbol\ell}(\beta_1\beta_2) \equiv 0 \, (d)} 1.
\end{align}
Define the associated Dirichlet series
\begin{align}
  \mathcal{D}_1(s) = \sum_{q\ge 1 \\ M | q } \frac{\tilde N_1(q)}{q^s}
  , && \mathcal{D}_2(s; \mathbf{c}, d) = \sum_{q\ge 1 \\ M | q } \frac{\tilde N_2(\mathbf{c}, d; q)}{q^s}.
\end{align}
We also write $\mathcal{D}_1^*(s) = \mathcal{D}_1(s)/\zeta(s)$, $\mathcal{D}_2^*(s; \mathbf{c}, d) = \mathcal{D}_2(s; \mathbf{c}, d)/\zeta(s)$ so that 
\begin{equation}\label{align:crb7s6n8bp}
  \mathcal{D}_1^*(s) = \sum_{q\ge 1 \\ M | q } \frac{N_1^*(q)}{q^s},\quad \text{ where } \quad  N_1^*(q) = \sum_{b | q/M } \mu(b) \tilde N_1(q/b)
\end{equation}
and
\begin{equation}
  \mathcal{D}_2^*(s; \mathbf{c}, d) = \sum_{q\ge 1 \\ M | q } \frac{N_2^*(\mathbf{c}, d;q)}{q^s},
  \quad \text{ where}\quad  N_2^*(\mathbf{c}, d ;q) = \sum_{b | q/M }
  \mu(b) \tilde N_2(\mathbf{c}, d; q/b).
\end{equation}

The absolute convergence of $\mathcal{D}_1(s)$, $\mathcal{D}_2(s; \mathbf{c}, d)$ for $\Re(s) > 1$ is clear. 
Further analytic properties of these two relevant to us are contained in the following
two propositions.

\begin{proposition}\label{lemma:cran7rw9t5}
  We have that 
  \begin{equation}\label{eq:crao4ulrs0}
    N_1^*(q) \ll \frac{M^4}{q^{2 - o(1)}}.
  \end{equation}
  In particular, we have that $\mathcal{D}_1^*(s)$ is absolutely convergent for $\Re(s) > -1$ so
  that $\mathcal{D}_1(s)$ has a meromorphic continuation to $\Re(s) > -1$ with
  \begin{equation}\label{eq:crb6w3pal3}
    \mathcal{D}_1^*(0) = M^2\prod_p \lim_{k\to\infty } \frac{1}{p^{6k}}
    \sum_{\beta_1, \beta_2 \, (p^k) \\ \ensuremath{\boldsymbol\ell}(\beta_1\beta_2) \equiv 0 \, (p^k) \\ \beta_i \equiv \beta_i' \, (p^{m_p})} 1 = M^2\prod_p \sigma_p.
  \end{equation}
  Furthermore, for $\sigma > \sigma_0 > -1$, we have that $|\mathcal{D}_1^*(s)| \ll_{\sigma_0} 1$.
\end{proposition}

In the computation of the local densities, we will encounter the following p-adic integral, 
\begin{align}\label{sigmapDef}
  \sigma_p(\mathbf{c}, d) =
  &\  \int_{\mathcal{O}_{K, p}^2} \mathbbm{1}_{ \substack{(\beta_1, \beta_2)\in V_{p^k}
  \\ d | \ensuremath{\boldsymbol\ell}(\beta_1\beta_2) }}
  \delta(\det(\mathbf{c}, \ensuremath{\boldsymbol\ell}(\beta_1\beta_2)))
  \, d \beta_1 \, d \beta_2\nonumber\\
  =&\ \lim_{k\rightarrow\infty} \frac{1}{p^{7k}} \sum_{\beta_1, \beta_2 \in V_{p^k}\\ p^k| \det(\mathbf{c}, \ensuremath{\boldsymbol\ell}(\beta_1\beta_2))\\ p^{d_p}| \ensuremath{\boldsymbol\ell}(\beta_1\beta_2) }1.
\end{align}
Here $d_p = v_p(d)$.

\begin{proposition}\label{prop:D_2_props}
  For $\mathbf{c}$ primitive and $d\ge 1$, $\mathcal{D}_2^*(s; \mathbf{c}, d)$ has an analytic continuation to $\Re(s) > -2$.
  Furthermore, for $\delta > 0$, all $\Re(s) > -2 + \delta$ satisfy the bound
  \begin{equation}\label{eq:cuhe1u1b0j}
    \mathcal{D}_2^*(s; \mathbf{c}, d) \ll_\delta (|\mathbf{c}| d)^{o(1)} M^4(c_1^4 + c_2^4, d).
  \end{equation}
  Furthermore, we have that
  \begin{equation}\label{eq:cuhe1zqnz0}
    \mathcal{D}_2^*(0; \mathbf{c} ,d) = M\prod_p \sigma_p(\mathbf{c}, d).
  \end{equation}

\end{proposition}

\begin{lemma}\label{lemma:crb7r6zwci}
  We have 
  \begin{equation}\label{eq:crb7r7g0xm}
    \sigma_p(\mathbf{c}, d) = \tau_p(\mathbf{c} d),
  \end{equation}
  where
  \begin{equation}
    \tau_p(\mathbf{v}) :=  \frac{1}{(v_1,v_2,p^\infty)}\sum_{k\ge 0} S_p(\mathbf{v};k)
  \end{equation}
  for $\mathbf{v} = (v_1, v_2)\in \mathbb{Z}^2\setminus\{(0,0)\}$, and 
  \begin{equation}\label{SDef}
    S_p(\mathbf{v};k) := \frac{1}{p^{9k+8\max\{0,m_p-k\}}} \sum_{\beta_1, \beta_2 \in V_{p^k}}\sum_{w\, (p^k)}
    \sum_{\mathbf{a}\, (p^k) \\ (a_1, a_2, p) = 1 } e_{p^k} (\langle \ensuremath{\boldsymbol\ell}(\beta_1\beta_2) - \mathbf{v} w,  \mathbf{a} \rangle)
  \end{equation}
  satisfies
  \begin{equation}\label{eq:crb7r7lmsp}
    S_p(\mathbf{v};k)\ll (k + 1) (v_1^4 + v_2^4, p^k)p^{-3k}.
  \end{equation}
  
\end{lemma}

\begin{proof}[Proof of Proposition \ref{lemma:cran7rw9t5}]
  The absolute convergence follows immediately from the bound \eqref{eq:crao4ulrs0}, which itself
  follows from Proposition \ref{prop.N1}.

  Now, for the product of local factors, note that $\mathcal{D}_1^*(0)$, due to the absolute
  convergence implied by \eqref{eq:crao4ulrs0}, satisfies
  \begin{align}
    \mathcal{D}_1^*(0) = \prod_p \biggl( \sum_{k\ge v_p(M) } N_1^*(p^k) \biggr)
    &= \prod_p \lim_{k\to\infty} \tilde N_1(p^k)
    = \prod_p \lim_{k\to\infty}\frac{1}{p^{6k}}\sum_{ \beta_1, \beta_2(p^k) \\ \ensuremath{\boldsymbol\ell}(\beta_1\beta_2) \equiv 0(p^{k - v_p(M)})} 1 \\ 
    &= M^2\prod_p \lim_{k\to\infty}\frac{1}{p^{6k}}\sum_{ \beta_1, \beta_2(p^k) \\ \ensuremath{\boldsymbol\ell}(\beta_1\beta_2) \equiv 0(p^{k})} 1
    = M^2\prod_p \sigma_p,
  \end{align}
  as desired.
\end{proof}

\begin{proof}[Proof of Proposition \ref{prop:D_2_props}]
  All of these straightforwardly follow from the bounds of Proposition \ref{prop:S20}, for
  \eqref{eq:cuhe1uxr7w} implies the absolute convergence of $\mathcal{D}_2^*(s; \mathbf{c}, d)$ for
  $\Re(s) > -2$. The bound \eqref{eq:cuhe1u1b0j} also follows from \eqref{eq:cuhe1uxr7w}.

  Finally, \eqref{eq:cuhe1zqnz0}, that $\mathcal{D}_2^*(0; \mathbf{c}, d)$  is a product of local densities, follows identically
  to the proof of \eqref{eq:crb6w3pal3}.
\end{proof}

\begin{proof}[Proof of Lemma \ref{lemma:crb7r6zwci}]

  What follows can be naturally written in terms of $p$-adic integrals, though we shall
  proceed equivalently by taking limits of counts in $V_{p^k}$.
  
  By the limit characterization of $\sigma_p(\mathbf{c}, d)$ in \eqref{sigmapDef}, we can
  perform a shift of $k$ to $d_p+k$ to write 
  \begin{equation}
    \sigma_p(\mathbf{c}, d) = \lim_{k\rightarrow\infty} \frac{1}{p^{7(d_p+k)}} \sum_{\beta_1, \beta_2 \in V_{p^{d_p+k}}\\ p^{d_p+k}| \det(\mathbf{c}, \ensuremath{\boldsymbol\ell}(\beta_1\beta_2))\\ p^{d_p}| \ensuremath{\boldsymbol\ell}(\beta_1\beta_2) }1.
  \end{equation}
  Taking $d \mathbf{u}  = \ensuremath{\boldsymbol\ell}(\beta_1\beta_2)$, we have
  \begin{equation}
    \sigma_p(\mathbf{c}, d) = \lim_{k\rightarrow\infty} \frac{1}{p^{7(d_p+k)}} \mathop{\sum_{\beta_1, \beta_2 \in V_{p^{d_p+k}}}\sum_{\mathbf{u}\, (p^k)}}_{\substack{\ensuremath{\boldsymbol\ell}(\beta_1\beta_2)\equiv d\mathbf{u}\, (p^{d_p+k})\\ c_1u_2\equiv c_2u_1\, (p^k)}}1.
  \end{equation}
  Take integers $a, b$ such that $a c_1 + b c_2 = 1$, which exist since $(c_1,c_2)=1$.
  Performing the change of variables 
  \begin{equation}
    \binom{w_1}{w_2}=\left(\begin{array}{cc}
      a & b \\
      -c_2 & c_1
    \end{array}\right)\binom{u_1}{u_2} \Longleftrightarrow\binom{u_1}{u_2}=\left(\begin{array}{cc}
      c_1 & -b \\
      c_2 & a
    \end{array}\right)\binom{w_1}{w_2}
\end{equation}
  yields that

  \begin{align}
    \sigma_p(\mathbf{c}, d) = &\ \lim_{k\rightarrow\infty} \frac{1}{p^{7(d_p+k)}} \mathop{\sum_{\beta_1, \beta_2 \in V_{p^{d_p+k}}}\sum_{w_1,w_2\, (p^k)}}_{\substack{\langle \beta_1\beta_2,1\rangle\equiv d(c_1w_1-bw_2)\, (p^{d_p+k})\\ \langle \beta_1\beta_2,\zeta\rangle\equiv d(c_2w_1+aw_2)\, (p^{d_p+k})\\ w_2\equiv 0\, (p^k)}} 1\nonumber\\
    =&\ \lim_{k\rightarrow\infty} \frac{1}{p^{7(d_p+k)}} \mathop{\sum_{\beta_1, \beta_2 \in V_{p^{d_p+k}}}\sum_{w\, (p^k)}}_{\ensuremath{\boldsymbol\ell}(\beta_1\beta_2)\equiv \mathbf{c}dw \, (p^{d_p+k})} 1.
  \end{align}
  Reindexing $w$, we conclude
  \begin{equation}
    \sigma_p(\mathbf{c}, d) =\lim_{k\rightarrow\infty} \frac{1}{p^{8d_p+7k}} \mathop{\sum_{\beta_1, \beta_2 \in V_{p^{d_p+k}}}\sum_{w\, (p^{d_p+k})}}_{\ensuremath{\boldsymbol\ell}(\beta_1\beta_2)\equiv \mathbf{c}dw \, (p^{d_p+k})} 1=\lim_{k\rightarrow\infty} \frac{1}{p^{d_p+7k}} \mathop{\sum_{\beta_1, \beta_2 \in V_{p^{k}}}\sum_{w\, (p^{k})}}_{\ensuremath{\boldsymbol\ell}(\beta_1\beta_2)\equiv \mathbf{c}dw \, (p^{k})} 1.
  \end{equation}

  To complete the proof of \eqref{eq:crb7r7g0xm}, we use the telescopic nature of $S(\mathbf{c}d;p^k)$.
  Precisely, for $k\geq 1$, we have 
  \begin{align}
    S_p(\mathbf{c}d; k ) = &\ \sum_{j=0,1}\frac{(-1)^j}{p^{9k+8\max\{0,m_p-k\}}} \sum_{\beta_1, \beta_2 \in V_{p^k}}\sum_{w\, (p^k)}
                             \sum_{\mathbf{a} \, (p^{k-j})} e_{p^{k-j}} (\langle \ensuremath{\boldsymbol\ell}(\beta_1\beta_2) - \mathbf{c}d w,  \mathbf{a} \rangle)\nonumber\\
    =&\ \sum_{j=0,1}\frac{(-1)^j}{p^{9k-9j+8(\max\{0,m_p-k\}+j\charf{k\leq m_p})}} \sum_{\beta_1, \beta_2 \in V_{p^{k-j}}}\sum_{w\, (p^{k-j})}\nonumber\\
                           &\ \quad \quad\sum_{\mathbf{a} \, (p^{k-j})} e_{p^{k-j}} (\langle \ensuremath{\boldsymbol\ell}(\beta_1\beta_2) - \mathbf{c}d w,  \mathbf{a} \rangle).
  \end{align}
  With
  \begin{equation}
    p^{9k+8\max\{0,m_p-k\}}=p^{9k+9-9+8(\max\{0,m_p-k-1\}+\charf{k+1\leq m_p})},
  \end{equation}
  the $j=0$ term in $S_p(\mathbf{c}d;k+1)$ cancels with the $j=1$ term in $S_p(\mathbf{c}d;k)$, so that
  the telescopic sum leads to
  \begin{align}
    \sum_{k\ge 0} S_p(\mathbf{c}d; k) =&\ \lim_{k\to\infty} \frac{1}{p^{9k}} \sum_{\beta_1, \beta_2 \in V_{p^k}}\sum_{w\, (p^k)}\sum_{\mathbf{a} \, (p^k)} e_{p^k} (\langle \ensuremath{\boldsymbol\ell}(\beta_1\beta_2) - \mathbf{c}d w,  \mathbf{a} \rangle)\nonumber\\
    =&\ \lim_{k\to\infty} \frac{1}{p^{7k}} \mathop{\sum_{\beta_1, \beta_2 \in V_{p^k}}\sum_{w\, (p^k)}}_{\ensuremath{\boldsymbol\ell}(\beta_1\beta_2)\equiv \mathbf{c}dw \, (p^k)} 1 = p^{d_p}\sigma_p(\mathbf{c}, d),
  \end{align}
  so we obtain \eqref{eq:crb7r7g0xm}.

  For the bound \eqref{eq:crb7r7lmsp}, write $p^h=(v_1^4+v_2^4,p^k)$ and write $\eta=\min\{m_p,k\}$.
  Recalling that $\ensuremath{\boldsymbol\ell}(\beta)= (\langle \beta,1\rangle, \langle \beta,\zeta\rangle)$ and detecting $\beta_2\equiv \beta_2'$ with additive
  characters, we have
  \begin{multline}
    S_p(\mathbf{v};k) = \frac{1}{p^{9k+4\eta}}\sum_{\gamma\in \mathcal{O}_K/ p^{\eta}\mathcal{O}_K} \sum_{\beta_1, \beta_2 \in \mathcal{O}_K/p^{k}\mathcal{O}_K\\ \beta_1\equiv \beta_1'\, (p^{\eta})}\sum_{w\, (p^k)}
    \sum_{a_1, a_2\, (p^k) \\ (a_1, a_2, p) = 1 } \\
    e_{p^k} (\langle \beta_1\beta_2,a_1+a_2\zeta\rangle+p^{k-\eta}\langle\beta_2'-\beta_2,\gamma\rangle -a_1v_1w - a_2v_2w).
  \end{multline} 
  Summing over $\beta_2$ and $w$, we get
  \begin{equation}
    S_p(\mathbf{v};k) = \frac{1}{p^{4(k+\eta)}}\sum_{\gamma\in \mathcal{O}_K/ p^{\eta}\mathcal{O}_K}\psi_{p^{\eta}}(\gamma\beta_2') \mathop{\sum_{\beta\in \mathcal{O}_K/ p^k\mathcal{O}_K \\ \beta\equiv\beta_1'\, (p^{\eta})}
      \sum_{a_1, a_2\, (p^k) \\ (a_1, a_2, p) = 1 }}_{\substack{(a_1+a_2\zeta)\beta\equiv p^{k-\eta}\gamma\, (p^k)\\ a_1v_1\equiv -a_2v_2\, (p^k)}} 1.
  \end{equation}
  With $\gamma$ completely determined by the other variables, we have
  \begin{equation}
    |S_p(\mathbf{v};k)| \leq \frac{1}{p^{4(k+\eta)}}\mathop{\sum_{\beta\in \mathcal{O}_K/ p^k\mathcal{O}_K}
      \sum_{a_1, a_2\, (p^k)}}_{\substack{(a_1, a_2, p) = 1\\ (a_1+a_2\zeta)\beta\equiv 0\, (p^{k-\eta})\\ a_1v_1\equiv -a_2v_2\, (p^k)}} 1\leq \frac{1}{p^{4k}}
    \sum_{a_1, a_2\, (p^k)\\ (a_1, a_2, p) = 1\\ a_1v_1\equiv -a_2v_2\, (p^k)} N((a_1+a_2\zeta,p^{k})).
  \end{equation}
  Write $v_1=p^s v_1'$ and $v_2=p^tv_2'$ with $(v_1'v_2', p)=1$. We separate into four cases.

  \underline{Case 1:} $s, t\geq k$. Then
  \begin{equation}
    |S_p(\mathbf{v};k)| \leq \frac{1}{p^{4k}}\sum_{a_1, a_2\, (p^k)\\ (a_1, a_2, p) = 1} N((a_1+a_2\zeta,p^{k})) \ll p^{-2k}
  \end{equation}
  by Lemma \ref{lem.sumNorm}.

  \underline{Case 2:} $s=t\leq k$. In this case, the congruence condition with $(a_1, a_2, p)=1$ forces $(a_1a_2,p)=1$. Hence we have 
  \begin{align}
    |S_p(\mathbf{v};k)| \leq &\ \frac{1}{p^{4 k}} \sumCp_{a_1, a_2 \, (p^k)} N((a_1+a_2 \zeta, p^k)) \\
    = & \ \frac{1}{p^{4 k}} \sumCp_{a_1 \, (p^k)} \sumCp_{z \, (p^s)} N((a_1+(-v_1' \overline{v_2'} a_1+z p^{k-s}) \zeta, p^k)) \\
    = &\  \frac{1}{p^{4 k}} \sum_{d=0}^s \sumCp_{a_1 \, (p^k)} \sumCp_{z\, (p^d)} N((a_1+(-v_1' \overline{v_2'} a_1+z p^{k-d}) \zeta, p^k)).
  \end{align}
  Lemma \ref{lem.Only1Prime} implies that this is equal to 
  \begin{equation}
    \frac{1}{p^{4 k}} \sum_{d=0}^s \sumCp_{a_1 \, (p^k)} \sumCp_{z\, (p^d)} (a_1^4+(-v_1' \overline{v_2'} a_1+z p^{k-d})^4, p^k).
  \end{equation}
  Applying a change of variable $z\mapsto a_1z$, this is equal to
  \begin{equation}
    \frac{1}{p^{4 k}} \sum_{d=0}^s \varphi(p^d) \sumCp_{a_1 \, (p^k)} (a_1^4+(-v_1' \overline{v_2'} a_1+ p^{k-d})^4, p^k).
  \end{equation}
  Lemma \ref{lem.sumNorm} then gives us the bound
  \begin{equation}
    |S_p(\mathbf{v};k)| \ll (s+1)p^{s+\ell-3k},
  \end{equation}
  with $\ell$ defined by $p^\ell || (v_1'^4+v_2'^4)$.

  \underline{Case 3:} $t<s<k$. In this case, the congruence condition with $(a_1, a_2, p)=1$ forces $(a_1, p)=1$ and $p^{s-t}||a_2$. Hence we have
  \begin{equation}
    |S_p(\mathbf{v};k)| \leq \frac{1}{p^{4k}} \mathop{\sumCp_{a_1 \, (p^k)} \sumCp_{a_2 \, (p^{k-s+t})}}_{a_1v_1'\equiv -a_2v_2' \, (p^{k-s})} N((a_1 + a_2 p^{s-t} \zeta, p^k))\leq p^{t-3k}.
  \end{equation}

  \underline{Case 4:} $s<t<k$. This case is symmetric to Case 3 and we have the bound
  \begin{equation}
    |S_p(\mathbf{v};k)| \leq p^{s-3k}.
  \end{equation}
  Combining all four cases, we obtain the desired bound \eqref{eq:crb7r7lmsp}.
\end{proof}

It will be useful in later sections to write, for $q\geq 1$, 
\begin{equation}\label{SvqDef}
  S(\mathbf{v}; q) = \prod_{p^k || q} S_p(\mathbf{v};k).
\end{equation}

\subsection{Archimedean local densities}
We shall require the following integral computation in \S\ref{sec:crao4uslaz}, and can be thought of
as amounting roughly to an archimedean analogue of some of the content of Lemma \ref{lemma:crb7r6zwci}.
\begin{lemma}\label{lemma:crb8be0e8e}
  For $\omega\in C_c^{\infty}(\mathbb{R}_{> 0})$, $\Phi\in C_c^\infty(K_\infty^2\setminus\set{(0, 0)})$, we have that
  \begin{multline}
    \int_{\mathbb{R}^2} \omega(|\mathbf{v}|)
    \int_{K_\infty^2} \Phi(x_1^\infty, x_2^\infty) \delta(\langle \mathbf{v}, \ensuremath{\boldsymbol\ell}(x_1^\infty x_2^\infty) \rangle)
    \,d x_1^\infty \,d x_2^\infty \, d \mathbf{v} \\ 
    = 2\tilde\omega(1) \int_{K_\infty^2} \Phi(x_1^\infty, x_2^\infty) |\ensuremath{\boldsymbol\ell}(x_1^\infty x_2^\infty)|^{-1}
    \,d x_1^\infty  \,d x_2^\infty.
  \end{multline}
\end{lemma}
\begin{proof}
  It will be sufficient to show that for any  $\mathbf{w}\in \mathbb{R}^2\setminus\set{(0, 0)}$, we have
  \begin{equation}\label{eq:crb7sc2j4j}
    \int_{\mathbb{R}^2} \omega(|\mathbf{v}|)\delta(\langle \mathbf{v}, \mathbf{w} \rangle) \, d \mathbf{v}
    = 2\tilde\omega(1) |\mathbf{w}|^{-1}.
  \end{equation}
  First, note that this is invariant under the action of $\SO(2)$, so we may suppose
  without loss of generality that $\mathbf{w} = (|w|, 0)$. Letting $u = (u_1, u_2)$, we obtain that
  the LHS of \eqref{eq:crb7sc2j4j} is
  \begin{equation}
    |\mathbf{w}|^{-1}\int_{\mathbb{R}^2} \omega(|u|) \delta(u_1)  \,d u = |w|^{-1} \int_{\mathbb{R}} \omega(u_2) \, d u_2
    = 2\tilde\omega(1)|w|^{-1},
  \end{equation}
  as desired.
\end{proof}
At a couple points, in our computation of the zero frequencies in
\S\ref{sec:crb0xrxfh5}, \S\ref{sec:crao4uslaz}, we will encounter the following
Riesz-type integral. 
\begin{lemma}\label{lemma:crb0xx4e28}
  Suppose that $w\in C_c^{\infty}(\mathbb{R}^n)$ and that $F : \mathbb{R}^n\to \mathbb{R}^m$ is such that $0$ is a regular value of
  $F|_{\mathrm{supp}\,w}$\footnote{For any $a$ in $\mathrm{supp}(w)$ such that $F(a)=0$, the matrix $(\frac{\partial}{\partial x_i}F_j(x)|_{x=a})$ has full rank, here $F|_{\mathrm{supp}\,w}=(F_j)_{j\le m}$.}. Then, we have that the function $U(s)$, defined for $\Re(s) > -m$ as
  \begin{equation}
    U(s) = \int_{\mathbb{R}^n} w(x) |F(x)|^{s} \,d x,
  \end{equation}
  has a meromorphic continuation to $\mathbb{C}$ with simple poles at $-m - 2k$ for $k\ge 0$.
  Furthermore, we have that
  \begin{equation}
    \underset{s = -m - 2k}{\res} U(s) =
    \frac{2\pi^{m / 2}}{4^k k!\Gamma(m/2 + k)} \int_{\mathbb{R}^n} \Delta_F^k w(x) \delta(F(x)) \,d x,
  \end{equation}
  where $\Delta_F$ satisfies $\Delta_F(\phi\circ F) = (\Delta_{\mathbb{R}^m}\phi)\circ F$.
  Furthermore, for any $k\ge 0$, $-m - 2k - 2 < \sigma_1 < \sigma_2 < -m - 2k$, and
  $\sup_{|x|\le 1} |F(x)|\le 1$, we have that $s$ with $\Re(s)\in [\sigma_1, \sigma_2]$ satisfy
  \begin{equation}
    U(s) \ll_{\sigma_1, \sigma_2} 1. 
  \end{equation}
\end{lemma}
\begin{proof}

  If $m > n$, the regularity condition implies that $F$ is nonzero on $\mathrm{supp}\,w$ from which
  the result follows trivially, so suppose $m\le n$ from now on.
  
  Let $K = \mathrm{supp}\, w$, and let $Z = K\cap F^{-1}(0)$. Because $0$ is a regular value of $F$ on $K$,
  there exists a neighborhood $\Omega$ of $Z$ on which $ d F$ is surjective at every point.

  By compactness, there also exists a neighborhood $0\in V \subset \mathbb{R}^m$ such that
  $K\cap F^{-1}(V)\subset \Omega$.

  Now, take $\chi\in C_c^{\infty} (\mathbb{R}^m)$ such that for some $\varepsilon > 0$, $\mathbbm{1}_{ |\cdot| < \varepsilon }\le \chi\le \mathbbm{1}_{ V }$, and take
  $w_0 = w\cdot (\chi\circ F)$, $ w_1 = w - w_0$,
  \begin{equation}
    U_j(s) = \int_{\mathbb{R}^n} w_j(x) |F(x)|^s \,d x .
  \end{equation}
  
  We have that $|F|\gg 1$ on $\mathrm{supp}\, w_1$, so $U_1$ is entire.

  Now, since $d F$ is surjective on $\mathrm{supp}\, w_0$, we have that
  \begin{equation}
    U_0(s) = \int_{\mathbb{R}^n} w_0(x) |F(x)|^s  \,d x  = \int_{\mathbb{R}^m} |y|^s g(y) \,d y .
  \end{equation}
  where $g(y) \,d y  = F_*(w_0(x) \,d x )$ (with $y\in \mathbb{R}^m$, $x\in \mathbb{R}^n$).

  To proceed, note that
  \begin{equation}\label{eq:crb5jmq3ny}
    U_0(s) = \int_{0}^\infty r^{s + m - 1} A(r) \, d r  
  \end{equation}
  where
  \begin{equation}
    A(r) = \int_{|y| = 1} g(yr) \,d y.
  \end{equation}
  We record that the surface area of the $(m - 1)$-sphere equals
  \begin{equation}
    \int_{|y| = 1} 1 \,d y  = \frac{2\pi^{m / 2}}{\Gamma(m/ 2)}. 
  \end{equation}
  Now, Taylor expanding the radial $A(r)$, we have that 
  \begin{equation}
    A(r) = \sum_{k\le N} \frac{2\pi^{m/2}}{4^kk!\Gamma(m/2 + k)} \Delta^kg(0)r^{2k} + O(r^{2N + 2}).
  \end{equation}
  The desired result follows for $\sigma > -m - 2N$ upon executing the integral \eqref{eq:crb5jmq3ny}, noting
  that
  \begin{equation*}
    \Delta^kg(0) = \int_{\mathbb{R}^n} \Delta_F^kw(x) \delta(F(x)) \,d x
  \end{equation*}
  by repeated integration by parts.
\end{proof}

\section{Estimation of $\Sigma_1$}\label{sec:cq6r430mfy}

In this section, we carry out the estimation of $\Sigma_1^S$ in \eqref{eq:crb5jnfhh8} for all $S\subseteq \set{1, 2}$.
The technical core of this paper is contained in the estimation of the nonzero frequencies
$\Sigma_1^{\set{1,2}}$.

\subsection{Estimation of $\Sigma_1^{\{1,2\}}$: nonzero frequencies}\label{sec:cuhge2k5o6}

In this section, we carry out the technical heart of our proof to gain a saving from the
sum over $q$ in $\Sigma_1^{\set{1,2}}$.

Our main result this subsection is as follows.
\begin{proposition}\label{proposition:cq7567d3xk}
  We have the bound
  \begin{equation}
    \Sigma_1^{\set{1,2}}\ll M^{12}L^{47/8} X^{2 - 1/48 + o(1)}
  \end{equation}
\end{proposition}

We shall rely on the contents of \S\ref{sec:cuhk0c94z7}, and mainly concerned with extracting savings
from $\alpha = \alpha_1\alpha_2$ with $\Gal(f_\alpha) = S_3 $ (recall that $f_\alpha $ was defined in \S\ref{sec:cq72jm8eaf}).
For any $\alpha $ with $\Gal(f_\alpha) $, we let $\pi_\alpha $ be the cuspidal automorphic representation of
Lemma \ref{proposition:cuhklr5wqk} (when $f = f_\alpha $) and let $(\lambda_\alpha(n))_{n\ge 1} $ be the Dirichlet coefficients of $L(s, \pi_\alpha) $.

We will require a bound to show that $\Gal(f_\alpha) = S_3 $ happens for most $\alpha $, which is achieved
through the following bound on the number of $\alpha $ for which this doesn't hold.
\begin{proposition}\label{proposition:cralvwgndq}
  We have that
  \begin{equation}
    \sum_{\alpha\in \mathcal{O}_K \\ |\alpha|_{\sup} < X } \charf{\Gal(f_\alpha)\neq S_3}
    \ll X^3\log^2 X.
  \end{equation}
\end{proposition}

\begin{proof}

  The case of $\langle \alpha,1\rangle=0 \Leftrightarrow \deg f_\alpha < 3$ can clearly be disposed of, so we may suppose that
  $\deg f_\alpha = 3$ from now on.
  
  We begin by disposing of the case that $|\Gal(f_\alpha)| < 3$, the case of 
  $f_\alpha$ reducible over $\mathbb{Q}$. This contribution we bound by summing over splittings
  of the polynomial.

  In particular, we have that
  \begin{equation}\label{eq:cq7573lpyz}
    \sum_{\alpha\in \mathcal{O}_K \\ |\alpha|_{\sup} < X } \charf{|\Gal(f_\alpha)| < 3}
    = \sum_{x,y, u, v, w\in \mathbb{Z} \\ wy \neq 0 \\  |(x + y\zeta)(u + v\zeta + w\zeta^2)|_{\sup}\ll X} 1
    = \sum_{x,y, u, v, w\in \mathbb{Z} \\ uwxy \neq 0 \\  |(x + y\zeta)(u + v\zeta + w\zeta^2)|_{\sup}\ll X} 1 + O(X^3)
  \end{equation}
  Note that the bound on $|\cdot|_{\sup}$ in the last two sums  forces $ux, wy \ll X $.
  For any fixed $x, y, u, w$, looking at the $\zeta^2$-coefficient implies the bound forces that
  \begin{equation}
    |vy + wx| \ll X.
  \end{equation}
  That $y \neq 0$ implies that the second inequality can hold for at most $O(X)$-many $v\in \mathbb{Z}$.
  Therefore, we have that \eqref{eq:cq7573lpyz} is 
  \begin{equation}
    \ll \sum_{x, y, u, w\ll X \\ uwxy \neq 0 \\ ux, wy \ll X } X + O(X^3)\ll X^3\log^2X.
  \end{equation}

  It remains therefore to bound the size of
  \begin{equation}
    \sum_{\alpha\in \mathcal{O}_K \\ |\alpha|_{\sup} < X }\charf{\Gal(f_\alpha) = A_3}.
  \end{equation}
  Note that a polynomial with Galois group $A_3$ has square discriminant, so it
  suffices to bound
  \begin{equation}\label{eq:cq7578pbd3}
    \sum_{|n_0|,\dots, |n_3| \ll X \\ n_3 \neq 0 } \sum_{m }
    \charf{n_1^2n_2^2 - 4 n_1^3n_3 - 4n_0n_2^3 - 27n_0^2n_3^2 + 18n_0n_1n_2n_3 = m^2},
  \end{equation}
  Letting $A = -27n_0^2, B = 18n_0n_1n_2 - 4n_1^3, C = n_1^2n_2^2 - 4n_0n_2^3$. The number of
  $(n_i)_{0\le i\le 2}$ with $ABC = 0$ can be checked to be $\ll X^2$, so we'll suppose from now
  on that $ABC \neq 0$.
  The equation in the indicator function then may be rewritten as
  \begin{equation}
    m^2 = An_3^2 + B n_3 + C
  \end{equation}
  so completing the square yields that
  \begin{equation}
    (2An_3 + B)^2 - 4A m^2 = B^2 - 4AC.
  \end{equation}
  Letting $U = 2An_3 + B, V = 2m$, this is equivalent to
  \begin{equation}\label{eq:cq7578rdes}
    U^2 + 3 (3n_0 V)^2 = B^2- 4AC.
  \end{equation}
  The number of $n_0, n_1, n_2$ with $B^2 - 4AC = 0$ is at most $O(X^2)$. If
  $B^2 - 4AC \neq 0$, the number of $n_3$ satisfying \eqref{eq:cq7578rdes} is at most
  $d(|B^2 - 4AC|)$. It follows that \eqref{eq:cq7578pbd3} is at most
  \begin{equation}
    \sum_{|n_0|, |n_1|, |n_2|\ll X\\ B^2 - 4AC \neq 0 } d(|B^2 - 4AC|) \ll X^3\log X,
  \end{equation}
  from which the desired result follows.

\end{proof}

\begin{proposition} \label{prop:moment_rs_bound}
  We have that for any $Y > 0, B\ge 1 $ that
  \begin{equation}
    \sum_{\alpha\sim Y \\ h | N(\alpha) \\ \Gal(f_\alpha) = S_3 }
    \biggl|\sumSf_{\substack{q\sim Q \\  (q, B N(\alpha)) = 1}} \lambda_\alpha(q) \biggr|\ll
    \biggl( \frac{Y^4}{h} Q^{3/4}Y^{3/8} + \frac{Y^{63/16}}{h^{15/16}} Q \biggr)(B QY)^{o(1)}.
  \end{equation}
\end{proposition}
\begin{proof}
  We shall omit the restriction $\Gal(f_\alpha) = S_3 $ throughout this proof; it should be assumed
  to hold.

  By H\"older's inequality, we have that
  \begin{equation}\label{eq:cuhknerlrx}
    \sum_{\alpha\sim Y \\ h | N(\alpha)} \biggl|
    \sumSf_{\substack{q\sim Q \\  (q, B N(\alpha)) = 1}} \lambda_\alpha(q) \biggr|
    \ll h^{o(1)}\biggl(\frac{Y^4}{h}\biggr)^{7/8}
    \biggl( \sum_{\alpha\sim Y \\ h | N(\alpha) }
    \biggl( \sumSf_{\substack{q\sim Q \\  (q, B N(\alpha)) = 1}}  \lambda_\alpha(q) \biggr)^8 \biggr)^{1/8}.
  \end{equation}
  Expanding and applying repeatedly the Hecke relation
  \begin{equation}
    \lambda_\alpha(q_1)\lambda_\alpha(q_2) = \sum_{d | q_1, q_2 } \omega_{\pi_\alpha}(d) \lambda_\alpha \biggl(\frac{q_1q_2}{d^2}\biggr),
  \end{equation}
  we have that
  \begin{equation}\label{eq:cuhknerfn0}
     \biggl( \sumSf_{\substack{q\sim Q \\  (q, B N(\alpha)) = 1}} \lambda_\alpha(q) \biggr)^8
     = \sum_{d\ge 1 } \omega_{\pi_\alpha}(d) \sumSf_{\substack{ q\ll Q^8 \\ (q, B N(\alpha)) = 1 }} \lambda_\alpha(q) c(q, d)
   \end{equation}
   where
   \begin{equation}
     c(q, d) = \sum_{d_1,\dots, d_{7} \\ q_1, \dots, q_8 \\ d_1\dots d_7 = d \\ q_1\dots q_8 = qd^2 }
     \mathbbm{1}_{ \substack{d_j | q_1\dots q_j/(d_1\dots d_{j - 1})^2, q_{j + 1}  \forall j < 8 \\ q_1, \dots q_8 \sim Q} }
   \end{equation}
   is clearly supported on $q\ll Q^8 $, $d\asymp \sqrt{Q^8/q} $ and satisfies $c(q, d)\ll (dq)^{o(1)} $.
   
   We'll split $q $ into dyadic intervals and consider for now just the contribution of $q\sim Q^* $
   for some scale $Q^*\ll Q^8 $. Then, we have that
   \begin{equation}\label{eq:cuhkneq033}
     \sum_{\alpha\sim Y \\ h | N(\alpha)} \biggl( \sumSf_{\substack{q\sim Q \\ (q, B N(\alpha)) = 1 } }
     \lambda_\alpha(q) \biggr)^8
     \ll Q^{o(1)}\sup_{Q^* \ll Q^8}\biggl| \sum_{ d\asymp Q^4/\sqrt{Q^*}   }
     \sumSf_{ \substack{q\sim Q^* \\ (q, B) = 1}} c(q, d) \sum_{\alpha \sim Y \\ h | N(\alpha) \\ (q, N(\alpha)) = 1} \omega_{\pi_\alpha}(d)\lambda_\alpha(q) \biggr|.
   \end{equation}
   We record the trivial bound
   \begin{equation}
     \biggl| \sum_{ d\asymp Q^4/\sqrt{Q^*}   }
     \sumSf_{\substack{ q\sim Q^* \\ (q, B) = 1}} c(q, d) \sum_{\alpha \sim Y \\ h | N(\alpha) \\ (q, N(\alpha)) = 1} \omega_{\pi_\alpha}(d)\lambda_\alpha(q) \biggr|
     \ll (h Q)^{o(1)} \frac{Q^8 Y^4}{h} \sqrt{\frac{Q^*}{Q^8}},
   \end{equation}
   and that we need to beat the trivial bound when $Q^* = Q^8 $ with a power saving.

   Applying the Cauchy-Schwarz inequality, inserting by positivity a smoothing
   $W\in C_c^\infty(\mathbb{R}_{ > 0}) $ with $\mathbbm{1}_{ [1,2 ] } \le W\le \mathbbm{1}_{ [1/2, 3/2] } $, we have that
   \begin{multline}\label{multline:cuhknepnfu}
     \biggl| \sum_{d\asymp Q^4/\sqrt{Q^*}} \sumSf_{\substack{ q\sim Q^* \\ (q, B) = 1}}
     c(q, d) \sum_{\alpha \sim Y \\ h | N(\alpha) \\ (q,N(\alpha))=1} \omega_{\pi_\alpha}(d)\lambda_\alpha(q) \biggr|\\
     \ll Q^{o(1)}(Q^4\sqrt{Q^*})^{1/2}
     \biggl( \sum_{d\asymp Q^4/\sqrt{Q^*} } \sum_{\alpha, \alpha' \sim Y \\ h | N(\alpha), N(\alpha')  }\omega_{\pi_\alpha}(d)\omega_{\pi_{\alpha'}}(d)
     \sumSf_{\substack{(q, B N(\alpha\alpha')) = 1}} \lambda_{\alpha}(q)\lambda_{\alpha'}(q)W \biggl(\frac{q}{Q^*}\biggr)\biggr)^{1/2}.
   \end{multline}
   Applying Proposition \ref{proposition:cuhkndxg2y} when $k_\alpha \neq k_{\alpha'} $ and applying a trivial
   bound otherwise, noting that by Lemma \ref{proposition:cuhklr5wqk}, the levels of
   $\pi_\alpha, \pi_{\alpha'} $ are $\ll Y^6 $, we obtain that
   \begin{multline}\label{multline:cuhknektzw}
      \sum_{d\asymp Q^4/\sqrt{Q^*} } \sum_{\alpha, \alpha' \sim Y \\ h | N(\alpha), N(\alpha')  }\omega_{\pi_\alpha}(d)\omega_{\pi_{\alpha'}}(d)
      \sumSf_{\substack{(q, B N(\alpha\alpha')) = 1}} \lambda_{\alpha}(q)\lambda_{\alpha'}(q)W \biggl(\frac{q}{Q^*}\biggr) \\
      \ll (QY)^{o(1)} Q^4 Y^6 \frac{Y^8}{h^2}
      + Q^{4 + o(1)}\sqrt{Q^*}\sum_{\alpha, \alpha'\sim Y \\ h | N(\alpha), N(\alpha') } \mathbbm{1}_{ k_\alpha\simeq k_{\alpha'} } .
    \end{multline}
    Now, note that by varying the constant term of $\alpha' $ and dropping the restriction
    $h | N(\alpha') $, we have that for every $\alpha $, there are at most $O(Y^{3 + o(1)}) $-many $\alpha'\sim Y$ for
    which $k_{\alpha}\simeq k_{\alpha'} $.
    Therefore, it follows that \eqref{multline:cuhknektzw} is 
    \begin{equation}
      \ll (QY)^{o(1)} \biggl( \frac{Y^8}{h^2} Q^{4} Y^6 + Q^8 \frac{Y^7}{h}\sqrt{\frac{Q^*}{Q^8}} \biggr).
    \end{equation}
    Combining with \eqref{multline:cuhknepnfu}, \eqref{eq:cuhkneq033}, \eqref{eq:cuhknerfn0}, and \eqref{eq:cuhknerlrx}, the desired result follows.
\end{proof}

We are now ready to prove Proposition \ref{proposition:cq7567d3xk}.

\begin{proof}[Proof of Proposition~\ref{proposition:cq7567d3xk}]
  Recall the definition of $\Sigma_1^{\set{1,2}}$ from \eqref{eq:crb5jnfhh8} that
  \begin{equation}
    \Sigma_{1}^{\{1,2\}} = \frac{X^4}{D^2} \sum_{q\ge 1 \\ M | q } \frac{1}{q^5}
    \sum_{\alpha_1, \alpha_2 \neq0} (S_1I_1)(\alpha_1, \alpha_2; q).
  \end{equation}
  By the definition of $I_1$ in \eqref{eq:crb5jnfhh8} (support of $\omega_1$) and Proposition \ref{proposition:crale1hrji}, we can
  impose the restrictions
  \begin{equation}
    q\ll \frac{MX}{D} \quad \text{ and } \quad |\alpha_1\alpha_2|_{\sup} \ll \frac{M^2}{D^2}X^{1+o(1)}
  \end{equation}
  at the cost of a negligible error $O_A(X^{-A})$. Performing dyadic subdivision on $q$ and
  $|\alpha_1\alpha_2|_{\sup}$ at the cost of multiplying the bound by $X^{o(1)}$, it suffices to show that
  for any $Q\ll \frac{MX}{D}$ and $Y\ll M^2X^{1+o(1)}/D^2$, the bound of Proposition \ref{proposition:cq7567d3xk}
  holds for 
  \begin{equation}
    \Sigma_1^{\set{1, 2}}(Q, Y) := \frac{X^4}{D^2}\sum_{q\sim Q \\ M | q} \frac{1}{q^5}
    \sum_{|\alpha_1\alpha_2|_{\sup} \sim Y} (S_1I_1)(\alpha_1, \alpha_2; q).
  \end{equation}
  For convenience at this point, we bring in the scale $T = M X/(QD) $, which was first
  introduced in \S\ref{sec:oscillatory}. We also let $U = M^2X/(D^2Y) $, so that $\tilde\alpha_1, $ $\tilde\alpha_2 $ defined in
  Proposition \ref{proposition:crale1hrji} satisfy $|\tilde\alpha_1\tilde\alpha_2|_{\sup}U \asymp |\alpha_1\alpha_2|_{\sup}/Y\asymp 1 $. The purpose of introducing these
  parameters is to give a transparent indication of where we are relative to the ``generic''
  case $T, U\asymp 1 $.
  
  Let $ \Sigma_{1,\mathrm{gen}}^{\set{1, 2}} $ be the contribution of $\alpha_1, \alpha_2 $ with $\Gal(f_{\alpha_1\alpha_2}) = S_3 $ and let $\Sigma_{1, \mathrm{exc}}^{\set{1,2}} $ be that
  of $\Gal(f_{\alpha_1\alpha_2}) \neq S_3 $. We shall first use Proposition \ref{proposition:cralvwgndq} to bound $\Sigma_{1, \mathrm{exc}}^{\set{1, 2}} $.
  
  By Propositions \ref{proposition:cq6y0pcoy5} and \ref{proposition:crale1hrji}, noting that $\frac{X^4}{D^2Q^5}\asymp T^{5} D^3/(M^5X) $ and pulling out the biggest rational divisor $g$ from $\alpha=\alpha_1\alpha_2$, we have that
  \begin{multline}
    \Sigma_{1, \mathrm{exc}}^{\set{1, 2}}(Q, Y) \ll  \frac{D^3M^3}{X^{1-o(1)}} U^{5/6}T^{2/3} \sum_{g\ll Y } \sum_{\alpha\sim Y/g \\ \Gal(f_\alpha) \neq S_3 \\ \not\exists \, p | \alpha  }
    \sum_{ q\sim Q } (g, q)^3 (q, \Delta_\alpha)^{2/3} \sum_{t | q/(g, q) } t \mathbbm{1}_{ t^2 | N(\alpha) } \\ 
    \ll \frac{D^3M^3}{X^{1-o(1)}} U^{5/6}T^{2/3} 
    \sum_{st \ll Q } s^3t \sum_{g\ll Y \\ s | g } \frac{Q}{(st)^{1/3}}
    \sum_{\alpha\sim Y/g \\ \Gal(f_\alpha)\neq S_3 \\ t^2 | N(\alpha) } 1.
  \end{multline}
  We have dropped the condition $M|q$ by positivity and used the observation that by H\"older's inequality,
  \begin{equation}
    \sum_{q\sim Q \\ st | q } (q, \Delta_{\alpha})^{2/3}\ll \biggl( \sum_{q\sim Q } (q, \Delta_{\alpha}) \biggr)^{2/3}
    \biggl( \sum_{q\sim Q \\ st | q } 1 \biggr)^{1/3} \ll  \frac{Q X^{o(1)}}{(st)^{1/3}}.
  \end{equation}
  Now, we bound the inner sum over $\alpha $ by either ignoring the $\Gal(f_\alpha) \neq S_3 $ condition,
  or by applying Proposition \ref{proposition:cralvwgndq} ignoring the $t^2 | N(\alpha) $ condition, which yields that
  \begin{multline}\label{multline:cuhli7dnls}
    \Sigma_{1, \mathrm{exc}}^{\set{1, 2}} \ll \frac{D^3M^3}{X^{1-o(1)}} U^{5/6}T^{2/3} Q
    \sum_{st\ll Q } s^{8/3}t^{2/3} \sum_{g\ll Y \\ s | g } \min \biggl( \frac{Y^4}{g^4t^2}, \frac{Y^3}{g^3} \biggr) \\ 
    \ll  \frac{D^3M^3}{X^{1-o(1)}} U^{5/6}T^{2/3}Q \sum_{st\ll Q } s^{8/3}t^{2/3}
    \sum_{g\ll Y \\ s | g } \frac{Y^{23/6}}{g^{23/6}t^{5/3}} \\
    \ll
    \frac{D^3M^3}{X^{1-o(1)}}  U^{5/6} T^{2/3} Y^{23/6}Q 
    \ll \frac{M^{35/3}X^{23/6+o(1)}}{D^{17/3}}.
  \end{multline}

  Now, we shall bound $\Sigma_{1, \mathrm{gen}}^{\set{1,2}} $ using Proposition \ref{prop:moment_rs_bound}, for whose
  application we begin preparations. Write $B_\alpha = M N(\alpha) \Delta_\alpha$. We'll be splitting $q = q_0q_b $
  for $(q_0, q_b) = 1 $ with $q_0 $ squarefree and coprime to $B_\alpha $. In particular, $q_b $ is
  such that $p | q_b\implies p^2 | q_b\lor p | B_\alpha $. We record for now the fact that for any $Q_b\geq1$ and
  positive integer $a$, 
  \begin{equation}\label{eq:cuhlg52mjc}
    \sum_{q_b \sim Q_b \\ p | q_b \implies p^2 | q_b\lor p | B_\alpha \\ a | q_b  } 1
    \ll Q_b^{1/2}a^{-1/2} B_\alpha^{o(1)}.
  \end{equation}
  Now, by Proposition \ref{proposition:crale1hrji} and partial summation (along with a further
  splitting into dyadic intervals), we have that
  \begin{equation}
    \Sigma_{1, \mathrm{gen}}^{\set{1,2 }}(Q, Y) \ll
    X^{o(1)}\sup_{Q_0Q_b\asymp Q} |\Sigma_{1, \mathrm{gen}}^{\set{1,2}}(Q_0, Q_b, Y)|,
  \end{equation}
  where
  \begin{multline}
    \Sigma_{1, \mathrm{gen}}^{\set{1,2}}(Q_0, Q_b, Y) \\
    := \frac{D^3M^3}{X} U^{5/6}T^{2/3}
    \sum_{|\alpha_1\alpha_2|_{\sup}\sim Y \\ \Gal(f_{\alpha_1\alpha_2}) = S_3 }
    \, \mathop{\sumSf_{q_0\sim Q_0}\sum_{q_b\sim Q_b}}_{\substack{p | q_b \implies p^2 | q_b\lor p | B_{\alpha_1\alpha_2} \\ M|q_b \\ (q_0, q_bB_{\alpha_1\alpha_2}) = 1}}
    S_1(\alpha_1, \alpha_2; q_0) S_1(\alpha_1, \alpha_2; q_b).
  \end{multline}

  Applications of Propositions \ref{proposition:cq6y0ob20l} and \ref{proposition:cq6y0pcoy5} yield that
  \begin{multline}
    \Sigma_{1, \mathrm{gen}}^{\set{1,2}}(Q_0, Q_b, Y) 
    \ll \frac{D^3M^3}{X^{1-o(1)}}U^{5/6}T^{2/3}
    \sum_{st\ll Q_b } s^3t \sum_{g\ll Y \\ s | g}
    \sum_{|\alpha|_{\sup} \sim Y/g \\ \not\exists \, p | \alpha \\ \Gal(f_{\alpha}) = S_3 \\ t^2 | N(\alpha) } \\ 
    \sumSf_{\substack{q_2 \ll Q_0 \\ (q_2,B_{g\alpha})=1}} r_{g\alpha}(q_2)
    \sum_{q_b\sim Q_b \\ st | q_b \\ p | q_b \implies p^2 | q_b\lor p | B_\alpha  } (q_b, \Delta_{\alpha})^{2/3} 
    \biggl| \sumSf_{\substack{q_1\sim Q_0/q_2 \\ (q_1,B_{g\alpha})=1}} a_{g\alpha}(q_1) \biggr|,
  \end{multline}
  where we have glued $\alpha_1\alpha_2=g\alpha$ with $\alpha$ primitive, and also summed over factorizations
  $q_0 = q_1q_2 $ after applying Proposition \ref{proposition:cq6y0ob20l} to obtain that
  \begin{equation}
    S_1(\alpha_1, \alpha_2; q_0) = \sum_{q_1q_2=q_0} a_{g\alpha}(q_1) r_{g\alpha}(q_2).
  \end{equation}
  Now, observe that since $(q_1, B_{g\alpha}) = 1 $, we have
  $(q_1, g)  = 1 $. Together with $q_1$ square-free, we have \begin{equation}
    a_{g\alpha}(q_1) = a_{\alpha}(q_1) =\lambda_\alpha(q_1).
  \end{equation}
  Moreover, $(q_2, N(g\alpha)) | (q_2, B_{g\alpha}) = 1 $ gives us
  that
  \begin{equation}
    r_{g\alpha}(q_2)\ll q_2^{-1 + o(1)}.
  \end{equation}
  By H\"older's inequality, recalling \eqref{eq:cuhlg52mjc}, we have that
  \begin{equation}
    \sum_{q_b\sim Q_b \\ p | q_b \implies p^2 | q_b\lor p | B_\alpha  \\ st | q_b} (q_b, \Delta_\alpha)^{2/3}
    \ll \biggl( \sum_{q_b\sim Q_b} (q_b, \Delta_\alpha) \biggr)^{2/3}
    \biggl( \sum_{q_b\sim Q_b \\ p | q_b \implies p^2 | q_b\lor p | B_\alpha \\ st | q_b} 1 \biggr)^{1/3}\\ 
    \ll X^{o(1)}\frac{Q_b^{5/6}}{(st)^{1/6}} .
  \end{equation}
  It follows that
  \begin{multline}
    \Sigma_{1, \mathrm{gen}}^{\set{1,2 }} (Q_0, Q_b, Y) \\ 
    \ll  \frac{D^3M^3}{X^{1-o(1)}} U^{5/6} T^{2/3} Q_b^{5/6} 
    \sum_{q_2\ll Q_0 } \frac{1}{q_2} \sum_{st\ll Q_b } s^{17/6} t^{5/6}
    \sum_{g\ll Y \\ s |g } \sum_{|\alpha|_{\sup} \sim Y/g \\ \not\exists \, p | \alpha \\ \Gal(f_\alpha) = S_3 \\ t^2 | N(\alpha) }
    \biggl| \sumSf_{\substack{q_1\sim Q_0/q_2 \\ (q_1,B_{g\alpha})=1}} \lambda_\alpha(q_1) \biggr|.
  \end{multline}
  At this point, we record that we have a trivial bound of 
  \begin{equation}\label{eq:cuhljd04hm}
    \ll \frac{D^3M^3}{X^{1-o(1)}} U^{5/6}T^{2/3} Q_0Q_b^{5/6}
    \sum_{s, t\ge 1 } s^{17/6} t^{5/6}\sum_{g\ll Y\\ s | g } \frac{Y^4}{g^4t^2}
    \ll \frac{M^{12}X^{4 + o(1)}}{D^6} \frac{1}{Q_b^{1/6}},
  \end{equation}
  so getting a power saving over $X^2 $ as long as $Q_b $ is at least sufficiently large power
  of $L $.

  Now, an application of Proposition \ref{prop:moment_rs_bound} yields that
  \begin{multline}
    \Sigma_{1, \mathrm{gen}}^{\set{1,2 }}(Q_0, Q_b, Y) \\ 
    \ll \frac{D^3M^3}{X^{1-o(1)}} U^{5/6}T^{2/3} Q_b^{5/6}\biggl( Q_0^{3/4}Y^{35/8}\sum_{s, t } \frac{1}{s^{37/24}t^{7/6}}
    + Q_0 Y^{63/16} \sum_{s, t } \frac{1}{s^{53/48}t^{25/24}}\biggr) \\
    \ll X^{o(1)} \biggl(Q_b^{1/12} \frac{M^{25/2}X^{33/8}}{D^{13/2}}
    + Q_b^{-1/6} \frac{M^{95/8}X^{63/16}}{D^{47/8}}\biggr)
  \end{multline}
  Combining with \eqref{eq:cuhljd04hm}, this is
  \begin{multline}
    \ll X^{o(1)} M^{12}\biggl(\min \biggl( Q_b^{1/12}\frac{X^{33/8}}{D^{13/2}}, Q_b^{-1/6}\frac{X^4}{D^6} \biggr) + \frac{X^{63/16}}{D^{47/8}} \biggr)
    \ll X^{o(1)}M^{12} \biggl( \frac{X^{49/12}}{D^{19/3}} + \frac{X^{63/16}}{D^{47/8}} \biggr) \\ 
    \ll M^{12}X^{2 + o(1)} \biggl(\frac{L^{19/3}}{X^{1/36}} + \frac{L^{47/8}}{X^{1/48}} \biggr)
    \ll M^{12}L^{47/8}X^{2 - 1/48 + o(1)},
  \end{multline}
  recalling that $L\ll X^{1/100} $ to absorb the first summand into the second
  for the final inequality.
  Combining with \eqref{multline:cuhli7dnls}'s bound of
  \begin{equation}
    \Sigma_{1, \mathrm{exc}}^{\set{1, 2}}\ll M^{35/3}L^{17/3}X^{2 - 1/18 + o(1)}
  \end{equation}
  for the exceptional case yields Proposition \ref{proposition:cq7567d3xk}.
\end{proof}

\subsection{Estimation of $\Sigma_1^{\set{1}}, \Sigma_1^{\set{2}}$: the partial zero frequencies}\label{sec:cuhge2lp8l}

In this section, we estimate $\Sigma_1^{\set{1}}, \Sigma_1^{\set{2}}$. Precisely, we show
\begin{proposition}\label{proposition:cq9d99egtf}
  Suppose that $j\in \set{1, 2}$ and that $q/X_2\gg X^{\varepsilon}$.
  Then, we have that
  \begin{equation}\label{eq:1}
    \Sigma_1^{\set{j}}\ll X^{2 + o(1)} \frac{X}{DX_j^2}.
  \end{equation}
  
\end{proposition}
For the proof, we shall require a minor lemma on the number of $x, y (q)$ for which $(x + y\zeta, q)$
has large norm, with averaging over $q$. Specifically, we show
\begin{lemma}\label{lemma:cq9d99fn96}
  For any $Q, B\ge 1$, we have that
  \begin{equation}
    \sum_{Q < q\le 2Q} \sum_{x, y(q) \\ N((x + y\zeta, q)) > B }  1\ll \frac{Q^3}{\sqrt{B}}.
  \end{equation}
\end{lemma}
\begin{proof}

  We shall sum over $g = (x ,y, q)$ and $\mathfrak{h}  = (x/g + y\zeta/g, q/g)$ separately, noting that
  \begin{equation}
    \#\set{x, y \, (q) : g \mathfrak{h}  | x + y\zeta} = \frac{q^2}{g^2 N\mathfrak{h} }.
  \end{equation}
  It follows that
  \begin{equation}
    \sum_{Q < q\le 2Q} \sum_{x, y \, (q) \\ N((x + y\zeta, q)) > B }  1 \le \sum_{g\ge 1 \\ \mathfrak{h} \subset \mathcal{O}_K \\  g^4 N\mathfrak{h} > B}\sum_{Q < q\le 2Q\\ g N\mathfrak{h} | q }
    \frac{q^2}{g N\mathfrak{h}}\le  Q^3\sum_{g\ge 1 \\ \mathfrak{h} \subset \mathcal{O}_K \\  g^4 N\mathfrak{h} > B}  \frac{1}{g^3 (N\mathfrak{h})^2}\ll \frac{Q^3}{\sqrt{B}},
  \end{equation}
  as desired.
\end{proof}
Armed with this, we can now quickly prove Proposition~\ref{proposition:cq9d99egtf}.
\begin{proof}[Proof of Proposition~\ref{proposition:cq9d99egtf}]
  Throughout the proof, fix $\varepsilon > 0$ sufficiently small.
  
  Without loss of generality, suppose that $j = 2$. When $M | q$, by orthogonality, we have
  \begin{equation}\label{multline:crb6gqw5u8}
    S_1(0, \alpha_2 ; q) = \frac{1}{M^8}\sum_{\gamma_1, \gamma_2\in \mathcal{O}_K/(M) } \psi_{M}(\beta_1'\gamma_1 + \beta_2'\gamma_2)
    \frac{1}{q}\sum_{x, y(q) } G_M(\alpha_2, x, y, \gamma_1, \gamma_2; q),
  \end{equation}
  where
  \begin{equation}
    G_M(\alpha_2, x, y, \gamma_1, \gamma_2; q)
    = \sum_{\beta_1\in \mathcal{O}_K/(q) \\ (x + y\zeta)\beta_1 \equiv (\alpha_2 + \gamma_2 q/M) \, (q) } \psi_M(\gamma_1\beta_1).
  \end{equation}

  All we shall require of $G_M$ is that it is $M(x + y\zeta, q)$-periodic in $\alpha_2$ and
  that it satisfies the bound
  \begin{equation}\label{eq:crb6xflijk}
    G_M(\alpha_2, x,y, \gamma_1, \gamma_2; q) \le N((x + y\zeta, q))\mathbbm{1}_{ (x + y\zeta, q)  |M\alpha_2 }.
  \end{equation}

  Opening up the definition of $I_1$ and inserting \eqref{multline:crb6gqw5u8} into the
  definition of $\Sigma_1^{\set{2}}$, we obtain that
  \begin{equation}\label{multline:crb6gfss40}
    \Sigma_1^{\set{2}} =
    \frac{1}{M^8}\sum_{\gamma_1, \gamma_2\in \mathcal{O}_K/(M) }\psi_M(\beta_1'\gamma_1 + \beta_2'\gamma_2)\Sigma_1^{\set{2}}(\gamma_1, \gamma_2),
  \end{equation}
  where
  \begin{multline}
    \Sigma_1^{\set{2}}(\gamma_1, \gamma_2) := \frac{X^4}{D^2} 
    \sum_{q\ge 1 \\ M | q} \frac{1}{q^5}
    \sum_{\alpha_2 \neq 0} \frac{1}{q} \sum_{x, y(q) } G_{M}(\alpha_2, x, y, \gamma_1, \gamma_2; q) \\
    \int_{K_\infty} \phi_1(x_1^\infty) \int_{K_\infty}  \omega_1 \biggl(\frac{|\ensuremath{\boldsymbol\ell}(x_1^\infty x_2^\infty)|X}{q D}\biggr)
    \phi_2(x_2^\infty) e \biggl( \frac{X_2}{q} \langle x_2^\infty, \alpha_2 \rangle\biggr) \,d x_1^\infty \,d x_2^\infty.
  \end{multline}
  In preparation for what will follow, note that the terms on the RHS of \eqref{multline:crb6gfss40}
  are supported on $q\ll X/D$ and, up to an error of $O_\varepsilon(X^{-100})$, $|\alpha_2|_{\sup}\le (q/X_2)X^\varepsilon$,
  for by repeated integration by parts, we have 
  \begin{multline}
    \iint_{K_\infty^2} \omega_1 \biggl( \frac{|\ensuremath{\boldsymbol\ell}(x_1^\infty x_2^\infty)|X}{q D} \biggr)
    \phi_1(x_1^\infty)\phi_2(x_2^\infty)
    e \biggl( \frac{X_2}{q} \langle x_2^\infty, \alpha_2 \rangle \biggr) \,d x_1^\infty \, d x_2^\infty \\ 
    \ll_A \biggl(\frac{q}{X/D}\biggr)^2\biggl( 1 + \frac{|\alpha_2|_{\sup}}{q/X_2} \biggr)^{-A}.
  \end{multline}
  The nonzeroness of $\alpha_2$ also implies that the contribution of $q \le X_2X^{-\varepsilon}$ is $O_\varepsilon(X^{-100})$.

  We shall discard those $x, y$ with $(x + y\zeta, q)$ of norm greater than
  $B_q = (q/X_2)^4M^{-4}X^{-2\varepsilon}$ so that we may apply Poisson summation in $\alpha_2$ and deal only
  with the zero-frequency.
  
  To this end, note that the contribution of $x, y$ with $N((x + y\zeta, q)) > B_q$ to \eqref{multline:crb6gfss40} is at
  most
  \begin{equation}
    \ll_\varepsilon X^{-100} + M^4X^{5\varepsilon}\frac{X^4}{D^2} \sum_{q\ll X/D}  \frac{1}{q^5} q 
    \biggl(\frac{q}{X_2}\biggr)^4 \biggl(\frac{q}{X/D}\biggr)^2  \frac{1}{\sqrt{(q/X_2)^4}}
    \ll_\varepsilon M^{6}X^{5\varepsilon} \frac{X^3 }{DX_2^2}.
  \end{equation}
  Therefore, we have that
  \begin{multline}\label{multline:crb6gwhpa3}
    \Sigma_1^{\set{2}}(\gamma_1, \gamma_2) = \frac{X^4}{D^2}
    \sum_{X_2X^{-\varepsilon} \le q\ll X/D } \frac{1}{q^6}
    \sum_{x, y(q) \\ N((x + y\zeta, q))\le B_q }
    \sum_{\alpha_2 \neq 0} G_M(\alpha_2, x, y, \gamma_1 ,\gamma_2; q) \\
    \int_{K_\infty} \phi_1(x_1^\infty)\int_{K_\infty} \omega_1 \biggl(\frac{|\ensuremath{\boldsymbol\ell}(x_1^\infty x_2^\infty)|X}{q D}\biggr) \phi_2(x_2^\infty) e \biggl( \frac{X_2}{q} \langle  x_2^\infty , \alpha_2 \rangle\biggr) \,d x_2^\infty \,d x_1^\infty
    + O_\varepsilon\biggl( M^{6}\frac{X^{3 + 5\varepsilon}}{DX_2^2} \biggr).
  \end{multline}
  We can reinsert the contribution of $\alpha_2 = 0$ at this point (critically, the restriction
  $q\ge X_2X^{-\varepsilon}$ has been made at this point, for otherwise the cost would be unacceptable)
  by Lemma \ref{lemma:cq9d99fn96} and the bound \eqref{eq:crb6xflijk} at the cost of
  an additional remainder of size
  \begin{equation}\label{eq:crb6gwhi70}
    \ll_\varepsilon
    \frac{X^4}{D^2} \sum_{X_2X^{-\varepsilon}\le q\ll X/D } \frac{1}{q^6}q^2 \biggl(\frac{q}{X/D}\biggr)^2
    \sqrt{B_q} 
    \ll_\varepsilon \frac{X^{3 + \varepsilon}}{X_2^2D}.
  \end{equation}
  It follows that
  \begin{multline}
    \Sigma_1^{\set{2}}(\gamma_1, \gamma_2) = \frac{X^4}{D^2}
    \sum_{X_2X^{-\varepsilon} \le q\ll X/D } \frac{1}{q^6}
    \sum_{x, y(q) \\ N((x + y\zeta, q))\le B_q }
    \sum_{\alpha_2} G_M(\alpha_2, x, y, \gamma_1 ,\gamma_2; q) \\
    \int_{K_\infty} \phi_1(x_1^\infty)\int_{K_\infty}
    \omega_1 \biggl(\frac{|\ensuremath{\boldsymbol\ell}(x_1^\infty x_2^\infty)|X}{q D}\biggr) \phi_2(x_2^\infty)
    e \biggl( \frac{X_2}{q} \langle  x_2^\infty , \alpha_2 \rangle\biggr) \,d x_2^\infty \,d x_1^\infty \\
    + O_\varepsilon\biggl( M^{6}\frac{X^{3 + O(\varepsilon)}}{DX_2^2}\biggr).
  \end{multline}
  
  In what remains, we shall fix $x, y, x_1^\infty, q$ for which $N(M(x + y\zeta, q))\le M^4B_q$,
  as is implied by the bounds we have put in thus far, and show a bound of $O(X^{-100})$ on 
  \begin{equation}\label{eq:crb6gvnl99}
    \sum_{\alpha_2 } G_M(\alpha_2, x,y , \gamma_1, \gamma_2; q)
    \int_{K_\infty} \omega_1 \biggl(\frac{|\ensuremath{\boldsymbol\ell}(x_1^\infty x_2^\infty)|X}{q D}\biggr)
    \phi_2(x_2^\infty) e \biggl(\frac{X_2}{q} \langle x_2^\infty, \alpha_2 \rangle\biggr) \,d x_2^\infty.
  \end{equation}
  Let $(\kappa) = M(x + y\zeta, q) $, and suppose that $|\kappa|_{\sup} \ll |N(\kappa)|^{1/4}$.
  Then, the fact that $G_M$ is $\kappa$-periodic in $\alpha_2$ implies that by Proposition
  \ref{proposition:cq6r4erq4l} (Poisson summation), we have that \eqref{eq:crb6gvnl99}
  equals
  
  \begin{equation}
    \biggl(\frac{q}{X_2}\biggr)^4 \sum_{\alpha_2}
    \biggl( \sum_{\beta \in \mathcal{O}_K/(\kappa)} G_M(\beta)\psi \biggl(-\frac{\alpha_2\beta}{\kappa}\biggr)\biggr)
    \omega_1 \biggl( \frac{X}{D X_2}
    \biggl| \ensuremath{\boldsymbol\ell} \biggl( \frac{x_1^\infty\alpha_2}{\kappa} \biggr) \biggr| \biggr)
    \phi_2 \biggl(\frac{q\alpha_2}{\kappa X_2}\biggr),
  \end{equation}
  where we're omitting the dependence of $G_M$ on $x$, $y$, $\gamma_1$, $\gamma_2$, $q$.
  Now, note that the compact support of $\phi_2$ along with the fact that for $\alpha_2$ nonzero, 
  \begin{equation}
    \biggl| \frac{q\alpha_2}{\kappa X_2} \biggr|_{\sup} \gg
    \frac{q}{X_2|N(\kappa)|^{1/4}}\ge \frac{q}{X_2(M^4B_q)^{1/4}}\gg X^{\varepsilon/2},
  \end{equation}
  and so we obtain that \eqref{eq:crb6gvnl99} vanishes for $X\gg_\varepsilon 1$. In particular, we have
  \begin{equation}\label{eq:crb6gwppew}
    \sum_{\alpha_2 } G_M(\alpha_2, x,y , \gamma_1, \gamma_2; q)
    \int_{K_\infty} \omega_1 \biggl(\frac{|\ensuremath{\boldsymbol\ell}(x_1^\infty x_2^\infty)|X}{q D}\biggr)
    \phi_2(x_2^\infty) e \biggl(\frac{X_2}{q} \langle x_2^\infty, \alpha_2 \rangle\biggr) \,d x_2^\infty\ll_\varepsilon X^{-100}.
  \end{equation}

  Taking $\varepsilon\to 0$ sufficiently slowly yields the desired result upon collecting
  \eqref{multline:crb6gfss40}, \eqref{multline:crb6gwhpa3}, \eqref{eq:crb6gvnl99}, and \eqref{eq:crb6gwppew}.
\end{proof}

\subsection{Estimation of $\Sigma_1^{\set{}}$: the full-zero frequencies}\label{sec:crb0xrxfh5}
We begin with the treatment of the full-zero frequencies $\Sigma_1^{\set{}}$.
The main result of this subsection is as follows.

\begin{proposition}\label{proposition:cran7rqg6g}
  We have 
  \begin{multline}
    \Sigma_1^{\set{}} =  - \frac{1}{2}X^2\sigma_\infty\prod_p\sigma_p 
    + \frac{X^3}{D}M^{-1}\mathcal{D}_1^*(1)\tilde\omega_1(1)\int_{K_\infty^2} \phi_1(x_1^\infty)\phi_2(x_2^\infty)
    |\ensuremath{\boldsymbol\ell}(x_1^\infty x_2^\infty)|^{-1}\,d x_1^\infty \,d x_2^\infty \\ 
    + O(M^6DX^{1 + o(1)}).
  \end{multline}
\end{proposition}

The remainder of this section is devoted to the proof of Proposition \ref{proposition:cran7rqg6g}.
We can open up the definition of $I_1$ from \eqref{eq:crb5jnfhh8} so that
\begin{equation}\label{eq:cq73wlj7um}
  \Sigma_1^{\set{}} = \frac{X^4}{D^2}
  \sum_{q\ge 1 \\ M | q} \frac{1}{q^5}S_1(0, 0; q)
  \int_{K_\infty^2} \phi_1(x_1^\infty)\phi_2(x_2^\infty)
  \omega_1 \biggl( \frac{|\ensuremath{\boldsymbol\ell}(x_1^\infty x_2^\infty)|MX}{qD} \biggr) \,d x_1^\infty \,d x_2^\infty.
\end{equation}
By Mellin inversion, we have that
\begin{equation}\label{eq:cralwom50t}
  \Sigma_1^{\set{}} = \frac{X^4}{D^2} \frac{1}{2\pi i}\int_{(3)} \mathcal{D}_1(s + 2) \tilde\omega_1(-s)
  U_1(s)
  \biggl(\frac{MX}{D}\biggr)^s \,d s, 
\end{equation}
where for $\Re(s) > -2$, 
\begin{equation}\label{eq:cq9cxlwwm4}
  U_1(s) = \int_{K_\infty^2} \phi_1(x_1^\infty)\phi_2(x_2^\infty) |\ensuremath{\boldsymbol\ell}(x_1^\infty x_2^\infty)|^s
  \,d x_1^\infty \,d x_2^\infty.
\end{equation}
All we shall require, coming from Lemma \ref{lemma:crb0xx4e28}, is that $U_1(s)$ has a
meromorphic continuation to $\Re(s) > -3$, in which region it has a pole only at $s = -2$ with
\begin{equation}\label{eq:crb0xzsdj6}
  \underset{s = -2}{\res} U_1(s) = 2\pi\int_{K_\infty^2} \phi_1(x_1^\infty)\phi_2(x_2^\infty)
  \delta(\ensuremath{\boldsymbol\ell}(x_1^\infty x_2^\infty)) \,d x_1^\infty \,d x_2^\infty,
\end{equation}
and that $|U_1(s)|\ll |s + 2|^{-1}$ for $\Re(s) > -3$. 

Therefore, applying Propositions \ref{lemma:cran7rw9t5} and moving the contour
in \eqref{eq:cralwom50t} to $ \Re(s) = -3 + 1/\log X$ yields that
\begin{multline}
  \Sigma_1^{\set{}} = \frac{X^4}{D^2}\frac{1}{2\pi i}
  \int_{(3)} \mathcal{D}_1(s + 2) \tilde\omega_1(-s) U_1(s) \biggl(\frac{MX}{D}\biggr)^s \,d s\\ 
  = \frac{X^3}{D}M^{-1}\mathcal{D}_1^*(1)\tilde\omega_1(1) U_1(-1) + X^2 M^{-2}\mathcal{D}_1^*(0)\zeta(0)\tilde\omega_1(2)
  \underset{s = -2}{\res} U_1(s) \\
  + \frac{X^4}{D^2}\frac{1}{2\pi i}
  \int_{(-3 + 1/\log X)} \zeta(s + 2)\mathcal{D}_1^*(s + 2)\tilde\omega_1(-s)
  U_1(s) \biggl(\frac{MX}{D}\biggr)^s \,d s.
\end{multline}
Recalling \eqref{eq:crb6w3pal3}, \eqref{eq:crb0xzsdj6}, that $2\pi\tilde\omega_1(2) = 1$ by \eqref{eq:cq5pgpuzx6}, and that $\zeta(0) = -1/2$,
Proposition \ref{proposition:cran7rqg6g} follows from bounds on $U_1$, $\zeta$ in vertical strips.

\section{Estimation of $\Sigma_2$}\label{sect:Sigma2}

In this section, we carry out the estimation of $\Sigma_2^S$, which is defined in \eqref{Sigma2SDef}.

\subsection{Estimation of $\Sigma_2^{\set{1,2}}$: nonzero frequencies}\label{sec:cuhge2kjb1}

We prove in this subsection the following bound on $\Sigma_2^{\set{1,2}}$.
\begin{proposition}\label{proposition:cran7w05id}
  We have that
  \begin{equation}\label{eq:crastnrpmq}
    \Sigma_2^{\set{1,2}}\ll X^{2 + o(1)} \biggl( \frac{D^{3/2}}{\sqrt{X}}
    + \frac{1}{D^{2/7}}\biggr) M^3.
  \end{equation}
\end{proposition}

We first record the following bound on incomplete norm forms satisfying a
divisibility condition.
\begin{lemma}\label{proposition:cravzd1g45}
  For any $b, Y\ge 1$,  we have that
  \begin{equation}
    \sum_{ |\alpha|_{\sup} < Y \\ \langle \alpha, 1 \rangle = 0 \\ b | N(\alpha) } 1 \ll 
    b^{o(1)}\frac{Y^3}{b^{3/4}} .
  \end{equation}
\end{lemma}
\begin{proof}
  If $b > Y^4$, the result is clear, so let's suppose $b\leq Y^4$ from now on.
  Take $r$ a prime in $[Y/b^{1/4}, 10Y(1 + \log b)/b^{1/4}]$ coprime to $b$. Then, we have that
  \begin{equation}
    \sum_{ |\alpha|_{\sup} < Y \\ \langle \alpha, 1 \rangle = 0 \\ b | N(\alpha) } 1 \le
    \sum_{\mathfrak{b}\subset \mathcal{O}_K \\ N\mathfrak{b} = b } \sum_{\beta\in \mathcal{O}_K/(r) \\ \langle \beta, 1 \rangle \equiv 0(r)}
    \sum_{ |\alpha|_{\sup} < Y \\ \mathfrak{b} | \alpha \\ \alpha \equiv \beta (r)} 1. 
  \end{equation}
  Note that the two congruence conditions, due to the coprimality of $b, r$, are
  together equivalent to a congruence condition modulo $\mathfrak{b} r$. Since $N(\mathfrak{b} r) = b r^4 \ll Y^4b^{o(1)}$,
  we may bound the above by
  \begin{equation}
    \ll  b^{o(1)}r^3 \frac{Y^4}{r^4b} \ll b^{o(1)}\frac{Y^3}{b^{3/4}},
  \end{equation}
  as desired.
\end{proof}

We are now ready to proceed with the proof of Proposition \ref{proposition:cran7w05id}, to which we shall
devote the remainder of this subsection.
Splitting the ranges of $q$, $d$, $|\alpha_1|_{\sup}$, $|\alpha_2|_{\sup}$ into dyadic intervals and applying
Propositions \ref{prop.S2non0} and \ref{prop.I2}, we have that 
\begin{equation}
  \Sigma_{2}^{\set{1,2}} \ll X^{-99} 
  + X^{o(1)}
  \sup_{\substack{D_1\ll D \\ Q\ll M\sqrt{XD}/D_1 \\
      A_1\ll X^{o(1)}\sqrt{XD}/X_1 \\ A_2\ll X^{o(1)}\sqrt{XD}/X_2 \\
      G'H' \ll D_1 \\ G''H''\ll Q
    }} \mathcal{S}(D_1, Q, A_1A_2, G', H', G'', H''),
\end{equation}
where $\mathcal{S} := \mathcal{S}(D_1, Q, A, G', H', G'', H'')$  is given by
\begin{multline}
  \frac{X^4}{D^2}
  \frac{1}{D_1^4\sqrt{DX}Q^5}\left(\frac{D_1Q}{M\sqrt{DX}}\right)^4
  \biggl(\frac{A}{M^2D}\biggr)^{-\frac{3}{2}}\frac{G^4H}{(D_1Q)^2}
  \sum_{ g'\asymp G' \\ h'\asymp H' \\ g''\asymp G'' \\ h''\asymp H'' } 
  \sum_{|\mathbf{c}| \asymp D/D_1 \\ (c_1, c_2) = 1}\sum_{d\asymp D_1  \\  g'h' | d } \sum_{q\asymp Q \\ g''h'' | q }\\ 
  \sum_{r | dq } \frac{1}{r} 
  \sum_{x, y, z \ll D_1Q \\ g | \gamma_{\mathbf{c}}(x, y, z; q)\\ y< z \text{ if } |c_1|\le |c_2|\\ z<y \text{ if } |c_2|< |c_1|} 
  \sum_{ |\alpha|_{\sup}\ll AM \\ g | \alpha \\ h^2 | N(\alpha/g)} \mathbbm{1}_{ dq/r | Y_\mathbf{c}(\alpha, x, y, z; q) },
\end{multline}
with the notation $g=g'g''$, $h=h'h''$, $G=G'G''$, $H=H'H''$ and $\gamma_\mathbf{c}$, $Y_\mathbf{c}$ as in Proposition \ref{prop.S2non0}. Here we first extend the $x,y,z$-sum mod $dq$ to any interval of size $D_1Q$ containing
$[d_1q,2d_1q]$ for any $d\asymp D_1$ and $q\asymp Q$. Since we can take two such intervals that are disjoint
from each other, we can impose the restriction $y<z$ or $z<y$ depending on the size
of $c_1,c_2$. We then extend the summations to $x,y,z\ll D_1Q$ under these restrictions. From now on, fix
some choice of $D_1, Q, A_1, A_2, G' , H', G'', H''$ satisfying the constraints in the
supremum, and suppress the dependence of $\mathcal{S}$ on these parameters.
Here, we have replaced $\alpha_1\alpha_2$ with $\alpha'$, and replaced $\alpha'$ with $\alpha =M\alpha'$ (at which point
we have dropped the conditions $M | \alpha, q$).

At this point, we may split (perhaps not uniquely) $r = uv$ for some $u | d$ and $v | q$.
Writing $d_1 = d/u$, $q_1 = q/v$ followed by a change of variable $\alpha\mapsto \alpha g$, we obtain that
\begin{multline}
  \mathcal{S} \le \frac{X^{3/2}}{D^3M} \frac{G^4H}{A^{3/2}D_1^2Q^3} 
  \sum_{ g'\asymp G' \\ h'\asymp H' \\ g''\asymp G'' \\ h''\asymp H'' }  \sum_{|\alpha|_{\sup}\ll AM/G \\ h^2 | N(\alpha) }
  \sum_{|\mathbf{c}| \asymp D/D_1 \\ (c_1, c_2) = 1} 
  \sum_{q_1v\asymp Q \\ g''h'' | q_1v } \frac{1}{v}\mathbbm{1}_{ q_1 | Y_\mathbf{c}(\alpha, x) }\sum_{x, y, z \ll D_1Q \\ g | \gamma_{\mathbf{c}}(x, y, z; q_1v)\\ y< z \text{ if } |c_1|\le |c_2|\\ z<y \text{ if } |c_2|< |c_1|}\\
  \sum_{d_1 u\asymp D_1 \\ g'h' | d_1u } \frac{1}{u}\mathbbm{1}_{ d_1 | Y_\mathbf{c}(\alpha, x, y, z; q_1v) }.
\end{multline}
Here, we have written $Y_\mathbf{c}(\alpha, x)$ to denote $Y_\mathbf{c}(\alpha, x, 0, 0; q) $, which doesn't depend on $q$.
We will also let $Y_\mathbf{c}(\alpha) = Y_\mathbf{c}(\alpha, 1)$ and observe that $Y_\mathbf{c}(\alpha, x) = x^3Y_\mathbf{c}(\alpha)$.
From now on, we may suppose that $H'\ll D_1$ and $H''\ll Q$. Splitting according to 
whether $Y_\mathbf{c}(\alpha)$ is nonzero or not yields that
\begin{equation}
  \mathcal{S}\ll \mathcal{S}_1 + \mathcal{S}_2,
\end{equation}
where
\begin{multline}
  \mathcal{S}_1 = \frac{X^{3/2}}{D^3M} \frac{G^4H}{A^{3/2}D_1^2Q^3}
  \sum_{ g'\asymp G' \\ h'\asymp H' \\ g''\asymp G'' \\ h''\asymp H'' }
  \sum_{|\mathbf{c}| \asymp D/D_1 \\ (c_1, c_2) = 1}\sum_{|\alpha|_{\sup}\ll AM/G \\ h^2 | N(\alpha) \\ Y_\mathbf{c}(\alpha)\neq 0 } 
  \sum_{q_1v\asymp Q \\ g''h'' | q_1v } \frac{1}{v}\mathbbm{1}_{ q_1 | Y_\mathbf{c}(\alpha, x) }\sum_{x, y, z \ll D_1Q \\ g | \gamma_{\mathbf{c}}(x, y, z; q_1v)\\ y< z \text{ if } |c_1|\le |c_2|\\ z<y \text{ if } |c_2|< |c_1|}\\
  \sum_{d_1 u\asymp D_1 \\ g'h' | d_1u } \frac{1}{u}\mathbbm{1}_{ d_1 | Y_\mathbf{c}(\alpha, x, y, z; q_1v) },
\end{multline}
\begin{multline}
  \mathcal{S}_2 = \frac{X^{3/2}}{D^3M} \frac{G^4H}{A^{3/2}D_1^2Q^3}
  \sum_{ g'\asymp G' \\ h'\asymp H' \\ g''\asymp G'' \\ h''\asymp H'' } 
  \sum_{|\mathbf{c}| \asymp D/D_1 \\ (c_1, c_2) = 1} \sum_{|\alpha|_{\sup}\ll AM/G \\ h^2 | N(\alpha) \\ Y_\mathbf{c}(\alpha) = 0 } 
  \sum_{q \asymp Q \\ g''h'' | q } \sum_{x, y, z \ll D_1Q \\ g | \gamma_{\mathbf{c}}(x, y, z; q)\\ y< z \text{ if } |c_1|\le |c_2|\\ z<y \text{ if } |c_2|< |c_1|}\\
  \sum_{d_1 u\asymp D_1 \\ g'h' | d_1u } \frac{1}{u}\mathbbm{1}_{ d_1 | Y_\mathbf{c}(\alpha, x, y, z; q) }.
\end{multline}

We start with a bound on $\mathcal{S}_1$, which we begin by discarding the $d_1$ and $y,z$ conditions, so that
\begin{multline}
  \mathcal{S}_1\ll
  \frac{X^{3/2}}{D^3M} \frac{G^4H}{A^{3/2}D_1Q^3}
  \sum_{g'\asymp G' \\ h'\asymp H' \\ g''\asymp G'' \\ h''\asymp H''}
  \frac{1}{g'h'}
  \sum_{|\mathbf{c}| \asymp D/D_1 \\ (c_1, c_2) = 1}\sum_{|\alpha|_{\sup}\ll AM/G \\ h^2 | N(\alpha) \\ Y_\mathbf{c}(\alpha)\neq 0 }\\ 
  \sum_{q_1v\asymp Q \\ g''h'' | q_1v } \frac{1}{v}\sum_{x, y, z \ll D_1Q \\ g | \gamma_{\mathbf{c}}(x, y, z; q_1v)}
  \mathbbm{1}_{ q_1 | x^3Y_\mathbf{c}(\alpha) }.
\end{multline}
We then further decompose $g'' = g''' w$ for some $g''' | q_1$ and $w | v$. Letting $q_2 = q_1/g'''$, we
obtain that
\begin{multline}
  \mathcal{S}_1 \ll 
  \frac{X^{3/2}}{D^3M} \frac{G^4H}{A^{3/2}D_1H'Q^3}
  \sum_{ w \ll G''  \\ h \asymp H}\frac{1}{w}
  \sum_{|\mathbf{c}| \asymp D/D_1 \\ (c_1, c_2) = 1}\sum_{|\alpha|_{\sup}\ll AM/G \\ h^2 | N(\alpha) \\ Y_\mathbf{c}(\alpha)\neq 0 } \\
  \sum_{x, y, z \ll D_1Q}
  \sum_{v\ll Q/w } \frac{1}{v}
  \sum_{g'''q_2\asymp Q/(vw) } \mathbbm{1}_{ g'''q_2 | x^3Y_\mathbf{c}(\alpha) } \\
  \ll \frac{X^{3/2}}{D^3M} \frac{G^4H}{A^{3/2}D_1H'Q^3}
  \biggl(\frac{D}{D_1}\biggr)^2
  \biggl(\frac{AM}{G}\biggr)^4\frac{1}{H} (D_1Q)^3\ll M^3X^{\frac{3}{2} + o(1)}D^{\frac{3}{2}},
\end{multline}
which is the first term in the statement of Proposition \ref{proposition:cran7w05id}.

Now, we turn to the degenerate case $\mathcal{S}_2$. We begin by similarly applying the divisor bound
to the sum over $d_1$, this time making no effort to make use of the condition $g'h' | d_1u$. It can be checked that as a polynomial in $q$, $Y_\mathbf{c}(\alpha, x,y , z; q)$ can only be identically
$0$ if $c_1y + c_2z = 0$ and $Y_\mathbf{c}(\alpha) = 0$. By the condition on $y,z$, the first of these cannot hold, and hence $Y_\mathbf{c}(\alpha, x,y , z; q)\neq0$. Thus, we sum over $d_1$ to obtain
\begin{equation}
  \mathcal{S}_2 \ll   \frac{X^{3/2}}{D^3M} \frac{G^4H}{A^{3/2}D_1^2Q^3}\sum_{g'\asymp G' \\ h'\asymp H' \\ g''\asymp G'' \\ h''\asymp H''} \sum_{|\mathbf{c}|\asymp D/D_1 \\ (c_1, c_2) = 1 }
  \sum_{|\alpha|_{\sup}\ll  AM/G \\ h^2 | N(\alpha) \\ Y_{\mathbf{c}}(\alpha) = 0}
  \sum_{x, y, z \ll D_1Q \\ y< z \text{ if } |c_1|\le |c_2|\\ z<y \text{ if } |c_2|< |c_1| }
  \sum_{q\asymp Q \\ g''h'' | q \\ g | \gamma_\mathbf{c}(x, y, z; q)}1.
\end{equation}
With $\gamma_{\mathbf{c}}(x, y, z; q) = c_2x+qy +(-c_1x+qz)\zeta$, the divisibility condition on $\gamma_{\mathbf{c}}(x, y, z; q)$
is satisfied a $\ll (g, q)/g^2$-fraction of the time, and so
\begin{equation}
  \sum_{q\asymp Q \\ g''h'' | q }  \sum_{x,y ,z\ll D_1Q \\ g | \gamma_{\mathbf{c}}(x, y,z; q) }
  \ll Q^{1+o(1)} \frac{(g,h'')}{g^2h''} (D_1Q)^3.
\end{equation}
Hence we arrive at
\begin{multline}
  \mathcal{S}_2 \ll  \frac{X^{3/2+o(1)}}{A^{3/2}D^3M} D_1G^2HQ\sum_{g'\asymp G' \\ h'\asymp H' \\ g''\asymp G'' \\ h''\asymp H''} \frac{(g, h'')}{h''}
  \sum_{|\mathbf{c}|\asymp D/D_1 \\ (c_1, c_2) = 1 }
  \sum_{|\alpha|_{\sup}\ll  AM/G \\ h^2 | N(\alpha) \\ Y_{\mathbf{c}}(\alpha) = 0}1 \\
  \ll
  \frac{X^{3/2+o(1)}}{A^{3/2}D^3M}D_1G^3H'Q\sum_{h\asymp H}  \sum_{|\mathbf{c}|\asymp D/D_1 \\ (c_1, c_2) = 1 }
  \sum_{|\alpha|_{\sup}\ll AM/G \\ h^2 | N(\alpha) \\ Y_\mathbf{c}(\alpha) = 0 } 1.
\end{multline}
Note that $\alpha$ with $Y_\mathbf{c}(\alpha) = 0$ are precisely those $\alpha$ of the form
$(c_2 - c_1\zeta)(m_0 + m_1\zeta + m_2\zeta^2)$, the second factor of which we call $\tau$.
Furthermore, we have that $|\alpha|_{\sup} \ll AM/G$ implies that $|m_i|\ll AD_1M/(GD)$. Summing
now over $\tau = m_0 + m_1\zeta + m_2\zeta^2$ yields that
\begin{equation}\label{eq:cravzedpyk}
  \mathcal{S}_2 \ll \frac{X^{3/2 + o(1)}}{A^{3/2}D^3M}D_1G^3H'Q
  \sum_{h\asymp H} \sum_{|\mathbf{c}| \asymp D/D_1  \\ (c_1, c_2) = 1 }
  \sum_{|\tau|_{\sup}\ll AD_1M/(GD) \\ \langle \tau, 1 \rangle = 0\\ h^2 | (c_1^4 + c_2^4) N(\tau)  } 1.
\end{equation}
At this point, we shall gather several different bounds. First, we note that
the divisor bound applied to the sum over $h$ (weakening the congruence condition
to $h | (c_1^4 + c_2^4)N(\tau)$) yields that
\begin{equation}\label{eq:cravzeqqkm}
  \mathcal{S}_2 \ll \frac{X^{3/2 + o(1)}}{A^{3/2}D^3M}D_1G^3H'Q \biggl(\frac{D}{D_1}\biggr)^2 \biggl(\frac{AD_1M}{GD}\biggr)^3
  \ll  M^3\frac{X^{2+o(1)}D_1H'}{D^2} \ll M^3\frac{X^{2+o(1)}D_1^2}{D^2}.
\end{equation}
This is nearly sufficient, apart from the extreme cases when $D_1, H'\approx D$.
To deal with these cases, we'll apply Lemma \ref{proposition:cravzd1g45}, which yields that
\begin{multline}
  \sum_{|\tau|_{\sup}\ll AD_1M/(GD) \\ \langle \tau, 1 \rangle = 0\\ h^2 | (c_1^4 + c_2^4) N(\tau)  } 1
  \le  \sum_{|\tau|_{\sup}\ll AD_1M/(GD) \\ \langle \tau, 1 \rangle = 0\\ h^2/(c_1^4+c_2^4, h^2) | N(\tau)  } 1
  \ll X^{o(1)}(h^2, c_1^4 + c_2^4)^{\frac{3}{4}} \frac{(AD_1M/(GD))^3}{h^{3/2}} \\ 
  \ll X^{o(1)} (D/D_1)^3\frac{(AD_1M/(GD))^3}{h^{3/2}}\ll X^{o(1)}M^3\frac{A^3/G^3}{h^{3/2}}.
\end{multline}
Plugged into \eqref{eq:cravzedpyk} yields that
\begin{equation}
  \mathcal{S}_2 \ll \frac{X^{3/2 + o(1)}}{A^{3/2}D^3M} D_1G^3H'Q H\biggl(\frac{D}{D_1}\biggr)^2\frac{A^3M^3}{G^3H^{3/2}}
  \ll M^3\frac{X^{2 + o(1)}D}{D_1^2/\sqrt{H'}}\ll M^3\frac{X^{2 + o(1)}D}{D_1^{3/2}}.
\end{equation}
Combined with \eqref{eq:cravzeqqkm}, we obtain that
\begin{equation}
  \mathcal{S}_2 \ll M^3\frac{X^{2 + o(1)}}{D^{\frac{2}{7}}},
\end{equation}
and the desired result follows.

\subsection{Estimation of $\Sigma_2^{\set{1}}, \Sigma_2^{\set{2}}$: the partial zero frequencies}\label{sec:cuhge2ozhj}

We prove in this subsection the following bound on the partial zero frequencies
$\Sigma_2^{\set{j}}$.
\begin{proposition}\label{proposition:crare1wa9u}
  We have that for $j\in\set{1, 2}$, 
  \begin{equation}
    \Sigma_2^{\set{j}} \ll X^{2 + o(1)} \frac{X^{3/2}}{X_j^3\sqrt{D}}.
  \end{equation} 
\end{proposition}

We will require the following estimate:
\begin{lemma}\label{proposition:cravzrp2bh}
  For any $D_1, Q, C, B$, we have that
  \begin{equation}
    \sum_{d\asymp D_1 \\ q\asymp Q} \sum_{|\mathbf{c}|\asymp C \\ (c_1, c_2) = 1 }
    \sum_{x(dq) \\y, z (d)  \\ N((x(c_2 - c_1\zeta) + q(y + z\zeta), dq)) > B}  1
    \ll \frac{C^2D_1^4Q^2}{B^{\frac{1}{4}}} (CD_1Q)^{o(1)}.
  \end{equation}
\end{lemma}

This follows immediately from the following result.
\begin{lemma}\label{lemma:crawy58bds}
  We have that for any $g | dq$, $h | (dq/g)^4$, and $(c_1, c_2) = 1$, 
  \begin{multline}
    \#\set{x (dq), y,z(d) : 
      g || x(c_2 - c_1\zeta) + q(y + z\zeta), h | N((x(c_2 - c_1\zeta) + q(y + z\zeta))/g)}\\
    \ll \frac{d^3q}{gh} (h, c_1^4 + c_2^4) (dq)^{o(1)}.
  \end{multline}

\end{lemma}
\begin{proof}
  By multiplicativity, we may suppose that $d, q$ are prime powers (of $p$, say).
  If $d = 1$, we have
  \begin{align}
    \# \set{x ( q) : g || x(c_2 - c_1\zeta), h | x^4(c_1^4 + c_2^4)/g^4} =  \frac{q}{g} \mathbbm{1}_{ h | c_1^4 + c_2^4 }\le \frac{q}{gh} (h, c_1^4 + c_2^4),
  \end{align}
  from which the desired result is clear, and if $q = 1$, we have that
  \begin{equation}
    \#\set{y, z (d) : g | y + z\zeta, h | (y^4 + z^4)/g^4}
    = \#\set{y_1, z_1 (d/g) : (y_1, z_1, p) = 1, h | y_1^4 + z_1^4} 
    \le  4\frac{d}{gh} .
  \end{equation}
  
  Therefore, we may suppose from now on that $d, q$ are powers of $p$ greater than $1$.
  We can split based on the power of $p$ dividing $x$.
  The contribution of $q | x$ reduces to the case of $q = 1$, and the contribution more
  generally of $(x, q) = q_1$ reduces to the case of $(x, q) = 1$ (with $q$ replaced by
  $q/q_1$). We'll therefore suppose from now on that $(x, q) = 1$.

  It can then be checked by Hensel's lemma that for any $w (d)$, we have
  \begin{equation}
    \#\set{y, z (d)  : (c_2 + qy)^4 + (-c_1 + qz)^4 \equiv c_1^4 + c_2^4 + wq (dq)}\ll d
  \end{equation}
  (in fact, for odd $p$, it is equal to $d$).
  The desired result follows.
\end{proof}

\begin{proof}[Proof of Lemma \ref{proposition:cravzrp2bh}]
  Letting $(x(c_2 - c_1\zeta) + q(y + z\zeta), dq) = g \mathfrak{h}$ for some $\mathfrak{h}$ of norm $h$ not contained in an ideal generated
  by a rational integer, we have that
  \begin{equation}
    \sum_{d\asymp D_1 \\ q\asymp Q} \sum_{|\mathbf{c}|\asymp C \\ (c_1, c_2) = 1 }
    \sum_{x(dq) \\y, z (d)  \\ N((x(c_2 - c_1\zeta) + q(y + z\zeta), dq)) > B}  1
    \le  \sum_{g\ge 1 \\ \mathfrak{h} \subset \mathcal{O}_K \\g^4 N\mathfrak{h} > B } \sum_{d\asymp D_1 \\ q\asymp Q \\ gh | dq }
    \sum_{|\mathbf{c}|\asymp C \\ (c_1, c_2) = 1 }
    \sum_{x (dq) \\ y,z(d) \\ g || x(c_2 - c_1\zeta) + q(y + z\zeta) \\ \mathfrak{h} | (x(c_2 - c_1\zeta) + q(y + z\zeta))/g }
    1.
  \end{equation}
  Applying Lemma \ref{lemma:crawy58bds}, we obtain that this is
  \begin{multline}
    (D_1Q)^{o(1)} \sum_{g, h\ge 1 \\ g^4h > B } \sum_{d\asymp D_1 \\ q\asymp Q \\ gh | dq }
    \sum_{|\mathbf{c}|\asymp C \\ (c_1, c_2) = 1 }
    \frac{d^3q}{gh} (h, c_1^4 + c_2^4)
    \ll (D_1Q)^{o(1)} D_1^4Q^2 \sum_{|\mathbf{c}|\asymp C \\ (c_1, c_2) = 1 }
    \sum_{g, h\ge 1 \\ g^4h > B } \frac{(h, c_1^4 + c_2^4)}{g^2h^2} \\
    \ll (CD_1Q)^{o(1)} \frac{D_1^4Q^2}{B^{\frac{1}{4}}}\sum_{|\mathbf{c}|\asymp C \\ (c_1, c_2) = 1 }
    \sum_{h\ge 1} \frac{(h, c_1^4 + c_2^4)}{h^{\frac{7}{4}}}
    \ll (CD_1Q)^{o(1)}  \frac{C^2D_1^4Q^2}{B^{\frac{1}{4}}},
  \end{multline}
  as desired.
\end{proof}
The remainder of this subsection is dedicated to the proof of Proposition \ref{proposition:crare1wa9u}.
We shall prove it in the case $j = 1$, for the other case is identical.

Throughout the proof, fix $\varepsilon > 0$ (which at the end we will send slowly to $0$).

Then, as in the proof of Proposition \ref{proposition:cq9d99egtf}, we have that 
\begin{equation}
  \Sigma_{2}^{\set{1}} = \frac{1}{M^8} \sum_{\gamma_1, \gamma_2\in \mathcal{O}_K/(M)} \psi_M(\beta_1'\gamma_1 + \beta_2'\gamma_2)\Sigma_2^{\set{1}}(\gamma_1, \gamma_2),
\end{equation}
where
\begin{multline}\label{multline:crb6hezh9q}
  \Sigma_2^{\set{1}}(\gamma_1, \gamma_2) = \frac{X^4}{D^2}
  \sum_{ d\ge 1 \\ \mathbf{c} \in \mathbb{Z}^2 \\ (c_1, c_2) = 1 } \frac{1}{d^5}
  \omega_1 \biggl(\frac{|\mathbf{c}|d}{D}\biggr) \frac{d}{\sqrt{DX}}
  \sum_{q\ge 1 \\ M | q} \frac{1}{q^5} \sum_{\alpha_2 \neq 0 } \frac{1}{d^2} \sum_{x (dq) \\ y, z (d) }
  G_M(\alpha_2) \\
  \int_{K_\infty^2} \phi_1(x_1^\infty)\phi_2(x_2^\infty)
  \biggl( \omega_2 \biggl( \frac{dq}{M\sqrt{DX}} \biggr)
  - \omega_2 \biggl(\frac{M\sqrt{X}\det(\mathbf{c}, \ensuremath{\boldsymbol\ell}(x_1^\infty x_2^\infty))}{q\sqrt{D}}\biggr)\biggr)
  e \biggl( \frac{X_2}{dq} \langle x_2^\infty, \alpha_2 \rangle \biggr)\,d x_1^\infty \,d x_2^\infty,
\end{multline}
with
\begin{equation}
  G_M(\alpha_2) = G_M(\alpha_2, x, y, z, \gamma_1, \gamma_2; d, q)
  = \sum_{\beta_2\in \mathcal{O}_K/ dq\mathcal{O}_K \\ \gamma_{\mathbf{c}}(x,y , z; q) \beta_2 \equiv (\alpha_2 + \frac{dq}{M}\gamma_1) \, (dq) }
  \psi_M(\gamma_2\beta_2).
\end{equation}
We'll omit the dependence of $\gamma_\mathbf{c}$ on $q$ from now on for convenience.
Similar to earlier, we have that $G_M$ is $M(\gamma_\mathbf{c}(x,y ,z), dq)$-periodic and
\begin{equation}
  |G_M(\alpha_2, x,y, z, \gamma_1, \gamma_2; d, q)| \le N((\gamma_\mathbf{c}(x, y, z), dq))
  \mathbbm{1}_{ (\gamma_\mathbf{c}(x,y, z), dq) | M\alpha_2 }.
\end{equation}

We have that the terms in \eqref{multline:crb6hezh9q} are supported on $d\ll D, dq \ll \sqrt{DX}$.
Accordingly, we focus on $d\sim D_1$, $q\sim Q$ for some $D_1\ll D$, $Q\ll \sqrt{DX}/D_1$.
It can also be checked that we have
\begin{multline}
  \int_{K_\infty^2} \phi_1(x_1^\infty)\phi_2(x_2^\infty)
  \biggl( \omega_2 \biggl( \frac{dq}{M\sqrt{DX}}\biggr) 
  - \omega_2 \biggl(\frac{M\sqrt{X}\det(\mathbf{c}, \ensuremath{\boldsymbol\ell}(x_1^\infty x_2^\infty))}{q\sqrt{D}}\biggr)\biggr)
  e \biggl( \frac{X_2}{dq} \langle x_2^\infty, \alpha_2 \rangle \biggr)\,d x_1^\infty \,d x_2^\infty \\ 
  \ll_A \frac{dq}{\sqrt{DX}} \biggl( 1 + \frac{|\alpha_2|_{\sup}}{dq/X_2} \biggr)^{-A}.
\end{multline}
Therefore, we may suppose that $D_1Q \gg X_2X^{-\varepsilon}$ at a cost of $O(X^{-100})$.

The same treatment, discarding those $x, y, z$ with
$N((\gamma_{\mathbf{c}}(x,y, z), dq)) > B = (D_1Q/X_2)^4M^{-4}X^{-\varepsilon}$ and then reinserting the contribution of
$\alpha_2 = 0$ at a cost of
\begin{multline}
  \ll \frac{X^{4 + o(1)}}{D^2} D_1 \frac{1}{D_1^5} \biggl(\frac{D}{D_1}\biggr)^2
  \frac{D_1}{\sqrt{DX}} Q \frac{1}{Q^5} \biggl(\frac{D_1QM}{X_2}\biggr)^4
  D_1Q \frac{D_1Q}{\sqrt{DX}} 
  \frac{1}{B^{1/4}} + \\ 
  \frac{X^4}{D^2} D_1 \frac{1}{D_1^5} \biggl(\frac{D}{D_1}\biggr)^2 \frac{D_1}{\sqrt{DX}}
  Q \frac{1}{Q^5} D_1Q  \frac{D_1Q}{\sqrt{DX}}B^{3/4}
  \\
  \ll \frac{X^{4 + O(\varepsilon)}D_1Q^2}{X_2^4DXB^{1/4}} + \frac{X^{3 + O(\varepsilon)}}{DD_1^3Q^2}B^{3/4}
  \ll \frac{X^{3 + O(\varepsilon)}}{X_2^3}\frac{Q}{D}\ll \frac{X^{7/2 + O(\varepsilon)}}{X_2^3\sqrt{D}}.
\end{multline}
It remains to bound, for $d\sim D_1$, $q\sim Q$, $\mathbf{c}$ primitive, and $x, y, z$ such that with
$(\kappa) = M(\gamma_\mathbf{c}(x, y, z), dq)$ of norm $\le B M^4$ (with $|\kappa|_{\sup}\ll |N(\kappa)|^{1/4}$), the sum 
\begin{equation}
  \sum_{\alpha_2 } G_M(\alpha_2) \int_{K_\infty} \phi_2(x_2^\infty)
  \biggl( \omega_2 \biggl(\frac{dq}{M\sqrt{DX}}\biggr)
  - \omega_2 \biggl(\frac{M\sqrt{X}\det(\mathbf{c}, \ensuremath{\boldsymbol\ell}(x_1^\infty x_2^\infty))}{q\sqrt{D}}\biggr)\biggr)
  e \biggl( \frac{X_2}{dq}\langle x_2^\infty, \alpha_2 \rangle \biggr) \,d x_2^\infty.
\end{equation}
By Poisson summation (recalling that $G_M$ is $\kappa$-periodic), we have that
this equals
\begin{equation}
  \sum_{\alpha_2 } \biggl( \sum_{\beta (\kappa) } G_M(\beta) \psi \biggl(-\frac{\alpha_2\beta}{\kappa}\biggr)  \biggr)
  \phi_2 \biggl(\frac{dq\alpha_2}{\kappa X_2}\biggr)
  \biggl( \omega_2 \biggl(\frac{dq}{M\sqrt{DX}}\biggr)
  - \omega_2 \biggl(\frac{d M\sqrt{X}}{X_2\sqrt{D}}\det\biggl(\mathbf{c}, \ensuremath{\boldsymbol\ell} \biggl(\frac{x_1^\infty\alpha_2}{\kappa}\biggr)\biggr)\biggr)\biggr).
\end{equation}
Now, note that for $\alpha_2$ nonzero, 
\begin{equation}
  \biggl| \frac{dq\alpha_2}{\kappa X_2} \biggr|_{\sup}
  \gg \frac{dq}{X_2|\kappa|_{\sup}}\gg X^{\varepsilon/2},
\end{equation}
from which the desired result follows from the compact support away from $0$ of $\phi_2$.

\subsection{Estimation of \texorpdfstring{$\Sigma_2^{\set{}}$}{Sigma2{}}: the full-zero frequencies}\label{sec:crao4uslaz}

We begin with the treatment of the full-zero frequencies $\Sigma_2^{\set{}}$ defined in
Proposition \ref{proposition:crb36eebaq}, which contributes a main term of size $X^3/D$, larger than the
``true'' main term of Theorem \ref{theorem:cq5o3jikeq} which has size $X^2$.

This is to cancel with the same illusory main term that arises in the evaluation of
$\Sigma_1^{\set{}}$ along with the true main term. We show the following.

\begin{proposition}\label{prop.Sigma2Final}
  We have 
  \begin{equation}
    \Sigma_2^{\set{}}= 2\tilde\omega_1(1)\frac{X^{3}}{D} U_1(-1) M^{-1}\mathcal{D}_1^*(1)
    + O\left(\frac{X^{2+o(1)}}{D^2}+\frac{X^{3+o(1)}}{D^4}\right).
  \end{equation}

\end{proposition}
Recall that $\mathcal{D}_1^*(s)$ is defined in \eqref{align:crb7s6n8bp} and $U_1(s)$ is defined in \eqref{eq:cq9cxlwwm4}.

Our proof will amount to two maneuvers. The first is to show that $\Sigma_2^{\set{}}$ is a smooth sum
over $\mathbb{Z}^2$ of a product of local densities. The second is to, with a good remainder term,
estimate this smooth sum of a product of local densities by expanding the product into
a convergent Dirichlet series whose coefficients are periodic and applying Poisson
summation (as it will turn out, we do not have to truncate the Dirichlet series).

\subsubsection{Extracting a sum of products of local densities}

Recalling the definition of $S_2$ in \eqref{S2Def} and $I_2$ in \eqref{I2Def}, we have that
\begin{equation}\label{eq:crb7r2ixlw}
  \Sigma_2^{\set{}} = \frac{X^4}{D^2}\sum_{d\ge 1 \\ \mathbf{c} \in \mathbb{Z}^2 \\ (c_1, c_2) = 1} \frac{1}{d^5}
  \omega_1 \biggl(\frac{|\mathbf{c}|d}{D}\biggr) \frac{d}{\sqrt{DX}} \Sigma_2^{\set{}}(\mathbf{c},d),
\end{equation}
where 
\begin{align}\label{qsumSigma2Star}
  \Sigma_2^{\set{}}(\mathbf{c}, d) = \sum_{q\ge 1 \\ M | q}\frac{1}{q^5}
  S_2(0,0; \mathbf{c}, d,q) I_2(0,0; \mathbf{c},d,q)
  =\sum_{q\ge 1\\ M|q}\frac{1}{d^3q^8}\sum_{(\beta_1, \beta_2)\in V_{dq} \\ 
  \det(\mathbf{c}, \ensuremath{\boldsymbol\ell}(\beta_1\beta_2))\equiv 0\,(d q/M) \\
  \ensuremath{\boldsymbol\ell} (\beta_1\beta_2)\equiv 0(d)} I_2(0,0; \mathbf{c},d,q),
\end{align}
Opening up the definition of $I_2$ and applying Mellin inversion yields that
\begin{multline}\label{eq:crb7rvd8db}
  \Sigma_2^{\set{}}(\mathbf{c}, d)  \\
  =  d^3 \frac{1}{2\pi i}\int_{(3)}
  \biggl( \biggl(\frac{M\sqrt{DX}}{d}\biggr)^s\hat\phi_1(0)\hat\phi_2(0)\tilde \omega_2(s)
  - \biggl(M\sqrt{\frac{X}{D}}\biggr)^s\tilde \omega_2(-s)U_2(s; \mathbf{c})  \biggr)
  \mathcal{D}_2(s + 1; \mathbf{c}, d)  \,d s,
\end{multline}
where $\mathcal{D}_2$ was defined in \S\ref{sec:crb7rws6oy}, and 
\begin{equation}
  U_2(s; \mathbf{c} ) = \int_{K_\infty^2} \phi_1(x_1^\infty)\phi_2(x_2^\infty)
  |\det(\mathbf{c}, \ensuremath{\boldsymbol\ell}(x_1^\infty x_2^\infty))|^s  \,d x_1^\infty \,d x_2^\infty. 
\end{equation}
Also, we have that
\begin{multline}
  \biggl( \biggl(\frac{M\sqrt{DX}}{d}\biggr)^s\hat\phi_1(0)\hat\phi_2(0)\tilde \omega_2(s)
  - \biggl(M\sqrt{\frac{X}{D}}\biggr)^s\tilde \omega_2(-s)U_2(s; \mathbf{c})  \biggr)
  \biggr|_{s = 0} \\
  = \tilde\omega_2(0) (\hat \phi_1(0)\hat\phi_2(0) - U_2(0; \mathbf{c} )) = 0,
\end{multline}
for $U_2(0; \mathbf{c}) = \hat\phi_1(0)\hat\phi_2(0)$ and that $\mathcal{D}_2(s + 1; \mathbf{c}, d)$ has only a simple pole at $s = 0$ by
Proposition \ref{prop:D_2_props}. Therefore, the residue of the integrand of \eqref{eq:crb7rvd8db}
at $s = 0$ is $0$.

Therefore, the only pole to the left of $\Re(s) > -3$ is at $s = -1$ by Lemma \ref{lemma:crb0xx4e28}
(with $F(x) = \det(\mathbf{c}/|\mathbf{c}|, \ensuremath{\boldsymbol\ell}(x))$ for $x\in K_\infty^2$), at which we obtain a residue of
\begin{multline}
  \underset{s = -1}{\res}\biggl[ \biggl( \biggl(\frac{M\sqrt{DX}}{d}\biggr)^s\hat\phi_1(0)\hat\phi_2(0)\tilde \omega_2(s)
  - \biggl(M\sqrt{\frac{X}{D}}\biggr)^s\tilde \omega_2(-s)U_2(s; \mathbf{c})  \biggr)
  \mathcal{D}_2(s + 1; \mathbf{c}, d)\biggr] \\
  = -\zeta(0)\tilde\omega_2(1) \frac{2\sqrt{\pi}}{\Gamma(1/2)}\frac{1}{M}\sqrt{\frac{D}{X}}\mathcal{D}_2^*(0; \mathbf{c}, d) \int_{K_\infty^2} \phi_1(x_1^\infty)\phi_2(x_2^\infty)
  \delta(\det(\mathbf{c} ,\ensuremath{\boldsymbol\ell}(x_1^\infty x_2^\infty))) \,d x_1^\infty \,d x_2^\infty.
\end{multline}
By Proposition \ref{prop:D_2_props}, this is equal to 
\begin{equation}
  \sqrt{\frac{D}{X}}\sigma_\infty(\mathbf{c})\prod_{p}\sigma_p(\mathbf{c}, d),
\end{equation}
recalling that
\begin{equation}
  \sigma_\infty(\mathbf{c})
  =  \int_{K_\infty^2} \phi_1(x_1^\infty)\phi_2(x_2^\infty)
  \delta(\det(\mathbf{c} ,\ensuremath{\boldsymbol\ell}(x_1^\infty x_2^\infty))) \,d x_1^\infty \,d x_2^\infty,
\end{equation}
\begin{equation}
  \sigma_p(\mathbf{c}, d) = \int_{\mathcal{O}_{K, p}^2} \mathbbm{1}_{ \substack{(\beta_1, \beta_2)\in V_{p^{m_p}} \\ d | \ensuremath{\boldsymbol\ell}(\beta_1\beta_2) }}
  \delta(\det(\mathbf{c}, \ensuremath{\boldsymbol\ell}(\beta_1\beta_2)))
  \, d \beta_1 \, d \beta_2.
\end{equation}

In particular, moving the contour in \eqref{eq:crb7rvd8db} to $\Re(s) = -3 + \delta$
for $\delta\to 0$ sufficiently slowly in terms of $X$, picking up the poles of $\zeta(s + 1)$ at $s = 0$
(which cancels with the difference which vanishes) and $U_2(s)$ at $s = -1$
(by Lemma \ref{lemma:crb0xx4e28}), we obtain that
\begin{equation}
  \Sigma_2^{\set{}}(\mathbf{c}, d) =  d^3 \sqrt{\frac{D}{X}} \sigma_\infty(\mathbf{c}) \prod_p \sigma_p(\mathbf{c}, d)
  +  O( d^6(DX)^{-\frac{3}{2}} (|\mathbf{c}| d)^{o(1)} ).
\end{equation}
Substituting this into \eqref{eq:crb7r2ixlw}, we obtain that
\begin{align}
  \Sigma_2^{\set{}} &=  \frac{X^3}{D^2} \sum_{d\ge 1 \\ \mathbf{c} \in \mathbb{Z}^2 \\ (c_1, c_2) = 1 }  \frac{1}{d}\omega_1 \biggl(\frac{|\mathbf{c}|d}{D}\biggr)
  \sigma_\infty(\mathbf{c})\prod_{p} \sigma_p(\mathbf{c}, d) \\ 
               &\hspace{2cm}+ O \biggl( \frac{X^{4 + o(1)}}{D^2}\sum_{d\ll D} \frac{1}{d^5}\biggl(\frac{D}{d}\biggr)^2\frac{d}{\sqrt{DX}} d^6 (DX)^{-\frac{3}{2}} \biggr) \\ 
               &= \frac{X^3}{D^2} \sum_{d\ge 1 \\ \mathbf{c} \in \mathbb{Z}^2 \\ (c_1, c_2) = 1 }  \frac{1}{d}\omega_1 \biggl(\frac{|\mathbf{c}|d}{D}\biggr)
  \sigma_\infty(\mathbf{c})\prod_{p} \sigma_p(\mathbf{c}, d) + O \biggl( \frac{X^{2 + o(1)}}{D} \biggr).
\end{align}
We shall now show that this is a sum over $\mathbf{v} = \mathbf{c}d$, so that we may sum smoothly over
$\mathbb{Z}^2$ and drop the primitivity condition.
First note that, by a change of variables,
\begin{equation}
  \sigma_\infty(\mathbf{c}) = \frac{d}{D} \sigma_\infty \biggl( \frac{\mathbf{c} d}{D} \biggr).
\end{equation}
Slightly more work is required at the non-archimedean places, which is contained in
Lemma \ref{lemma:crb7r6zwci}.

Applying \eqref{eq:crb7r7g0xm} in Lemma \ref{lemma:crb7r6zwci}, we obtain that
\begin{equation}\label{eq:crb7r8f9bf}
  \Sigma_2^{\set{}} = \frac{X^3}{D^3} \sum_{\mathbf{v} \in \mathbb{Z}^2} \sigma_\infty \biggl(\frac{\mathbf{v}}{D}\biggr)\prod_p \sigma_p(\mathbf{v})
  \omega_1 \biggl(\frac{|\mathbf{v}|}{D}\biggr)
  + O \biggl(\frac{X^{2 + o(1)}}{D}\biggr).
\end{equation}

\subsubsection{Summing the products of local factors}\label{sec:cuhlxu5ldf}

The main result of this subsection is the following.
\begin{proposition}\label{proposition:crb7r8dzyl}
  We have that
  \begin{multline}
    \sum_{\mathbf{v} \in \mathbb{Z}^2} \sigma_\infty \biggl(\frac{\mathbf{v}}{D}\biggr)\prod_p \sigma_p(\mathbf{v})
    \omega_1 \biggl(\frac{|\mathbf{v}|}{D}\biggr) \\ 
    = 2 D^2M^{-1}\tilde\omega_1(1) \mathcal{D}_1^*(1)
    \int_{K_\infty^2} \phi_1(x_1^\infty)\phi_2(x_2^\infty) |\ensuremath{\boldsymbol\ell}(x_1^\infty x_2^\infty)|^{-1}
    \,d x_1^\infty \,d x_2^\infty + O \biggl( \frac{1}{D^{1 - o(1)}} \biggr).
  \end{multline}
\end{proposition}
This and \eqref{eq:crb7r8f9bf} imply Proposition \ref{prop.Sigma2Final}.

Expanding out with \eqref{eq:crb7r7g0xm} (with absolute convergence following
from the bounds \eqref{eq:crb7r7lmsp}), we have that
\begin{equation}
  \sum_{\mathbf{v} \in \mathbb{Z}^2} \sigma_\infty \biggl(\frac{\mathbf{v}}{D}\biggr)\prod_p \sigma_p(\mathbf{v})
  \omega_1 \biggl(\frac{|\mathbf{v}|}{D}\biggr) =
  \sum_{q\ge 1 }\sum_{\mathbf{v} \in \mathbb{Z}^2 }S(\mathbf{v}; q)
  \sigma_\infty \biggl(\frac{\mathbf{v}}{D}\biggr)\omega_1 \biggl(\frac{|\mathbf{v}|}{D}\biggr).
\end{equation}
Applying Poisson summation, we obtain that this equals
\begin{equation}\label{eq:crb8bepw0f}
  D^2 \sum_{q\ge 1 } \frac{1}{q^2} \sum_{\mathbf{w} \in \mathbb{Z}^2 } \hat S(\mathbf{w}; q)
  \int_{\mathbb{R}^2} \sigma_\infty(\mathbf{v}) \omega_1(|\mathbf{v}|)
  e\biggl(-\frac{D}{q}\langle \mathbf{v}, \mathbf{w} \rangle\biggr)  \,d \mathbf{v} ,
\end{equation}
where
\begin{equation}
  \hat S(\mathbf{w} ; q)
  := \sum_{\mathbf{v} (q) } S(\mathbf{v} ; q)
  e_q ( - \langle \mathbf{v}, \mathbf{w} \rangle).
\end{equation}
The evaluation of $\hat S(\mathbf{w}; q)$ required is contained in the following lemma.
\begin{lemma} \label{lem:S0_fourier}
  We have that
  \begin{equation}
    \hat S(\mathbf{0} ; q)  = \frac{q}{M^2}N_1^*(qM).
  \end{equation}
  Furthermore, for $\mathbf{w} \neq 0$, we have that
  \begin{equation}\label{eq:crb8bjp1jg}
    \hat S(\mathbf{w} ; q) \ll  M^{O(1)}\frac{(w_1^4 + w_2^4, q)}{q^{2 - o(1)}}.
  \end{equation}
\end{lemma}
We are now ready to prove Proposition \ref{proposition:crb7r8dzyl}
\begin{proof}[Proof of Proposition \ref{proposition:crb7r8dzyl} assuming Lemma \ref{lem:S0_fourier}]
  The contribution of $\mathbf{w} = 0$ to \eqref{eq:crb8bepw0f} is, by Lemma \ref{lem:S0_fourier}
  and Lemma \ref{lemma:crb8be0e8e}, equal to  
  \begin{multline}
    D^2 \int_{\mathbb{R}^2} \sigma_\infty(\mathbf{v})\omega_1(|\mathbf{v}|) \,d \mathbf{v}
    \sum_{q\ge 1 } \frac{\hat S(\mathbf{0} ; q)}{q^2}
    = 2\tilde\omega_1(1) U_1(-1) D^2M^{-1}
    \sum_{q\ge 1 \\ M | q } \frac{N_1^*(q)}{q} \\ 
    = 2\tilde\omega_1(1) U_1(-1) D^2M^{-1} \mathcal{D}_1^*(1),
  \end{multline}
  matching the main term in Proposition \ref{proposition:crb7r8dzyl}.
  It remains to bound the contribution of the nonzero frequencies.
  First, note that by repeated integration by parts, we have that
  \begin{equation}
    \int_{\mathbb{R}^2} \sigma_\infty(\mathbf{v})
    \omega_1(|\mathbf{v}|) e \biggl( -\frac{D}{q}\langle \mathbf{v}, \mathbf{w} \rangle \biggr)\ll_A
    \biggl( 1 + \frac{D|\mathbf{w}|}{q} \biggr)^{-A},
  \end{equation}
  so by the bound \eqref{eq:crb8bjp1jg}, we are reduced to bounding by
  $M^{O(1)}/D^{1 - o(1)}$ the quantity
  \begin{multline}
    M^{O(1)}D^2\sum_{ q } \frac{1}{q^4}
    \sum_{ \mathbf{w} \neq 0} (w_1^4 + w_2^4, q)
    \biggl( 1 + \frac{D|\mathbf{w}|}{q} \biggr)^{-A} \\ 
    \ll M^{O(1)}D^2 \sum_{Q = 2^k \\ k\ge 0 } \frac{1}{Q^4}\sum_{\mathbf{w} \neq 0 }
    \biggl( 1 + \frac{D|\mathbf{w}|}{Q} \biggr)^{-A}
    \sum_{q\sim Q } (w_1^4 + w_2^4, q) \\ 
    \ll M^{O(1)}D^2 \sum_{Q = 2^k \\ k \ge 0 } \frac{1}{Q^3}
    \sum_{|\mathbf{w}| \neq 0 } |\mathbf{w}|^{o(1)}
    \biggl( 1 + \frac{D|\mathbf{w}|}{Q} \biggr)^{-A}
    \ll M^{O(1)} \sum_{Q = 2^k  \\ Q\gg D^{1 - o(1)}  } \frac{1}{Q^{1 - o(1)}}
    \ll \frac{1}{D^{1 - o(1)}},
  \end{multline}
  as desired.
\end{proof}

\begin{proof}[Proof of Lemma \ref{lem:S0_fourier}]
  By \eqref{SvqDef}, we have
  \begin{equation}\label{tempMult}
    \hat{S}(\mathbf{w} ; q)=\prod_{p^k||q} \hat{S}_p(\mathbf{w};k),
  \end{equation}
  where
  \begin{equation}
    \hat{S}_p(\mathbf{w} ; k) = \sum_{\mathbf{v} (p^k) } S_p(\mathbf{v} ; k)
    e_{p^k} ( - \langle \mathbf{v}, \mathbf{w} \rangle).
  \end{equation}
  Let $p$ be a prime and $k\geq0$. Recall the definition of $S_p(\mathbf{v};k)$ from \eqref{SDef}.
  For $k = 0$, we have 
  \begin{equation}\label{Shatk0Case}
    \hat S_p(\mathbf{0}; 0) =  p^{-8m_p}=p^{-2m_p}N_1^*(p^{m_p}).
  \end{equation}
  Technically $N_1^*(M)$ contains the condition $\ensuremath{\boldsymbol{\ell}}(\beta_1'\beta_2')\equiv 0\, (M)$, but the reader is
  reminded that this is assumed throughout the entire paper.
  
  For $k\geq1$, write $\eta=\max\{0,m_p-k\}$. We have that
  \begin{align}
    \hat S_p(\mathbf{0} ; k) = &\ \frac{1}{p^{9k+8\eta}} \sum_{\mathbf{v} \, (p^k) }
                                 \sum_{\mathbf{a}=(a_1,a_2)\, (p^k) \\ (a_1,a_2, p) = 1 }\sum_{w\, (p^k)}\sum_{(\beta_1, \beta_2)\in V_{p^k} }
    e_{p^k} (\langle \ensuremath{\boldsymbol\ell}(\beta_1\beta_2)-\mathbf{v} w,\mathbf{a} \rangle)\nonumber\\
    =&\  \frac{1}{p^{9k+8\eta}}\sum_{j=0,1}(-1)^{j}\sum_{\mathbf{v} \, (p^k) }
       \sum_{\mathbf{a}\, (p^{k-j})}\sum_{w\, (p^k)}\sum_{(\beta_1, \beta_2)\in V_{p^k} }
       e_{p^{k-j}} (\langle \ensuremath{\boldsymbol\ell}(\beta_1\beta_2)-\mathbf{v} w,\mathbf{a}\rangle)\nonumber\\
    =&\ \sum_{j=0,1}\frac{(-1)^{j}}{p^{7k+2j+8\eta}}\sum_{\mathbf{v} \, (p^k) }
       \sum_{w\, (p^k)}\sum_{(\beta_1, \beta_2)\in V_{p^k} \\ \ensuremath{\boldsymbol\ell}(\beta_1\beta_2)\equiv\mathbf{v} w \, (p^{k-j}) }1\nonumber\\
    =&\ \sum_{j=0,1}\frac{(-1)^{j}}{p^{7k-j+8(\eta-j\charf{k>m_p})}}\sum_{\mathbf{v} \, (p^{k-j}) }
       \sum_{w\, (p^{k-j})}\sum_{(\beta_1, \beta_2)\in V_{p^{k-j}} \\ \ensuremath{\boldsymbol\ell}(\beta_1\beta_2)\equiv\mathbf{v} w \, (p^{k-j}) }1.
  \end{align}
  Here in the last step $\charf{k>m_p}$ appears because of scaling $\beta_1,\beta_2$ restricted to $\beta_i\equiv \beta_i' \, (p^{m_p})$. The reader is reminded of this fact throughout the proof. Pulling out the powers of $p$ from $w$, we obtain
  \begin{align}
    \hat S_p(\mathbf{0} ; k) = &\ \sum_{j=0,1}\sum_{r=0}^{k-j}\frac{(-1)^{j}}{p^{7k-j+8(\eta-j\charf{k>m_p})}}\sum_{\mathbf{v} \, (p^{k-j}) }
                                 \sumCp_{w\, (p^{k-j-r})}\sum_{(\beta_1, \beta_2)\in V_{p^{k-j}} \\ \ensuremath{\boldsymbol\ell}(\beta_1\beta_2)\equiv\mathbf{v} w p^{r} \, (p^{k-j}) }1\nonumber\\
    =&\ \sum_{j=0,1}\sum_{r=0}^{k-j}\frac{(-1)^{j}\varphi(p^{k-j-r})}{p^{7k-j-2r+8(\eta-j\charf{k>m_p})}}\sum_{\mathbf{v} \, (p^{k-j-r}) }
       \sum_{(\beta_1, \beta_2)\in V_{p^{k-j}} \\ \ensuremath{\boldsymbol\ell}(\beta_1\beta_2)\equiv\mathbf{v} p^{r} \, (p^{k-j}) }1\nonumber\\
    =&\ \sum_{j=0,1}\sum_{r=0}^{k-j}\frac{(-1)^{j}\varphi(p^{k-j-r})}{p^{7k-j-2r+8(\eta-j\charf{k>m_p})}}
       \sum_{(\beta_1, \beta_2)\in V_{p^{k-j}} \\ \ensuremath{\boldsymbol\ell}(\beta_1\beta_2)\equiv0 \, (p^{r})}1.
  \end{align}
  Notice that the contribution of $r=k$ is given by
  \begin{equation}
    \frac{1}{p^{5k+8\eta}}\sum_{(\beta_1, \beta_2)\in V_{p^{k}} \\ \ensuremath{\boldsymbol\ell}(\beta_1\beta_2)\equiv0 \, (p^{k})}1=\frac{1}{p^{5k+8\eta+ 8\min\{m_p,k\}}}\sum_{(\beta_1, \beta_2)\in V_{p^{k+m_p}} \\ \ensuremath{\boldsymbol\ell}(\beta_1\beta_2)\equiv0 \, (p^{k})}1=p^{k-2m_p}\tilde{N}_1(p^{k+m_p}).
  \end{equation}
  Similarly, the contribution of $r=k-1$ is given by
  \begin{align}
    & \frac{p-1}{p^{5k+2+8\eta}}\sum_{(\beta_1, \beta_2)\in V_{p^{k}} \\ \ensuremath{\boldsymbol\ell}(\beta_1\beta_2)\equiv0 \, (p^{k-1})}1-\frac{1}{p^{5k+1+8(\eta-\charf{k>m_p})}}\sum_{(\beta_1, \beta_2)\in V_{p^{k-1}} \\ \ensuremath{\boldsymbol\ell}(\beta_1\beta_2)\equiv0 \, (p^{k-1})}1\nonumber\\
    &\ = -\frac{1}{p^{5k+2+8(\eta-\charf{k>m_p})}}\sum_{(\beta_1, \beta_2)\in V_{p^{k-1}} \\ \ensuremath{\boldsymbol\ell}(\beta_1\beta_2)\equiv0 \, (p^{k-1})}1\nonumber\\
    &\ =-\frac{1}{p^{5k+2+8(\eta-\charf{k>m_p}+\min\{m_p,k-1\})}}\sum_{(\beta_1, \beta_2)\in V_{p^{k+m_p-1}} \\ \ensuremath{\boldsymbol\ell}(\beta_1\beta_2)\equiv0 \, (p^{k-1})}1= -p^{k-2m_p}\tilde{N}_1(p^{k+m_p-1}),
  \end{align}
  and for $r\le k-2$, the contribution is
  \begin{equation}
    \frac{p-1}{p^{6k-r+1+8\eta}}\sum_{(\beta_1, \beta_2)\in V_{p^{k}} \\ \ensuremath{\boldsymbol\ell}(\beta_1\beta_2)\equiv0 \, (p^{r})}1
    -\frac{p-1}{p^{6k-r+1+8(\eta-\charf{k>m_p})}}\sum_{(\beta_1, \beta_2)\in V_{p^{k-1}} \\ \ensuremath{\boldsymbol\ell}(\beta_1\beta_2)\equiv0 \, (p^{r})}1=0.
  \end{equation}
  Hence we conclude for $k\geq 1$,
  \begin{equation}\label{Shatkgeneralcase}
    \hat S_p(\mathbf{0}; k) = p^{k-2m_p}N_1^*(p^{k+m_p}).
  \end{equation}

  Inserting \eqref{Shatk0Case} and \eqref{Shatkgeneralcase} into \eqref{tempMult}, we obtain the first part of the lemma.

  It remains to prove \eqref{eq:crb8bjp1jg}. By \eqref{tempMult}, it suffices to prove a local bound for
  $\hat S_p(\mathbf{w}; k)$ when $\mathbf{w}\neq 0$. The case $k = 0$ is trivial, so assume $k\ge 1$ and write
  $\eta_0 = \min\{m_p, k\}$.

  Proceeding as in the proof of \eqref{eq:crb7r7lmsp} (starting from \eqref{SDef}, detecting $\beta_2\equiv\beta_2'$ by additive
  characters, and summing over $\beta_2$), inserting the resulting expression for $S_p(\mathbf{v};k)$ into the definition of
  $\hat S_p(\mathbf{w};k)$ and summing over $\mathbf{v}(p^k)$ yields that
  \begin{equation}\label{eq:crb8n0u0ak}
    |\hat S_p(\mathbf{w}; k)|
    \le \frac{p^{O(\eta_0)}}{p^{3k}}\sum_{u\, (p^k)}\sum_{\substack{a_1,a_2\, (p^k)\\ (a_1, a_2, p)=1\\ u(a_1,a_2)\equiv -\mathbf{w}\, (p^k)}}
    N((a_1 + a_2\zeta, p^k)).
  \end{equation}
  By Lemma \ref{lem.Only1Prime}, we have $N((a_1 + a_2\zeta, p^k)) = (a_1^4 + a_2^4, p^k)$.
  For each $u$ with $p^r || u$, the congruence $u(a_1,a_2)\equiv -\mathbf{w}\, (p^k)$ has $\le p^{2r}$ solutions
  $(a_1,a_2)\, (p^k)$. Furthermore, for any such solution, we have
  \begin{equation}
    w_1^4 + w_2^4 \equiv u^4(a_1^4 + a_2^4)\, (p^k),
  \end{equation}
  so $(a_1^4 + a_2^4, p^k)\le (w_1^4 + w_2^4, p^k)/p^{4r}$.
  Since there are $\ll p^{k-r}$ choices of $u\, (p^k)$ with $p^r||u$, we conclude from \eqref{eq:crb8n0u0ak} that
  \begin{align}
    |\hat S_p(\mathbf{w}; k)|
    &\ll \frac{p^{O(\eta_0)}}{p^{3k}}\sum_{r=0}^k p^{k-r}\cdot p^{2r}\cdot \frac{(w_1^4 + w_2^4, p^k)}{p^{4r}}
      \ll \frac{p^{O(\eta_0)}}{p^{2k}}(w_1^4 + w_2^4, p^k).
  \end{align}
  Multiplying over primes gives \eqref{eq:crb8bjp1jg}.
\end{proof}

\section{The unbalanced case}\label{sec:cuhe1tf31a}
In this section, we shall quickly show the main theorem with a remainder of
\begin{equation}
  \ll M^{O(1)}X_1^2X_2^2\min \biggl( \frac{X_1}{X_2}, \frac{X_2}{X_1} \biggr).
\end{equation}
We will be brief and omit a number of details, for this result has already been shown
implicitly in prior work (one way, for example, is to reverse the passage from Theorem \ref{theorem:cq5o3jikeq}
to Theorem \ref{theorem:cq5oju0wf8}, noting that both Daniel \cite{MR1670278} and
Friedlander--Iwaniec \cite[Theorem 22.20]{FI} deal with divisors in such an unbalanced range).

Assume without loss of generality that $X_1\ge X_2$ throughout this section, so
we wish to show a remainder of $\ll M^{O(1)}X_1X_2^3$. Throughout, we will ignore factors of
$M^{O(1)} $, absorbing them into implied constants without mention.

We will execute the sum over $\alpha_1$ first, which reduces to counting
points of magnitude $\ll X_1$ in a codimension $2$ sublattice of $\mathbb{Z}^4$.
This codimension $2$ lattice has shortest vector of length $\ll |\alpha_2|_{\sup}$, so Poisson
summation or basic geometry of numbers estimates yields an accurate estimate for
this inner sum over $\alpha_1$. Executing the sum of these main terms over $\alpha_2$ then yields
the desired result.

Write
\begin{equation}
  \Lambda(\alpha_2) = \set{\alpha_1 \in \mathcal{O}_K : \ensuremath{\boldsymbol\ell}(\alpha_1\alpha_2) = 0},
\end{equation}
and let $\mathrm{covol}(\Lambda(\alpha_2))$ denote the covolume of $\Lambda$ in its $\mathbb{R}$-span, a plane which we
denote $W(\alpha_2)\subset K_\infty$.

We wish to estimate
\begin{equation}
  \sum_{\alpha_2 \\ \alpha_2  \equiv \beta_2' (M) } \sum_{\alpha_1\in \Lambda(\alpha_2) \\ \alpha_1 \equiv \beta_1'(M) }
  \Phi^{\infty} \biggl( \frac{\alpha_1}{X_1}, \frac{\alpha_2}{X_2} \biggr).
\end{equation}
Poisson summation in the $\alpha_1 $ sum (or simply counting lattice points in a region, bounding
the error based on the length of the boundary) yields that the inner sum equals
\begin{equation}
  \frac{X_1^2}{M^2 \mathrm{covol}(\Lambda(\alpha_2))} \mathbbm{1}_{ \ensuremath{\boldsymbol\ell}(\beta_1'\alpha_2) \equiv 0(M) }
  \int_{W(\alpha_2)} \Phi^\infty \biggl(x_1^\infty, \frac{\alpha_2}{X_2}\biggr) \,d x_1^\infty
  + O \biggl( \frac{X_1}{\sqrt{\mathrm{covol}(\Lambda(\alpha_2))}} \biggr).
\end{equation}
We record that because of the support conditions on $\Phi^\infty $, we have that
\begin{equation}
  |\alpha_2|_\infty \asymp |\alpha_2|_{\sup}\asymp X_2
\end{equation}
for all terms that contribute to the above.
Therefore, we have that $\mathrm{covol}(\Lambda(\alpha_2))\asymp |\alpha_2|_\infty^2 \asymp X_2 $, so summing over $\alpha_2 $ yields the desired
remainder term and it remains to analyze
the sum of main terms
\begin{equation}
  \frac{1}{X_2^2}\sum_{\alpha_2 } \frac{1}{M^2} \mathbbm{1}_{ \ensuremath{\boldsymbol\ell}(\beta_1'\alpha_2) \equiv 0(M) }
  \frac{1}{\mathrm{covol}(\Lambda(\alpha_2))} \int_{W(\alpha_2)} \Phi^\infty \biggl( x_1^\infty, \frac{\alpha_2}{X_2} \biggr) \,d x_1^\infty,
\end{equation}
which we shall show equals $\sigma_\infty\prod_p \sigma_p + O(1/X_2^{1 - o(1)}) $.

By a change of variable in the $x_1^\infty $ integral and a closer analysis of the covolume, it
can be checked that
\begin{equation}
  \frac{1}{M^2} \mathbbm{1}_{ \ensuremath{\boldsymbol\ell}(\beta_1'\alpha_2) \equiv 0(M) }
  \frac{1}{\mathrm{covol}(\Lambda(\alpha_2))} \int_{W(\alpha_2)} \Phi^\infty \biggl( x_1^\infty, \frac{\alpha_2}{X_2} \biggr) \,d x_1^\infty
  = X_2^{-2} I_\infty \biggl(\frac{\alpha_2}{X_2}\biggr) \prod_p I_p(\alpha_2),
\end{equation}
where
\begin{equation}
  I_p(\alpha_2) := \int_{\mathcal{O}_{K, p}} \mathbbm{1}_{ \substack{\beta_1 \equiv \beta_1' (M) \\ \alpha_2  \equiv \alpha_2' (M) }}
  \delta(\ensuremath{\boldsymbol\ell}(\beta_1\alpha_2)) \, d \beta_1,
\end{equation}
\begin{equation}
  I_\infty(x_2^\infty) := \int_{K_\infty} \Phi^\infty(x_1^\infty, x_2^\infty) \delta(\ensuremath{\boldsymbol\ell}(x_1^\infty x_2^\infty))
  \, d x_1^\infty
\end{equation}
are the expected local densities.

The desired result follows from Poisson summation in the $\alpha_2 $ sum by expanding out
the Euler product and applying Poisson summation (akin to \S\ref{sec:cuhlxu5ldf}) as follows.

At this point, we have reduced ourselves to showing
\begin{equation}
  \frac{1}{X_2^4}\sum_{\alpha_2 }  I_\infty \biggl(\frac{\alpha_2}{X_2}\biggr)
  \prod_p I_p(\alpha_2) 
  = \sigma_\infty\prod_p \sigma_p + O(X_2^{-1 + o(1)}).
\end{equation}
Expanding the Euler product, 
\begin{equation}
  \prod_p I_p(\alpha_2) = \sum_{q\ge 1} N_{\alpha_2}^*(q),
\end{equation}
with
\begin{equation}
  N^*(\alpha, q) := \sum_{d | q } \mu(d) \frac{d^2}{q^2}\sum_{ \beta (q/d) }
  \mathbbm{1}_{ \ensuremath{\boldsymbol\ell}(\alpha\beta) \equiv 0(q/d) }.
\end{equation}
It can be checked that 
\begin{equation*}
  \tilde N^*(\alpha, q) \ll q^{1 + o(1)} \sum_{d | q } d \mathbbm{1}_{ dq^3 | N(\alpha) }
\end{equation*}
similarly (but in a simpler fashion) to the proofs of Propositions \ref{prop.N1} and \ref{prop:S20}, so it 
follows that
\begin{equation}
  \sum_{|\alpha|_{\sup}\ll X_2 } \sum_{q > Q } N^*(\alpha, q)\ll \frac{X_2^4}{Q^{1 - o(1)}}.
\end{equation}
Therefore, taking $Q = X_2^{1 - \varepsilon} $, we may discard the contribution of $q > Q $ (if $\varepsilon\to 0 $)
for we have
\begin{equation}
  \frac{1}{X_2^4}\sum_{\alpha_2 } \prod_p I_p(\alpha_2) I_\infty \biggl(\frac{\alpha_2}{X_2}\biggr)
  = \frac{1}{X_2^4}\sum_{\alpha_2 } I_\infty \biggl(\frac{\alpha_2}{X_2}\biggr) \sum_{q\le Q } N^*(\alpha, q) +  O(X_2^{-1 + o(1)}).
\end{equation}
It can be checked that $I_\infty $ is smooth (and that it is compactly supported is clear from
the support of $\Phi^\infty $) and that $N^*(-, q) $ is $q $-periodic, so it follows that
\begin{equation}
  \frac{1}{X_2^4}\sum_{\alpha_2 } I_\infty \biggl(\frac{\alpha_2}{X_2}\biggr) \sum_{q\le Q } N^*(\alpha, q)
  = \int_{K_\infty} I_\infty(x_2^\infty) \,d x_2^\infty \sum_{ q\le Q} \frac{1}{q^4}\sum_{ \beta_2 (q) } N^*(\beta_2, q) + O(X_2^{-A}).
\end{equation}
The desired result follows upon noting that
\begin{equation}
  \sigma_\infty = \int_{K_\infty} I_\infty(x_2^\infty) \,d x_2^\infty, 
\end{equation}
\begin{equation}
  N_1^*(q) = \frac{1}{q^4} \sum_{\beta_2 (q) } N^*(\beta_2, q), 
\end{equation}
recalling the bound $N_1^*(q) \ll q^{-2 + o(1)} $ of Proposition \ref{lemma:cran7rw9t5} to reinsert the contribution of
$q > Q $, and recalling that
\begin{equation}
  \sum_{q } N_1^*(q) = \prod_p \sigma_p.
\end{equation}

\section{Proof of Theorem \ref{theorem:cq5oju0wf8}}\label{sec:crbv97ueg1}
In this section, we give a proof of Theorem \ref{theorem:cq5oju0wf8}.

It will actually follow relatively quickly from the more useful statement below,
in whose proof we will carry out the passage from a sum over binary quartic form
to a restricted sum over a quartic number field where we apply Theorem \ref{theorem:cq5o3jikeq}.
\begin{theorem}\label{proposition:crb1e3s2w0}
  There exists $\delta > 0$ such that for $Q, X\ge 1$, $q_0\ge 1$, $a_0\in \mathbb{Z}/q_0 \mathbb{Z}$, convex $\mathcal{R}\subset [0,1]^2$
  with $\partial \mathcal{R} $ bounded and $0\not\in \mathcal{R}$, we have
  \begin{equation}
    \biggl|\sum_{q\le Q \\  q \equiv a_0 (q_0) } \biggl( \sum_{m, n\in X \mathcal{R} \\ q | m^4 + n^4 } 1
    -  \frac{\rho(q)}{q^2} X^2\vol \mathcal{R} \biggr)\biggr|
    \ll q_0^{O(1)} X^{2 - \delta} \biggl(1 + \biggl(\frac{Q}{X^2}\biggr)^{O(1)}\biggr) ,
  \end{equation}
  where $\rho(q) = \#\set{x_1, x_2\in \mathbb{Z}/q \mathbb{Z} : x_1^4 + x_2^4 \equiv 0(q) }$.
\end{theorem}

\begin{proof}[Proof of Theorem \ref{theorem:cq5oju0wf8} assuming Theorem \ref{proposition:crb1e3s2w0}]
  We have that
  \begin{equation}
    \sum_{n_1, n_2\in \mathbb{Z} \\ 0 < n_1^4 + n_2^4 \le N } d(n_1^4 + n_2^4)
    = 4\sum_{n_1, n_2\ge 1 \\ n_1^4 + n_2^4\le N} d(n_1^4 + n_2^4) + O(N^{1/4}).
  \end{equation}
  Opening up the divisor function and summing over $d_1d_2 = n_1^4 + n_2^4$, we obtain that
  the above equals
  \begin{equation}
    2\sum_{d_1\le \sqrt{N}} \sum_{n_1^4 + n_2^4\le N \\ d_1 | n_1^4 + n_2^4} 1
    -
    2\sum_{d_1\le \sqrt{N}} \sum_{n_1^4 + n_2^4\le d_1^2 \\ d_1 | n_1^4 + n_2^4} 1.
  \end{equation}
  
  By Theorem \ref{proposition:crb1e3s2w0}, this equals
  \begin{equation}
    2\kappa N^{1/2} \sum_{d_1 \le \sqrt{N} } \frac{\rho(d_1)}{d_1^2}
     - 2\kappa \sum_{d_1\le \sqrt{N} } \frac{\rho(d) }{d} + O(N^{1/2 - \delta})
   \end{equation}
  
  the desired result follows upon noting that (see, e.g., \cite[\S7]{MR1670278})
  \begin{equation}
    \sum_{d\le D }\frac{\rho(d)}{d^2} = c_{-1}\log D + c_0 + O(D^{-{1/4}}),
  \end{equation}
  and so by partial summation, we have that
  \begin{equation}
    \sum_{d\le D } \frac{\rho(d)}{d} = c_{-1} D + O(D^{3/4}).
  \end{equation}
  Collecting, the desired result follows.
\end{proof}

The remainder of this section is devoted to the proof of Theorem \ref{proposition:crb1e3s2w0}.
We begin with some reductions.
Existing level of distribution results for binary forms, namely
\cite[Lemma 3.3]{MR1670278}, implies that
\begin{align*}
  \biggl|\sum_{q\le Q \\  q \equiv a_0 (q_0) } \biggl( \sum_{m, n\in X \mathcal{R} \\ (m, n)\neq (0, 0) \\ q | m^4 + n^4 } 1
  -  \frac{\rho(q)}{q^2} X^2\vol \mathcal{R} \biggr)\biggr|
  &\le 
    \sum_{q\le Q} \biggl| \sum_{m, n\in X \mathcal{R} \\ (m, n)\neq (0, 0)\\ q | m^4 + n^4 } 1
  -  \frac{\rho(q)}{q^2} X^2\vol \mathcal{R} \biggr|\\ 
  &\ll (X\sqrt{Q}+ Q)\log^{O(1)}X, 
\end{align*}
which is an acceptable $O(X^{2 - 1/200 + o(1)})$ when $Q\le X^{2 - 1/100}$. Note too that we may suppose
that $Q\le X^{2 + 1/100}$, for the contribution of $q\ge X^{2 + 1/100}$ is acceptable from the divisor
bound by virtue of the term $(Q/X^2)^{O(1)}$.

Now, write
\begin{equation}
  R(Q, X; \Psi, W, a_0, q_0) :=
  \sum_{q \equiv a_0 (q_0) } \Psi\biggl(\frac{q}{Q}\biggr) \biggl( S_q(X; W) - M_q(X ; W)\biggr),
\end{equation}
where
\begin{equation}
  S_q = S_q(X; W) :=  \sum_{(m, n)\neq (0, 0)\\ q | m^4 + n^4 }
  W  \biggl( \frac{m}{X}, \frac{n}{X} \biggr),
\end{equation}
\begin{equation}
  M_q = M_q(X; W) := \frac{\rho(q)}{q^2} X^2\int_{\mathbb{R}^2} W( x,y ) \,d x \,d y.
\end{equation}
We now apply a dyadic partition of unity to the sum over $q$.
Fix some $\Psi\in C_c^\infty((3/2, 5/2))$ satisfying
\begin{equation}
  \sum_{k\in \mathbb{Z} } \Psi \biggl(\frac{x}{2^k}\biggr) = 1
\end{equation}
for all $x\in \mathbb{R}_{> 0}$. Then, we are reduced to showing that
\begin{equation}
  R(Q, X; \Psi, \mathbbm{1}_{ \mathcal{R} }, a_0, q_0) \ll q_0^{O(1)}X^{2 - \delta}
\end{equation}
for any scale $X^{2 - 1/100}\ll Q\ll X^{2 + 1/100}$.

Now, fix some nonzero $\phi\in C_c^\infty((0, 1)^2)$, write $\phi_\Omega(x) = \Omega^2\phi(\Omega x)$, and let
$W_{\mathcal{R}} = \mathbbm{1}_{ \mathcal{R}_1 } * \phi_\Omega$, where
\begin{equation}
  \mathcal{R}_1 = \set{x\in \mathcal{R} : \mathrm{dist}(x, \partial \mathcal{R}) > 2/\Omega}.
\end{equation}
Then, by a divisor bound, we have that
\begin{equation}
  R(Q, X; \Psi, \mathbbm{1}_{ \mathcal{R} }, a_0, q_0)
  = R(Q, X; \Psi, W_{\mathcal{R}}, a_0, q_0) + O \biggl( \frac{X^{2 + o(1)}}{\Omega} \biggr),
\end{equation}
so Theorem \ref{proposition:crb1e3s2w0} will actually follow from the following statement.

\begin{proposition}\label{proposition:crb358nxke}
  There exist $\delta > 0$ with the following property. Let $X, \Omega, q_0 \ge1$,
  \begin{equation}\label{eq:6}
    X^{2 - 1/100}\ll Q\ll X^{2 + 1/100},
  \end{equation}
  and let $W\in C_c^{\infty}(\mathbb{R}^2\setminus \set{0})$ satisfy
  \begin{enumerate}
  \item For all $x\in \mathrm{supp}(W)$ we have $ \Omega^{-1}\ll |x|\ll \Omega$.
  \item For all $j_1, j_2\ge 0$, we have 
    \begin{equation}
      \biggl\| \frac{\partial^{j_1 + j_2}}{\partial x^{j_1}\partial y^{j_2}}W  \biggr\|_{\infty}
      \ll_{j_1, j_2} \Omega^{j_1 + j_2}.
    \end{equation}
  \end{enumerate}
  Then, we have that
  \begin{equation}
    R(Q, X ; \Psi, W, a_0, q_0)\ll (q_0\Omega)^{O(1)}X^{2 - \delta}.
  \end{equation}
\end{proposition}

In the remainder of this section, we shall prove Proposition \ref{proposition:crb358nxke}, beginning with
preparation for the application of Theorem \ref{theorem:cq5o3jikeq}.

We start by analyzing the condition $q | m^4 + n^4$. If $(m , n) = 1$, we would have a
correspondence $\set{q | m^4 + n^4} \iff \set{\mathfrak{q} | m + n\zeta}$, but we must deal with the possibility
of primes that aren't totally split dividing both $m$ and $n$.

To this end, note that
\begin{equation}
  S_q(X; W) = \sum_{ d\ge 1} \sum_{(m', n') = 1 \\ q | d^4(m'^4 + n'^4) } W \biggl( \frac{dm'}{X}, \frac{dn'}{X} \biggr)
  = \sum_{d\ge 1 \\ r | d^4, q \\ (r, q/r) = 1 } \sum_{(m', n') = 1 \\ q/r | m'^4 + n'^4 }
  W \biggl( \frac{ dm'}{X}, \frac{dn'}{X} \biggr).
\end{equation}
Letting $q/r = a_1$, we obtain that
\begin{equation}
  \sum_{q \equiv a_0 (q_0)} \Psi \biggl(\frac{q}{Q}\biggr) S_q(X;W)
  = \sum_{d\ge 1 \\ r | d^4 } \sum_{(m', n') = 1 \\ a_1 | m'^4 + n'^4 \\ (a_1, r) = 1}
  \Psi \biggl(\frac{a_1r}{Q}\biggr) W \biggl( \frac{dm'}{X}, \frac{dn'}{X} \biggr).
\end{equation}
We can discard the contribution of $d > D$ (to be specified, but it will be
a small power of $X$ ultimately) at the cost of a remainder $O(X^{2 + o(1)}/D)$, so
\begin{equation}
  \sum_{q \equiv a_0 (q_0)} \Psi \biggl(\frac{q}{Q}\biggr) S_q(X;W)
  = \sum_{d\le D \\ r | d^4, q } \sum_{(m', n') = 1 \\ a_1 | m'^4 + n'^4 \\ (a_1, r) = 1 \\ a_1r \equiv a_0(q_0)}
  \Psi \biggl(\frac{a_1r}{Q}\biggr) W \biggl( \frac{dm'}{X}, \frac{dn'}{X} \biggr)
  + O \biggl(\frac{X^{2 + o(1)}}{D}\biggr)
\end{equation}
This sets up our passage to $K$. There is a bijection 
\begin{align}
  \set{a_1 | m'^4 + n'^4} &\leftrightarrow\set{\mathfrak{a}_1 | (m' + n'\zeta)} 
\end{align}
given by $N\mathfrak{a}_1\leftrightarrow \mathfrak{a}_1 $.
Therefore, we have that
\begin{multline}
  \sum_{q \equiv a_0 (q_0) } \Psi \biggl(\frac{q}{Q}\biggr)S_q \\ 
  = \sum_{d\le D \\ r | d^4 }
  \sum_{\substack{\alpha'\in \mathcal{O}_K \\ \ensuremath{\boldsymbol\ell}(\alpha') = 0 \\ (\langle \alpha', \zeta^3 \rangle, \langle \alpha', \zeta^2 \rangle) = 1}}
  \sum_{\substack{\mathfrak{a}_1 | (\alpha') \\ (N\mathfrak{a}_1, r) = 1 \\ rN\mathfrak{a}_1 \equiv a_0 (q_0)}}
  \Psi \biggl(\frac{rN\mathfrak{a}_1}{Q}\biggr)
  W \biggl( \frac{d \langle \alpha', \zeta^3 \rangle}{X}, \frac{d \langle \alpha', \zeta^2 \rangle}{X} \biggr) + O \biggl(\frac{X^{2 + o(1)}}{D}\biggr)
  \\
  = \sum_{d\le D \\ r | d^4 } \sum_{h\ge 1 } \mu(h)
  \sum_{\substack{\alpha\in \mathcal{O}_K \\ \ensuremath{\boldsymbol\ell}(\alpha) = 0 \\ h | \alpha }}
  \sum_{\substack{\mathfrak{a}_1 | (\alpha) \\ (N\mathfrak{a}_1, r) = 1 \\ r N\mathfrak{a}_1 \equiv a_0 (q_0)}}
  \Psi \biggl(\frac{rN\mathfrak{a}_1}{Q}\biggr)
  W \biggl( \frac{d \langle \alpha, \zeta^3 \rangle}{X}, \frac{d \langle \alpha, \zeta^2 \rangle}{X} \biggr) + O \biggl(\frac{X^{2 + o(1)}}{D}\biggr) \\
  = \sum_{d h \le D \\ r | d^4 } \mu(h)
  \sum_{\substack{\alpha\in \mathcal{O}_K \\ \ensuremath{\boldsymbol\ell}(\alpha) = 0 \\ h | \alpha }}
  \sum_{\substack{\mathfrak{a}_1 | (\alpha) \\ (N\mathfrak{a}_1, r) = 1 \\ r N\mathfrak{a}_1 \equiv a_0 (q_0)}}
  \Psi \biggl(\frac{rN\mathfrak{a}_1}{Q}\biggr)
  W \biggl( \frac{d \langle \alpha, \zeta^3 \rangle}{X}, \frac{d \langle \alpha, \zeta^2 \rangle}{X} \biggr) + O \biggl(\frac{X^{2 + o(1)}}{D}\biggr). \\
\end{multline}
At this point, we wish to move from our sum over ideals to a sum over $\mathcal{O}_K$.
To do this, take $u : K_\infty^\times\to \mathbb{R}$ given by
\begin{equation*}
  u(\alpha) = \log|\alpha|_{v_1} - \log|\alpha|_{v_2}
\end{equation*}
for the two complex places $v_1, v_2$ of $K$ so that $u(\varepsilon_0) \neq 0$ and pick any $\xi\in C_c^\infty(\mathbb{R})$ such
that
\begin{equation}
  \#\mu_K\sum_{k\in \mathbb{Z} } \xi(t -  u(\varepsilon_0^k)) = 
  \#\mu_K\sum_{k\in \mathbb{Z} } \xi(t -  ku(\varepsilon_0)) = 1
\end{equation}
for all $t\in \mathbb{R}$. Taking $\Xi(x) = \xi(u(x))$, we obtain that
\begin{equation}\label{eq:crb4m1u7r4}
  \sum_{\varepsilon\in \mathcal{O}_K^\times } \Xi(\varepsilon x) = 1
\end{equation}
for all $x\in K_\infty^\times$.

Applying this, we obtain that
\begin{equation}\label{sec:crb42lc0ha}
  \sum_{q \equiv a_0 (q_0) } \Psi \biggl(\frac{q}{Q}\biggr) S_q
  = \sum_{d h\le D \\ r | d^4 } \mu(h) \Sigma_{d, h, r}(X; W) + O \biggl(\frac{X^{2 + o(1)}}{D}\biggr),
\end{equation}
where
\begin{equation}
  \Sigma_{d,h,r} (X; W)
  = \sum_{\alpha_1, \alpha_2\in \mathcal{O}_K \\ \ensuremath{\boldsymbol\ell}(\alpha_1\alpha_2) = 0}
  \Phi^\infty \biggl( \frac{\alpha_1}{X_1}, \frac{\alpha_2}{ X_2 } \biggr) \Phi^f_{h, r}(\alpha_1, \alpha_2),
\end{equation}
where $\Phi^\infty \in C_c^{\infty}(K_\infty^2)$, $\Phi^f_{h, r} : (\mathcal{O}_K/M \mathcal{O}_K)^2\to \mathbb{C}$ are given by
\begin{equation}
  \Phi^\infty(x_1^\infty, x_2^\infty) = \Xi(x_1^\infty) \Psi(N(x_1^\infty)) W(\langle x_1^\infty x_2^\infty, \zeta^3 \rangle, \langle x_1^\infty x_2^\infty, \zeta^2 \rangle)
  \phi(|\ensuremath{\boldsymbol\ell}(x_1^\infty x_2^\infty)|^2),
\end{equation}
\begin{equation}
  \Phi^f_{h, r}(\beta_1, \beta_2) = \mathbbm{1}_{ \substack{N(\beta_1)r \equiv a_0(q_0) \\ (\beta_1, r) = 1\\ h | \beta_1\beta_2 }}, 
\end{equation}
where we write $X_1 = X_1(r) = (Q/r)^{1/4}$, $X_2(d, r) = (X/d)/X_1$, $M = M(h, r) = hq_0r$, and
$\phi\in C_c^\infty(\mathbb{R})$ is any fixed bump function with $\phi(0) = 1$.

It can be checked that the conditions of Theorem \ref{theorem:cq5o3jikeq} hold with $M, \Omega, X_1, X_2, \ensuremath{\boldsymbol\ell}$ in the
theorem as we have them here, $X$ taken to be $X/d$. Therefore, we obtain that 
\begin{equation}
  \Sigma_{d, h , r} = \frac{1}{d^2}X^2\sigma_\infty \prod_p \sigma_p(h, r)
  + O \biggl( (\Omega D)^{O(1)}X^{2 - \eta + o(1)}\biggr),
\end{equation}
where
\begin{align}
  \sigma_{\infty} &= \int_{K_\infty^2} \Phi^\infty(x_1^\infty, x_2^\infty) \delta(\ensuremath{\boldsymbol\ell}(x_1^\infty x_2^\infty))
                    \,d x_1^\infty \,d x_2^\infty, \\ 
  \sigma_p(h, r) &= \int_{\mathcal{O}_{K, p}^2} \Phi^f_{h, r}(x_1^p, x_2^p)
                   \delta(\ensuremath{\boldsymbol\ell}(x_1^px_2^p)) \,d x_1^p \,d x_2^p.
\end{align}
We obtain that
\begin{equation}
  \sum_{q } \Psi \biggl(\frac{q}{Q}\biggr) S_q(X ; W) =
  X^2 \sigma_\infty \sum_{d h \le D \\ r | d^4} \frac{\mu(h)}{d^2}\prod_p \sigma_p(h, r)
  + O \biggl( X^{2 + o(1)}\biggl( \frac{1}{D}
  + \frac{(\Omega D)^{O(1)}}{X^{\eta}} \biggr) \biggr).
\end{equation}
It remains to evaluate the sum of products of local factors, and show that it
lines up with
\begin{equation}
  \sum_{q } \Psi \biggl(\frac{q}{Q}\biggr) M_q(X ; W).
\end{equation}
We capture all our local computation in the following lemma, from which
Theorem \ref{proposition:crb1e3s2w0} follows upon taking $D$ a sufficiently small power of $X$.
\begin{lemma}
  We have that
  \begin{equation}
    \sigma_\infty \sum_{ d h \le D }\frac{\mu(h)}{d^2} \sum_{r | d^4 } \prod_p\sigma_p(h, r)
    = \biggl( \int_{\mathbb{R}^2} W(x, y) \,d x \,d y \biggr)\sum_{q \equiv a_0(q_0) }
    \Psi \biggl(\frac{q}{Q}\biggr) \frac{\rho(q)}{q^2}
    + O \biggl(\frac{1}{D^{1 - o(1)}}\biggr).
  \end{equation}
\end{lemma}
\begin{proof}

  First, note that with $Y = Q^{20}$, we have 
  \begin{multline}
    \biggl( \int_{\mathbb{R}^2} W(x, y) \,d x \,d y \biggr)\sum_{q \equiv a_0(q_0) }
    \Psi \biggl(\frac{q}{Q}\biggr) \frac{\rho(q)}{q^2} \\ 
    = \frac{1}{Y^2} \sum_{q } \Psi \biggl(\frac{q}{Q}\biggr)
    \sum_{m, n \\ m^4 + n^4 \equiv 0(q) } W \biggl( \frac{m}{Y}, \frac{n}{Y} \biggr)
    + O \biggl(\frac{1}{Y^2}\biggr) \\ 
    = \frac{1}{Y^2}\sum_{q } \Psi \biggl(\frac{q}{Q}\biggr) S_q(Y; W) + O \biggl(\frac{1}{Y^2}\biggr).
  \end{multline}
  We may now perform the exact same maneuvers leading up to \eqref{sec:crb42lc0ha} (with the same
  choice of $D$) and apply Theorem \ref{theorem:cq5o3jikeq} to yield the desired result. The error term
  follows from the rate of convergence of the RHS.
\end{proof}

\newcommand{\etalchar}[1]{$^{#1}$}
\def\cprime{$'$} \def\cprime{$'$} \def\cprime{$'$} \def\cprime{$'$}

\end{document}